\documentclass[12pt]{amsart}

\usepackage{tasks}
\usepackage{hyperref}
\usepackage{blindtext}
\usepackage{setspace}
\usepackage[toc,page]{appendix}

\usepackage{graphicx}
 
\usepackage{subcaption}
\usepackage[utf8]{inputenc}
\usepackage[export]{adjustbox}
\usepackage{wrapfig}

\usepackage{epsfig,epstopdf}
\usepackage{marginnote}
\usepackage{datetime}
\usepackage{color}
\usepackage{amsrefs}
\usepackage{amsmath}
\usepackage{mathtools}
\usepackage{graphicx}  
\usepackage[top=3.5cm, bottom=2.5cm, outer=2.5cm, inner=2.5cm]{geometry}
\usepackage{hyperref}
\hypersetup{linkcolor=red,
citecolor=black
}
\usepackage{amsmath}

\hypersetup{
	colorlinks = true,
	linkcolor = {magenta},
	citecolor = {blue}
}

\newtheorem{thm}{Theorem}[section]
\newtheorem{lem}[thm]{Lemma}
\newtheorem{prop}[thm]{Proposition}
\newtheorem{cor}[thm]{Corollary}

\theoremstyle{definition}

\theoremstyle{remark}
\newtheorem{re}[thm]{Remark}

\newcommand{\RNum}[1]{\uppercase\expandafter{\romannumeral #1\relax}}

\numberwithin{equation}{section}

\usepackage[utf8]{inputenc}
\usepackage[english]{babel}
\usepackage{paralist}
\usepackage{amsthm}
\usepackage{marginnote}

\newenvironment{Ali}{\color{blue}}{\color{black}}
\newcommand{\BBB}{\begin{Ali}}
\newcommand{\FFF}{\end{Ali}}
\newenvironment{tat}{\color{red}}{\color{black}}
\newcommand{\RRR}{\begin{tat}}
\newcommand{\GGG}{\end{tat}}

\theoremstyle{definition}
\newtheorem{definition}{Definition}[section]
 
\theoremstyle{remark}

 \usepackage{bm,amsmath}

\newcommand{\RN}[1]{%
  \textup{\uppercase\expandafter{\romannumeral#1}}%
}

\newcommand{\bfu}{\mathbf{u}}
\newcommand{\bfv}{\mathbf{v}}

\newcommand{\bfV}{\mathbf{V}}

\newcommand{\bfx}{\mathbf{x}}
\def\bfF{\mathbf{F}}

\allowdisplaybreaks[4]

\usepackage[utf8]{inputenc}

\hypersetup{
	colorlinks = true,
	linkcolor = {magenta},
	citecolor = {blue}
}

\title{Continuous data assimilation for two-phase flow: analysis and simulations}

\author{Yat Tin Chow}
\address{Department of Mathematics, University of California, Riverside, CA, USA}
\email{yattinc@ucr.edu}

\author{Wing Tat Leung}
\address{Department of Mathematics, University of California, Irvine, CA, USA}
\email{wtleung@uci.edu}

\author{Ali Pakzad}
\address{Department of Mathematics, Indiana University Bloomington, IN, USA}
\email{apakzad@iu.edu}

 \date{\today}

\begin{document}
\begin{abstract}

We propose, analyze, and test a novel continuous data assimilation two-phase flow algorithm for reservoir simulation. We show that the solutions of the algorithm, constructed using coarse mesh observations, converge at an exponential rate in time to the corresponding exact reference solution of the two-phase model. More precisely, we obtain a stability estimate which illustrates an exponential decay of the residual error between the reference and approximate solution, until the error hits a threshold depending on the order of data resolution. Numerical computations are included to demonstrate the effectiveness of this approach, as well as variants with data on sub-domains. In particular, we demonstrate numerically that  synchronization is achieved for data collected from a small fraction of the domain.


\end{abstract}
\maketitle

\section{Introduction}
\subsection{Data assimilation}
 Data assimilation (DA) refers to a class of methodologies which combines information from grain coarse observational data with simulation/dynamical model in order to obtain a more accurate forecast.  The method has a long history, with applications in weather modeling and environmental forecasting  \cite{K03}, as well as the medical, environmental and biological science,  \cites{KKMMMP11,  MKKNPMM13}, imaging, traffic control, finance and oil exploration \cite{ABN16}. There are a variety of data assimilation techniques, with which actual measured quantities over time are incorporated in system models.   One classical technique,  which is based on a linear-quadratic estimation, is known as the Kalman filter.  This Bayesian approach gives exact probabilistic predictions, although the underlying system and any corresponding observation models are assumed to be linear.  This approach has been modified to cover more general cases in ensemble Kalman filter, extended Kalman filter and the unscented Kalman filter;  see  \cites{Asch-Data2016, Evensen-Data2009,Law-AMathematical2015}. 
 
  A major difficulty of applying a physical model to real life applications is that, usually the initial condition cannot be known/measured exactly, and only an approximation over a coarse spatial resolution is known.    This imprecision in measuring the initial condition may sometimes cause an exponential growing error in the solution of a nonlinear system.  To overcome this difficulty, a promising approach, known as the continuous data assimilation, was proposed and analyzed by Azouani, Olson, and Titi in  \cites{AOT14-2, AOT14} based on techniques coming from control theory.  This approach introduces a feedback control term at the
PDE level to synchronize the computed solution with the true solution corresponding to the observed data.  To describe the method, we consider a  dynamical system 
\begin{align}\label{origODE} 
\frac{d \bfu}{dt} =\bfF(\bfu)\, ,
\end{align}
with insufficient/inaccurate knowledge of the initial state $\bfu(0)$, but with a solution $\bfu(t)$ on a coarse grid that is believed to accurately reflect some aspects of the underlying physical reality. 
Given observational data of the system at a coarse spatial resolution of size $H$, i.e. $\Pi_H(\bfu(t))$ from some given interpolation operator $\Pi_H$, the algorithm is to construct an approximate solution $\bfv(t)$ from the observations that satisfies the auxiliary  equation
\begin{align}\label{nudgedODE} 
\frac{d \bfv}{dt} =\bfF(\bfv) -\mu \, \Pi_H\, (\bfv-\bfu)\;, \quad \bfv(0)=\text{arbitrary}\;,
\end{align} 
where $\mu >0$ is a relaxation (nudging) parameter.   The goal is to pick $\mu >0$ and $H>0$ such that
$$ \bfv(t)\rightarrow \bfu(t) $$
as $t \rightarrow \infty$ in a suitable spatial space. The above algorithm is designed to work for many nonlinear dissipative dynamical systems of the form \eqref{origODE}, with their solutions well-known to be unstable.  Owing to this instability, it is expected that any small error in the initial data could lead to an exponentially growing error in the solutions.  In these cases, the dissipative term (only) controls the small scales and instabilities occur at large scales.  The feedback term in \eqref{nudgedODE} that is newly introduced then aims to stabilize the system and damp the error term at large scales by forcing  (nudging) the large spatial scales of the measured solution (of the auxiliary equation) back to the reference solution.
  
 In the context of the incompressible 2D Navier-Stokes equations,  the  authors  in  \cite{AOT14}  proved that, for large enough $\mu$ and small enough $H$, the approximate solution  $\bfv$ to  \eqref{nudgedODE}, converges exponentially fast to the exact solution $\bfu$.  Numerical experiments were carried out successfully to test this algorithm for many nonlinear systems, for instance, the 2D Navier-Stokes equation \cites{Gesho-Acomputational2016,  LRZ19, Garcia-Archilla-Uniform2020, JP21, BBM21},  the Rayleigh-B\'{e}nard equations \cites{Farhat-Assimilation2018, Altaf-Downscaling2017},  and  the Kuramoto-Sivashinsky equations \cites{LT17}.  Continuous data assimilation applied to other physical phenomena and PDEs includes   non-Newtonian fluids \cite{CGJP21}, magnetohydrodynamic   equations \cite{Biswas-Continuous2018},   Leray-$\alpha$ model of turbulence \cite{Jolly-Adata2017},    quasi-geostrophic equation \cite{Jolly-Adata2017},   Darcy’s equation \cite{Markowich-Continuous2019},  KdV equations  \cite{Jolly-Determining2017},     primitive equations \cite{Pei-Continuous2019}, and many others.  While the aforementioned works assume noise-free observations,  the method is later extended to the case when only noisy data can be obtained, e.g. in \cites{Bessaih-Continuous2015, Jolly-Continuous2019, Foias-ADiscrete2016}. The authors refer the readers to other recent literature on this topic; see, e.g. \cites{Auroux-ANudging2008,Biswas-Data2021, Farhat-Data2020, Farhat-Continuous2015, Farhat-Continuous2017,Farhat-Data2016,Farhat-Abridged2016, FJT, FLT,Foias-ADiscrete2016, Hudson-Numerical2019,Ibdah-Fully2020, LT17, Mondaini-Uniform2018, Pawar-Long2020,ZRSI19}.

 \subsection{Multi-phase  flow}
 Thanks to the development of numerical reservoir simulation, oilfield corporations have deeply benefited from the technology in terms of confidently predicting oil
recovery estimates and determining the selection of operations during
the deployment of specific recovery technologies. 
 During the screening stage, numerical reservoir model is established and simulations are run to determine the feasibility of many injection/production options. This requires reservoir simulation to be accurate and time efficient. Furthermore, for oil and gas reservoirs that are already in production, it is necessary to determine reservoir parameters (e.g. permeability, porosity) with increasing certainty \cites{oliver2008inverse,aanonsen2009ensemble}. As the oil field matures, more effective practices such as production enhancement, chemical treatment or infill drilling can greatly extend the life of the oil reservoir, thereby increasing the overall recovery rate. 

A commonly used model in reservoir simulation is the multi-phase flow model \cites{kueper1991two,whitaker1986flow,blunt1992simulation,trangenstein1989mathematical}.  In recent decades, both the mathematical analysis and numerical simulation of the two-phase flows have been a focus of study for many researchers and practitioners, thanks to their important applications in petroleum engineering and hydrology.  The system of equations governing two-phase immiscible incompressible flows in porous media consists of a nonlinear elliptic Darcy-type equation for the global pressure and a nonlinear parabolic equation with degenerate diffusion term for the saturation, which are coupled by means of the total velocity, recuperated from Darcy’s equation \cites{chavent1986mathematical,antontsev1989boundary,chen1997single}.

 Due to the nonlinear nature of the problem, the velocity of the fluid in different phases highly depends on the saturation and the pressure of the respective phase.  So as to obtain an accurate simulation, we are required to have a good initialization of the model parameters as well as an accurate initial condition of the saturation.  While uncertainty quantification (UQ) and parameter estimation have been used to predict reservoir parameters \cites{osypov2013model,park2013history,jiang2017multiscale}, 
 unfortunately it is not feasible to obtain accurate microscope data of saturation and pressure at a particular time slice.  
 Nonetheless, a coarse scale approximation of the saturation and pressure field can be obtained using seismic waves data and well logs data.
 In this work, we consider a simple two-phase model \eqref{eq:pressure_eq2} and \eqref{eq:S_eq2}, and inject these coarse-scale data directly into a system via our proposed data assimilation algorithm
 to control the error of the solution without using any microscope initial condition.



 \subsection{Main result of this paper}
 
 While two-phase  models have a long history of success on certain problems, they tend to lose accuracy
on more complicated problems due to the insufficient and inaccurate knowledge of the initial state. Meanwhile, continuous data assimilation (DA) has recently been used to improve accuracy  by incorporating measurement data into the simulation.   In this work, we introduce a data assimilation model \eqref{eq: DA2phase} which combines the coarse grid saturation measurement data with the multi-phase flow problem. For illustrative purpose, instead of a general multi-phase flow model, we  only focus on an immiscible incompressible two-phase flow model  \eqref{eq:pressure_eq2} and \eqref{eq:S_eq2}. After  performing  an error analysis for the data assimilation algorithm,   we prove an exponentially decaying error bound between the exact and  approximate  solutions until the error reaches a certain level (Theorem \ref{main_theorem}). More precisely, for a given  data resolution $H$,  we find   $t_0$ in terms of $H$ such that synchronization is
guaranteed for all $t < t_0$ for large enough $\mu \sim \mathcal{O}(1/H)$. In this case,  coarser data resolution leads to a smaller $t_0$ and vise versa. 

In addition, we illustrate the efficiency of the algorithm by extensive computational studies. We find  that the nudging algorithm achieves synchronization with
data that is much more coarser than required by the rigorous estimates in Theorem \ref{main_theorem}.  Furthermore, we demonstrate numerically that observation on a small fraction of the domain suffices for global assimilation, which may in practice indeed be the case for data collected on the smaller portion of the whole domain.

\subsection*{Organization of this paper}
In section \ref{sec2}, we briefly introduce basic notations and preliminaries used in the analysis. Section \ref{sec3}  provides background on the  two-phase  model and revisits the existence and uniqueness argument  of the  model.  Later, in
section \ref{sec4}, after introducing the data assimilation algorithm,  we state and prove our main results. We give conditions under which
the approximate solution, obtained by the data assimilation algorithm, converges to the solution
of the  two-phase  model.  In section \ref{sec5}, we  present numerical results to demonstrate the performance of our proposed data assimilation model.

\section{Notation and Preliminaries} \label{sec2}

 Let $\Omega \subset \mathbb{R}^d$ be a bounded open domain.  From what follows, we  always write $C$ as a generic positive constant independent of the  model parameters.   Let $p \in [1 , \infty]$,  and the
Lebesgue space $L^p(\Omega)$ is the space of all measurable functions $\bfv$ on $\Omega$ with which
$$\|\bfv\|_{L^p} :=\big( \int_{\Omega} |\bfv (\bfx)| ^p\, d\bfx\big)^{\frac{1}{p}} \, < \infty,  \hspace{1cm} \text{if} \hspace{0.2cm} p \in [1 , \infty),$$
$$\|\bfv\|_{L^{\infty}} := {\text{ess-sup}}_{\bfx \in \Omega}  |\bfv (\bfx)| < \infty,  \hspace{1cm} \text{if} \hspace{0.2cm} p  = \infty. $$
The $ L^2$  norm and inner product will be denoted by $\|\cdot\|$ and $( \cdot ,  \cdot)$  respectively, while all other norms will be labeled with subscripts.  Let $\bfV$ be a Banach space of functions defined on $\Omega$ with the associated norm $\| \cdot \|_\bfV$. We denote by $L^p(a, b; \bfV)$,  $p \in [1 , \infty]$,  the space of functions $\bfv : (a, b) \rightarrow  \bfV$ such that
$$\|\bfv\|_{L^p (a , b ; \bfV)} :=\big( \int_a^b \|\bfv(t)\|_\bfV ^p\, d\bfx\big)^{\frac{1}{p}} \, < \infty,  \hspace{1cm} \text{if} \hspace{0.2cm} p \in [1 , \infty),$$
$$\|\bfv\|_{L^{\infty} (a , b ; \bfV)} := \text{ess-sup}_{t  \in (a , b)} \|\bfv( t )\|_\bfV < \infty,  \hspace{1cm} \text{if} \hspace{0.2cm} p  = \infty. $$
From now on, for notational sake, we will denote an integral of a function $f$ over a domain $\Omega$ with one of the following three notations if no confusion will arise
\[
\int_\Omega f(x) dx \,, \, \quad \int_\Omega f dx \, ,   \quad \text{ and } \quad \int_\Omega f \,.
\]
In addition, we consider
$$V_{0}=\{v\in H^{1}(\Omega)|\;\int_{\Omega}v=0\}\, , $$ 
$$V_{0}^{*}=\{v\in H^{-1}(\Omega)|\int_{\Omega}v=0\}.$$

\begin{definition}\label{bilinear}
The  bi-linear operator
$a (\cdot,\cdot)$ and semi-norm $|\cdot|_{K}$ are given as 
\[
a(u,v)=\int_{\Omega}K \, \,  \nabla u\cdot\nabla v,\;\quad \forall\,  \, \, u,v\in H^{1}(\Omega)\, , 
\]
and 
\[
|u|_{K}^{2}=a(u,u), 
\]
 where $K$ is a tensor with $K_{ij}\in L^{\infty}(\Omega)$.
\end{definition}
\begin{re}
Note  that $|\cdot|_{K}$ defines a norm for $V_{0}$.
\end{re}

\begin{definition}\label{Greenfun}
 The Green's operator $G:V_{0}^{*}\rightarrow V_{0}$ is given  by 
\[
a(G(u),v)=(u,v).
\]
We define the norms of $V_{0}$ and $V_{0}^{*}$ respectively as $\|v\|_{V_{0}}=|v|_{K}$
and $\|u\|_{V_{0}^{*}}:=\|G(u)\|_{V_{0}}$ for all $u\in V_{0}^{*},v\in V_{0}$.
\end{definition}
Our data assimilation method requires that the observational measurements
$\Pi_{H}(\bfu)$,  be given as linear interpolant observables,  satisfying
$\Pi_{H}:L^{2}(\Omega)\rightarrow L^{2}(\Omega)$ such that 
\begin{equation}
\begin{split}\|\Pi_{H}\varphi\| & \leq c_{I}\|\varphi\|,\hspace{1cm}\forall\,\varphi\in L^{2}(\Omega),\\
\|\varphi-\Pi_{H}\,\varphi\| & \leq c_{0}\,H\|\varphi\|_{H^{1}(\Omega)},\hspace{1cm}\forall\,\varphi\in H^{1}(\Omega).
\end{split}
\label{I_h}
\end{equation}
An example of such interpolation operators can be given by a projection operator onto the
Fourier modes with wave numbers $|k|\le1/H$. Other physical examples
include the volume elements and constant finite element interpolation \cite{JonesTiti}.

\section{A simple two-phase  flow Model}  \label{sec3}

In this section, we  describe an immiscible incompressible two-phase
flow model. In this model, the flow of the fluids is governed by the
Darcy's law of the water and the oil phases. More precisely,  the velocity $\tilde{v}_{\alpha}$
of phase $\alpha$   is described as 
\[
\tilde{v}_{\alpha}=-\cfrac{k_{r\alpha}}{\mu_{\alpha}}\, K\, \nabla P_{\alpha},\;\text{in }(0,T)\times\Omega,\; \hspace{0.5cm}\alpha=o \,  (\text{oil}),w \, (\text{water}),
\]
where $k_{r\alpha},\mu_{\alpha},P_{\alpha}$,  and $\rho_{\alpha}$
are the relative permeability, viscosity, pressure,  and density of
phase $\alpha$, $K$ is the absolute permeability, $T$ is the final
time and $\Omega$ is the reservoir domain. We consider the capillary
pressure $P_{cow}$ defined as 
\[
P_{cow}:=P_{o}-P_{w}.
\]
The saturation of phase $\alpha$, denoted by   $S_{\alpha}$, is governed by the mass balance equation
\[
\cfrac{\partial(\phi\rho_{\alpha}S_{\alpha})}{\partial t}+\nabla \cdot(\rho_{\alpha}\tilde{v}_{\alpha})=\rho_{\alpha}\tilde{q}_{\alpha},\;\text{in }(0,T)\times\Omega,\; \hspace{0.5cm}\alpha=o ,w,
\]
where $\tilde{q}_{\alpha}$ is the volumetric input of phase $\alpha$
and $\phi$ is the porosity of the medium. The saturation $S_{\alpha}$ then 
satisfies
\begin{equation}
S_{o}+S_{w}=1.\label{eq:sum_sat}
\end{equation}

To simplify the model, in our work, we only consider the case when the
density $\rho_{\alpha}$, the viscosity $\mu_{\alpha}$,  and the porosity $\phi$ are constant functions,
and that the relative permeability $k_{r\alpha}$, the capillary pressure
$P_{cow}$ are functions only depending on $S_{\alpha}$. Using \eqref{eq:sum_sat},
we have $S_{o}=1-S_{w}$ and therefore $k_{r\alpha}$ and $P_{cow}$
can be written as a function of $S_{w}$.  We then define the function
$\kappa_{\alpha}$ as 
\[
\kappa_{\alpha}(S_{w})=\cfrac{k_{r\alpha}(S_{w})}{\phi\mu_{\alpha}}.
\]
To simplify the notation, we will omit the sub-index $w$ and denote $S_w$  as $S$  in the following discussion,. We furthermore assume $\kappa_{\alpha}\in L^{\infty}(0,1)$, $K$
is a $L^{\infty}$ positive tensor  and
\begin{align*}
\infty>\overline{\kappa}\geq(\kappa_{w}(S_{w})+\kappa_{o}(S_{w})) & \geq\underline{\kappa}>0,
\end{align*}
\begin{align*}
\infty>\overline{K}\geq\xi^{T}K\xi & \geq\underline{K}>0\;\quad \forall\xi\in\mathbb{R}^{d},\|\xi\|=1, 
\end{align*}
for some constants $\underline{\kappa}$, $\overline{\kappa}$, $\underline{K}$, 
and $\overline{K}$. Since $\phi$, $\rho_{\alpha}$ are constants,
the mass balance equations can be simplified as 
\begin{equation}
\cfrac{\partial (S_{\alpha})}{\partial t}+\nabla \cdot (v_{\alpha})=q_{\alpha},\;\text{in }(0,T)\times\Omega,\;\alpha=o,w, \label{eq:mass_balance}
\end{equation}
where $q_{\alpha}=\cfrac{\tilde{q}_{\alpha}}{\phi}$ and 
\[
v_{\alpha}=-\kappa_{\alpha}\, K \, \nabla P_{\alpha} ,\;\text{in }(0,T)\times\Omega,\;\alpha=o,w.
\]
We  then obtain a pressure equation by summing the equations \eqref{eq:mass_balance}
\begin{equation}
-\nabla\cdot\Big(K(\kappa_{o}(\nabla P_{o})+\kappa_{w}(\nabla P_{w}))\Big)=\nabla \cdot (v_{o}+v_{w})=q_{o}+q_{w}.\label{eq: pressure_eq0}
\end{equation}
With this, the system can be simplified to 
\begin{align}
-\nabla\cdot((\kappa_{w}+\kappa_{o})K\nabla P_{w}) & =q_{t} + \nabla\cdot(\kappa_{o}\nabla P_{cow}), \label{eq: pressure_eq1}\\
\cfrac{\partial (S_{w})}{\partial t} - \nabla\cdot(\kappa_{w}K\nabla P_{w}) & =q_{w}, \label{eq:S_eq1}
\end{align}
where $q_{t}=q_{o}+q_{w}$.

\subsection*{Global pressure}
We  recall how to reformulate the equations with
the help of the concept of global pressure. We assume $\beta = -\cfrac{\partial P_{cow}}{\partial S_{w}}>0$. The global pressure  $P$
is defined as 
\[
P(S,t,x)=P_{o}(t,x)+\int_{0}^{S}f_{w}\beta(\xi)d\xi, 
\]
where  $f_{i}=\cfrac{\kappa_{i}}{\kappa}$
and $\kappa=\sum_{i=o,w}\kappa_{i}$  and $S= S(t,x)$ is a function of $(t,x)$. In case there will not cause any confusion, we would, by an abuse of notation, simplify $p(S(t,x),t,x)$ as $p(t,x)$. We may then check directly that
\[
\nabla\left(\int_{0}^{S}f_{w}\beta(\xi)d\xi\right)=f_{w}\beta\nabla S=-f_{w}\nabla P_{cow}.
\]
Hence,  we have  $\kappa\nabla P=\kappa\nabla P_{o}-\kappa_{w}\nabla P_{cow}=\kappa_{o}\nabla P_{o}+\kappa_{w}\nabla P_{w}$.

Now define a function $\theta$ as
\begin{equation}\label{DefTheta}
    \theta(S):=\int_{0}^{S}\cfrac{\kappa_{w}(\xi)\kappa_{o}(\xi)}{\kappa(\xi)}\beta(\xi)d\xi, 
\end{equation}
which is monotone since $\cfrac{\kappa_{w}(\xi)\kappa_{o}(\xi)}{\kappa(\xi)}\beta(\xi)>0$ and obtain 
\[
\kappa(S)\, \nabla\theta(S)=\kappa_{w}(S)\, \kappa_{o}(S)\, \beta (S)\, \nabla S=-\kappa_{w}(S)\, \kappa_{o}(S)\, \nabla P_{cow}(S) \,.
\]
We therefore have
\begin{align*}
\nabla\theta+\kappa_{w}\nabla P= & -\cfrac{\kappa_{o}\kappa_{w}}{\kappa}\nabla P_{cow}+\cfrac{\kappa_{o}\kappa_{w}}{\kappa}\nabla P_{o}+\cfrac{\kappa_{w}^{2}}{\kappa}\nabla P_{w}\\
= & \kappa_{w}\nabla P_{w}=-K^{-1}v_{w}
\end{align*}
and 
\begin{align*}
-\nabla\theta+\kappa_{o}\nabla P= & \cfrac{\kappa_{o}\kappa_{w}}{\kappa}\nabla P_{cow}+\cfrac{\kappa_{o}^{2}}{\kappa}\nabla P_{o}+\cfrac{\kappa_{o}\kappa_{w}}{\kappa}\nabla P_{w}\\
= & \kappa_{o}\nabla P_{o}=-K^{-1}v_{o}.
\end{align*}
This leads to the equation
\[
-\nabla(\kappa(S)K(x)\nabla P(t,x))=q_{t}(x), \; \quad \forall x\in\Omega, 
\]
where $q_{t}=q_{o}+q_{w}$. We then define an inverse map $\mathbb{T}$ of
$\theta$, such that 
\[
\mathbb{T}(\theta(S))=S \,.
\]
We notice that $\mathbb{T}$ is well-defined and $\theta$ is irreversible since $\theta$ is monotone. With the help of the aforementioned notations,
the system can be now rewritten as: find the solution pair $(p,S)$ satisfying 

\begin{align}
-\nabla(\kappa K\nabla P) & =q_{t},\label{eq:pressure_eq2}\\
\partial_{t}(S)-\nabla\cdot\Big(K\nabla\theta(S)+\kappa_{w}K\nabla P\Big) & =q_{w}\,.\label{eq:S_eq2}
\end{align}

\subsection{Analysis of the two-phase  model; Existence and Uniqueness results}  

In this subsection, we  first revisit some classical results of the
two-phase  model,  and then  review  the existence and uniqueness results  of the weak solution of the problems \eqref{eq:pressure_eq2} and \eqref{eq:S_eq2}
as in \cites{alt1985nonsteady,chen2001degenerate,cances2012existence,amaziane2013existence,arbogast1992existence}.  The remaining analysis is proceed with the following common assumptions \cites{chen2001degenerate,yeh2006holder}.

\begin{enumerate}
\item[A1.] $\Omega \subset \mathbb{R}^{d}$ for $d \in \{2,3\}$ is a connected Lipschitz domain.
\item[A2.] $\kappa_{\alpha}$ are continuous and $\kappa_{w}(0)=\kappa_{o}(0)=0$
\begin{align*}
\kappa_{o}(S_{o}) & >0\;\text{if }S_{o}>0\,,\\
\kappa_{w}(S_{w}) & >0\;\text{if }S_{w}>0\,.
\end{align*}
\item[A3.] $q_{w},q_{o} \in L^{\infty}(0,T;H^{-1}(\Omega))$. 
\item[A4.] $\kappa_{w}$ and $\kappa_{o}$ satisfy 
\begin{align*}
\kappa_{w}(S) & =C_{w}S^{1+\xi_{w}},\\
\kappa_{o}(S) & =C_{o}(1-S)^{1+\xi_{o}},
\end{align*}
for some $\xi_{o}, \, \xi_{w}, \, C_{w}, \, C_{o}>0$. 
\item[A5.] Denote $b(S)=\cfrac{\kappa_{w}(S)\kappa_{o}(S)}{\kappa(S)}\, \beta(S)$, 
and then  $b(S)\leq C_{0}$ for all $0\leq S\leq1$.
\item[A6.] $\beta(S)$ satisfies the bound
\[
c_{\alpha}(S)^{-\beta_{w}}(1-S)^{-\beta_{o}}\leq\beta(S)\leq C_{\alpha}(S)^{-\beta_{w}}(1-S)^{-\beta_{o}}\,, 
\]
for some $\beta_{o}, \, \beta_{w}, \, C_{\alpha}, \, c_{\alpha}>0$. 
\item[A7.] $\kappa$ satisfies the inequality
\[
|\kappa(S_2)-\kappa(S_1)| \leq C(\theta(S_2)-\theta(S_1),S_2-S_1)^{\frac{1}{2}},  \hspace{0.5cm} \text{ for  }S_1,S_2\in [0,1].
\]

\item[A8.] $p \in L^{\infty}(0,T; W^{1,\infty}(\Omega))$.
\end{enumerate}
\begin{re}
 Assumption $A8$ can also  be derived directly  from some  mild assumptions given in \cites{chen2001degenerate,yeh2006holder}. 
\end{re}
\begin{prop}
With assumptions $A4-A6$, there are positive constants $\delta<1/2,$ $\tau_{i}>0$ and $C_{i}>0$ such that 
\begin{align*}
C_{1}S^{\tau_{w}} & \leq b(S)\leq C_{2}S^{\tau_{w}},\; \hspace{1.7cm}S\in[0,\delta],\\
C_{3} & \leq b(S)\leq C_{4},\; \hspace{2.3cm}S\in[\delta,1-\delta],\\
C_{5}(1-S)^{\tau_{o}} & \leq b(S)\leq C_{6}(1-S)^{\tau_{o}},\; \hspace{0.8cm} S\in[1-\delta,1], 
\end{align*}
where $\tau_{w}=1+\xi_{w}-\beta_{w}$,  and $\tau_{o}=1+\xi_{o}-\beta_{o}$. 
\end{prop}

With the aforementioned assumptions, we  arrive at the following lemma, which will be important for our subsequent analysis.
\begin{lem} 
\label{lem:3.3}
 Assuming  $A4-A6$, for any $S_{1},S_{2}\in[0,1],$  then  we have 
\[
\tilde{C}(S_{2}-S_{1})^{2+\tau}\leq(\theta(S_{2})-\theta(S_{1}))(S_{2}-S_{1})\leq C_{0}(S_{2}-S_{1})^{2}
\]
where $\tau=\max(\tau_{w},\tau_{o})$. 
\end{lem}

\begin{proof}
Without loss of generality,    assume $S_{2}\geq S_{1}$, then by the definition of $\theta$ in \eqref{DefTheta}, we have 
\begin{align*}
(\theta(S_{2})-\theta_{1}(S_{1})) & =\int_{S_{1}}^{S_{2}}b(S)dS\leq C_{0}\int_{S_{1}}^{S_{2}}dS\,,\\
(\theta(S_{2})-\theta_{1}(S_{1})) & \geq c\int_{S_{1}}^{S_{2}}(1-S)^{\tau_{o}}S^{\tau_{w}}\,, 
\end{align*}
and therefore 
\[
\tilde{C}(S_{2}-S_{1})^{1+\tau}\leq|\theta(S_{2})-\theta(S_{1})|\leq C_{0}(S_{2}-S_{1})\,.
\]
\end{proof}

We now state the well-posedness result for the model  \eqref{eq:pressure_eq2} and \eqref{eq:S_eq2}, equipped with Neumann boundary conditions for  $P$ and $S$ and initial condition for $S$,  although   the results can be extended to more  general  boundary conditions.  The complete proof  can be found in the above mentioned references. Since some of the proofs are related to the techniques we used in our main theorem, we will put the proof of some lemmas in the appendix. Readers may refer to \cite{chen2001degenerate} for further details.

\begin{thm}[\textbf{Existence and uniqueness of weak solutions of the two-phase  flow model}]\label{Thm2Phase}

\cite{chen2001degenerate} Let  $\, \Omega \subset \mathbb{R}^{d}\,$ for $d \in \{2,3\}$ be a connected Lipschitz domain. Assume 
 $ S(0,\cdot) \in L^2(\Omega)$, $S$  and $P$ with homogeneous  Neumann boundary conditions in \eqref{eq:pressure_eq2} and \eqref{eq:S_eq2}.    Problems \eqref{eq:pressure_eq2} and \eqref{eq:S_eq2}  have a unique weak solution satisfying for
all $T > 0$
 \[
\partial_{t}S_{w}\in L^{2}(0,T;H^{-1}(\Omega)). 
\]
Moreover
\begin{equation}\label{Weak2phase}
   \begin{split}
       \int_{\Omega}\kappa(S_{w})K \, \nabla P \cdot\nabla w & =\int_{\Omega}q_{t}w, \;\quad \hspace{0.4cm} \forall w\in H^{1}(\Omega)\,,\\
\int_{0}^{T}\int_{\Omega}\Big(\partial_{t}S_{w}v+K(\nabla\theta(S_{w})+\kappa_{w}\nabla P)\cdot\nabla v\Big) & =\int_{0}^{T}\int_{\Omega}q_{w}v, \;\quad \forall v\in L^{2}(0,T;H^{1}(\Omega)).
   \end{split} 
\end{equation}
\end{thm}

The proof of existence starts with defining a discrete time solution of the problem.  Denote $$\partial^{\eta}v(t)=\cfrac{v(t+\eta)-v(t)}{\eta},$$
for any function $v(t)$, where    $\eta=T/N$. 
 Next define 
 
 \begin{align*}
 I_{i,\eta}(V)  = \Big\{ v\in L^{\infty}(0,T;V):  v\; & \text{is piece-wise   polynomial with degree } \\
  &  \text{$i$ in time on each sub-interval}  \, \, J_{i}\subset J \Big\},   
\end{align*}
where $J_{i}=(t_{i},t_{i+1}]$, $t_{i}=i\eta$ and $t_{N}=T$. We
then define the discrete time solution $P^{\eta}\in I_{0,\eta}(H^{1}(\Omega))$
with $$\int_{\Omega}P^{\eta}(t,\cdot)=0,  \hspace{0.5cm}  \text{and} \hspace{0.5cm} \theta^{\eta}\in I_{1,\eta}(H^{1}(\Omega)),$$
which satisfy  
\begin{equation}\label{DisSol}
\begin{split}
    \int_{0}^{T}\int_{\Omega}\kappa(S_{w})K\nabla P^{\eta}\cdot\nabla w & =\int_{0}^{T}\int_{\Omega}q_{t}w\;\quad \forall w\in I_{0,h}(H^{1}(\Omega))\,,\\
\int_{0}^{T}\int_{\Omega}\Big(\partial^{\eta}(T(\theta^{\eta}))v+\kappa_{w}K\nabla P^{\eta}\cdot\nabla v+K\nabla\theta^{\eta}\cdot\nabla v\Big) & =\int_{0}^{T}\int_{\Omega}q_{w}v\;\quad \forall v\in I_{0,h}(H^{1}(\Omega))\,.
\end{split}
\end{equation}

 We now need  the following  result, which are stated  without proof.
\begin{lem}
\cite{chen2001degenerate} The discrete problems   \eqref{DisSol} are  well-posed . 
\end{lem}

\begin{lem}\label{lem:lem3}
\cite{chen2001degenerate} Let $d:\mathbb{R}\rightarrow\mathbb{R}$ be an increasing
function such that $d(0)=0$ and $\{c_{i}\}$ be a sequence of real
numbers. Then for any number $m>0$, 
\[
\sum_{k=1}^{m}(d(c_{k})-d(c_{k-1}))c_{k}\geq D(c_{m})-D(c_{0})\geq-D(c_{0})
\]
where 
\[
D(c)=\int_{0}^{c}(d(c)-d(\xi))d\xi.
\]
\end{lem}
Next, we recall the following Lemma, which is important to our subsequent analysis.
\begin{lem} \label{lem:3.7}\cite{chen2001degenerate}
The discrete solution $P^{\eta}$, $\theta^{\eta}$ satisfy 
\[
\|P^{\eta}\|_{L^{\infty}(0,T;H^{1}(\Omega))}+\|\theta^{\eta}\|_{L^{2}(0,T;H^{1}(\Omega))}\leq C \, , 
\]
where $C$ is independent of $\eta$. 
\end{lem}

From  Lemma \ref{lem:lem3}, and since $P^{\eta}$ and $\theta^{\eta}$ remain bounded, we  have the following corollary. 
\begin{cor} \cite{chen2001degenerate}
Let  $P^{\eta}$ and $\theta^{\eta}$ be solutions to \eqref{DisSol}, then
\begin{enumerate}
\item For any $2\leq r<\infty$,  there exists a subsequence $P^{\eta}\rightharpoonup p$
weakly in $L^{r}(0,T;H^{1}(\Omega))$,  and $\theta^{\eta}\rightharpoonup\theta$
weakly in $L^{2}(0,T;H^{1}(\Omega))$.
\item There is a subsequence $\theta^{\eta}\rightarrow\theta$ strongly
in $L^{2}(0,T;L^{2}(\Omega))$. 
\item  There is a subsequence $\theta^{\eta}\rightarrow\theta$ strongly
in $L^{2}(0,T;H^{1-\alpha})$ for any $0< \alpha <1/2$,  and $S^{\eta}\rightarrow S$
pointwise a.e. on $(0,T]\times\Omega$ where $S^{\eta}=T(\theta^{\eta})$.
\end{enumerate}
 \end{cor}
In the rest of this section,   we  analyze the stability of the weak solution by bounding the difference between the two solutions $(S_{i},P_{i})$, $i=1,2$ as follows.
\begin{lem}
\label{lem:eq1_bound} \cite{chen2001degenerate} Let $(S_{i},P_{i})$, $i=1,2$ be two weak solutions to  \eqref{Weak2phase} given by Theorem \ref{Thm2Phase} with source term $q_{t,i}$ and $q_{w,i}$ with respectively. Then we have 
\[
\|\nabla(P_{2}-P_{1})\|_{L^{2}}\leq C\left(\|\kappa(S_{2})-\kappa(S_{1})\|_{L^{2}}+\|q_{t,2}-q_{t,1}\|_{L^{2}}\right).
\]
\end{lem}

With $\pi$ to be  the average operator $\pi (u)=\cfrac{1}{|\Omega|}\int_{\Omega}u$, and   $e=(I-\pi)(S_{2}-S_{1})$, where $s_1$ and $s_2$ are solutions with two different sources, one can prove the following stability result. 

\begin{lem}
\label{lem:eq2_bound}
\cite{chen2001degenerate}
Let $(S_{i},P_{i})$, $i=1,2$ be two weak solutions to  \eqref{Weak2phase} given by Theorem \ref{Thm2Phase} with source term $q_{t,i}$ and $q_{w,i}$ with respectively. With the assumptions $A1-A8$, we have 
\begin{align*}
 & \cfrac{1}{2}\partial_{t}\int_{0}^{t}\|e\|_{V_{0}^{*}}^{2}+\int_{0}^{t}(\theta_{2}-\theta_{1},S_{2}-S_{1})\\
\leq & \cfrac{\delta_{p}^{p}}{p}\int_{0}^{t}\|S_{1}-S_{2}\|_{L^{p}}^{p}+\cfrac{C_{q}}{q\delta_{p}^{q}}\Big(\|\pi(S_{2}-S_{1})(0,\cdot)\|_{L^{q}}^{q}+\|q_{w,2}-q_{w,1}\|_{L^{q}(0,t;L^{q})}^{q}\Big)\\
 & +C\cfrac{\delta}{2}\int_{0}^{t}\Big(\|\kappa(S_{2})-\kappa(S_{1})\|_{L^{2}}^{2}+\|\nabla(P_{2}-P_{1})\|_{L^{2}}^{2}+\|q_{w,2}-q_{w,1}\|_{L^{2}}^{2}\Big)+C\cfrac{1}{2\delta}\int_{0}^{t}\|\nabla G(e)\|_{L^{2}}^{2}
\end{align*}
where $\delta$, $\delta_{p}$ are arbitrary positive constants, $p$ is arbitrary positive constant larger than $1$ and $\cfrac{1}{p}+\cfrac{1}{q}=1$.   
\end{lem}

We will give the stability estimate and thus the uniqueness with the following theorem.
\begin{thm} \label{thm:3.11}\cite{chen2001degenerate}
Let $(S_{i},P_{i})$, $i=1,2$ are two weak solutions to  \eqref{Weak2phase} given by Theorem \ref{Thm2Phase}. With the assumptions $A1-A8$, we have 
\begin{align*}
\|(I-\pi) & (S_{2}-S_{1})\|_{L^{\infty}(0,T;V_0^{*})}^{2}\leq  Ce^{Ct}\Big(\|\pi(S_{2}-S_{1})(0,\cdot)\|_{L^{q_{0}}(\Omega)}^{q_{0}}+\|q_{w,2}-q_{w,1}\|_{L^{q_{0}}(0,T;L^{q_{0} }(\Omega))}^{q_{0}}\\
 & +\|(I-\pi)(S_{2}-S_{1})(0,\cdot)\|_{V_0^{*}}^{2}+\|q_{w,2}-q_{w,1}\|_{L^{2}(0,T;L^{2}(\Omega))}^{2}+\|q_{t,2}-q_{t,1}\|_{L^{2}(0,T;L^{2}(\Omega))}^{2}\Big)\, , 
\end{align*}
for some $C>0$, $\tau$ defined as in Lemma \ref{lem:3.3}, and $q_{0}=\cfrac{2+\tau}{1+\tau}$.
\end{thm}

\section{Data assimilation algorithm  for the  two-phase flow problem}\label{sec4}
In this section, we first describe the  nudging algorithm for  the two-phase  flow equations \eqref{eq:pressure_eq2} and \eqref{eq:S_eq2}. We can consider we can obtain some observed data of the saturation. We consider the data collecting operator is denoted as $\Pi_{H}^{*}(S)$. We remarked that the image of the data collecting operator is usually in a finite dimension space. The data assimilation algorithm for two-phase  flow problem  is defined
as 
\begin{equation}\label{eq: DA2phase}
\begin{split}
-\nabla\cdot(\kappa K\nabla P) & =q_{t},\\
\cfrac{\partial (\tilde{S})}{\partial t}+\nabla\cdot\left(\cfrac{\kappa_{w}}{\kappa}K\nabla P\right)+\mu \, (\Pi_{H}^{*}(\tilde{S})-\Pi_{H}^{*}(S)) & =q_{w}, 
\end{split}  
\end{equation}
where  $\Pi_{H}:H^{1}(\Omega)\rightarrow L^{2}(\Omega)$ is a linear interpolant operator satisfying 
\begin{equation}
  \begin{split}
  \|\Pi_H(S)-S\|& \leq CH\|S\|_{H^{1}}, \\
\|\Pi_H(S) \| &\leq C \|S\|, 
  \end{split}  
\end{equation}
which can naturally be extended to $L^2(\Omega)$. The dual operator is given by $\Pi_{H}^{*}:L^{2}(\Omega)\rightarrow H^{-1}(\Omega)$
\[
\int_{\Omega}\Pi_{H}(S)\, v=\int_{\Omega}S\, \Pi_{H}^{*}(v), \;\quad \forall v\in L^{2}(\Omega).
\]
For instance,  $\Pi_{H}$ may be considered to be the $L^{2}$ projection
operator to the piece-wise constant finite element space (for more details see \cite{JonesTiti}), namely, 
\[
\Pi_H(S)(x)=\int_{K_{i}}S, \;\quad \forall x\in K_{i}, 
\]
where $K_{i}$ is a coarse element in $K_{i}$, and in this case, we have \begin{equation}
    \Pi_{H}^{*}=\Pi_{H}.
\end{equation}
 In this case, the domain $\Omega$ is partitioned into a coarse partition $\mathcal{T}_{H}$, where the observed data  are collected. More precisely, we collect the data of the averaged water saturation  $S$ in each coarse element containing $\{x_{i}\}_{i}^{M}\subset\Omega$, where $x_i$ are the points that physical measurements are performed where $M$ is the number of measurements. For example, we can consider the $\{x_{i}\}_{i}^{M}\subset\Omega$ to be all of the center point of the coarse grid element.
\begin{re}
We remark that there are many choices of $\Pi_H$, for example, we can consider $\Pi_H$ as a standard polynomial interpolation operator .
\end{re}

\begin{definition}[\textbf{Weak solution to the data assimilation algorithm}]

Let  $(S,P)$ be the solution to the  two-phase  problem from Theorem \ref{Thm2Phase}. The continuous data assimilation equations \eqref{eq: DA2phase} has a unique  weak solutions $(\tilde{S},\tilde{P})$ that satisfies for all $T >0 $
\[
\partial_{t}\tilde{S}\in L^{2}(0,T;H^{-1}(\Omega))
\]
and
\begin{equation}\label{WeakDA}
\begin{split}
    \int_{\Omega}\kappa(\tilde{S})K\nabla\tilde{P}\,\cdot \nabla w & =\int_{\Omega}q_{t}w,  \quad \forall w\in H^{1}(\Omega)\,,\\
\int_{0}^{T}\int_{\Omega}\Big(\partial_{t}\tilde{S}\, v+K(\nabla\tilde{\theta}(\tilde{S})+\kappa_{w}(\tilde{S})\nabla\tilde{P})\cdot\nabla v & +  \mu \, \Pi_{H}^{*}(\tilde{S}-S)(v)\Big)\\
& =\int_{0}^{T}\int_{\Omega}q_{w}v, \; \quad \forall v\in L^{2}(0,T;H^{1}(\Omega)). 
\end{split} 
\end{equation}
\end{definition}

\begin{re}
Although the  well-posedness  of  \eqref{WeakDA} is not the focus of this work, we speculate that it can be proved by a usual compactness argument \cite{chen2001degenerate}, or by the fixed-point argument which has recently been suggested in \cite{CGJP21} for a similar problem. 
\end{re}

In the next theorem, we  analyze the residual error coming from the data assimilation algorithm with unknown initial condition, which is the main analytical result of this work.

\begin{thm} \label{main_theorem}  Let  $\Omega \subset \mathbb{R}^{d}$ for $d \in \{2,3\}$ be a connected Lipschitz domain.    Consider $S$ be a  solution  of the two-phase  equations with initial data $ S(0) \in L^2(\Omega)$, ensured by  Theorem \ref{Thm2Phase}, and $\Pi_H: L^2(\Omega) \rightarrow L^2 (\Omega)$  be a linear map satisfying \eqref{I_h}.  Let $\tilde{S}$ be a solution to the data assimilation algorithm given by \eqref{eq: DA2phase} with homogeneous Neumann boundary conditions.  Then for all $H >0$, if 
\begin{equation}
\mu :=  \mu(\tilde{\gamma}, \tilde{C}, C_2 ,H)=2\tilde{\gamma}\Big(\cfrac{C_{2}}{\bar{C}}\Big)^{-\frac{1}{2}}H^{-1}=O(H^{-1})\,,
\end{equation}
the following bound holds for $t<t_{0}$
\begin{equation}\label{MainBound}
    \|(I-\pi)(\tilde{S}-S)\|_{V_0^{*}}^{2}\leq e^{-\frac{\mu}{2}t}\left(\|e(0,\cdot)\|_{V_0^{*}}^{2}+2C_{5}\|\pi(\tilde{S}-S)(0,\cdot)\|_{L^{q_0}}^{q_0}\right), 
\end{equation}
where   $t_0:= t_{0}(\tilde{\gamma},\tilde{C}, C_2,H)$  is given as
\begin{equation}\label{T_zero}
\begin{split}
    t_{0}  = \max\Big\{ 0\leq \zeta <T,  \hspace{0.1cm}\text{s.t. }    \hspace{0.1cm}  c^{*}\, H^{2+\tau}  \leq \, 
      \| \tilde{S}(t)-S(t)\|_{L^{2+\tau}}^{2\tau+\tau^2},\;\forall \, 0\leq t\leq \zeta \Big\}. 
    \end{split}
\end{equation}
Herein, $ \tilde{C}, \,  C_2, \, C_5 $ are constants, defined in the proof of this theorem, depending only on $\Omega$, and  $ c^* := c^* (\tilde{\gamma}, \tilde{C}, C_2)$ is a  constant appearing in \eqref{cStar}, whereas $\tilde{\gamma}=O(1)$ with respect to $H$. 

\end{thm}

\begin{re}
The above theorem showed that if we choose  $\mu=O(H^{-1})$, the error of the solution will decay exponentially until the error is reduced to a certain level.
\end{re}

\begin{proof}

Subtracting \eqref{Weak2phase} and \eqref{WeakDA}, the difference  satisfies the following error equations
\begin{equation}
    \int_{\Omega}\kappa(\tilde{S})K\nabla(\tilde{P}-P) \cdot \nabla w=\int_{\Omega}\Big(\kappa(S)-\kappa(\tilde{S})\Big)K\nabla P\cdot\nabla w, \;\quad \forall w\in H^{1}(\Omega), 
\end{equation}

\begin{equation}\label{eq:thm12_p}
  \begin{split}
      \int_{0}^{T}\int_{\Omega}\Big(\partial_{t}(\tilde{S}-S)v & +K\nabla(\tilde{\theta}-\theta)\cdot\nabla v  +\kappa_{w}(\tilde{S})K\nabla(\tilde{P}-P)\cdot\nabla v+\mu(\tilde{S}-S)\Pi_{H}(v)\Big)\\
= & \int_{0}^{T}\int_{\Omega}\Big(\kappa_{w}(S)-\kappa_{w}(\tilde{S})\Big)K\nabla P\cdot\nabla v, \;\quad \forall v\in L^{2}(0,T;H^{1}(\Omega)).
  \end{split}  
\end{equation}
Denote $e=(I-\pi)(\tilde{S}-S)$ and set $v=G(e)$ in \eqref{eq:thm12_p}. With that, we obtain 
\begin{align}
 & \int_{0}^{T}\int_{\Omega}\Big(\partial_{t}(\tilde{S}-S)G(e)+K\nabla(\tilde{\theta}-\theta)\cdot\nabla G(e)+\mu(\tilde{S}-S)\Pi_{H}(G(e))\Big)\nonumber \\
& =  \int_{0}^{T}\int_{\Omega}\Big(\Big(\kappa_{w}(S)-\kappa_{w}(\tilde{S})\Big)K\nabla P\cdot\nabla G(e)-\kappa_{w}(\tilde{S})K\nabla(\tilde{P}-P)\cdot\nabla G(e)\Big).\label{eq:thm12_p-1}
\end{align}
With a similar argument in the proof of Lemma \ref{lem:eq2_bound}, we obtain 
\begin{align}
\int_{\Omega}\partial_{t}(\tilde{S}-S)G(e) & =\cfrac{1}{2}\partial_{t}\|G(e)\|_{V}^{2}=\cfrac{1}{2}\partial_{t}\|e\|_{V^{*}}^{2}, \label{eq:dt_bound}
\end{align}
and 
\begin{align}
\int_{\Omega}K\nabla(\tilde{\theta}-\theta)\cdot\nabla G(e) & =(\tilde{\theta}-\theta,\tilde{S}-S)-(\tilde{\theta}-\theta,\pi(\tilde{S}-S)).\label{eq:theta_bound}
\end{align}
Then choose a test function $v=p(\pi(\tilde{S}-S))^{p-1}$ in (\ref{eq:thm12_p})
to get
\begin{align*}
\|\pi(\tilde{S}-S)(t,\cdot)\|_{L^{p}(\Omega)}^{p}+p\mu\int_{0}^{t}\|\pi(\tilde{S}-S)\|_{L^{p}(\Omega)}^{p} & =\|\pi(\tilde{S}-S)(0,\cdot)\|_{L^{p}(\Omega)}^{p}, 
\end{align*}
and 
\[
\|\pi(\tilde{S}-S)(t,\cdot)\|_{L^{p}(\Omega)}^{p}=e^{-p\mu t}\|\pi(\tilde{S}-S)(0,\cdot)\|_{L^{p}(\Omega)}^{p}, 
\]
for any $p>1$. For any $u\in V^{*}$, we have

\begin{align*}
\int_{\Omega}K\nabla G(u)\cdot\nabla G(u)  =\int_{\Omega}G(u)(u)  & =\int_{\Omega}(I-\Pi_{H})G(u)(u)+\int_{\Omega}\Pi_{H}G(u)(u)\\
 & \leq\bar{C}^{\frac{1}{2}}H\|K^{\frac{1}{2}}\nabla G(u)\|_{L^{2}}\|u\|_{L^{2}}+\int_{\Omega}\Pi_{H}G(u)(u) \,.
\end{align*}
We hence obtain
\begin{equation}
\|u\|_{V_0^{*}}^{2}\leq\bar{C}H^{2}\|u\|_{L^{2}}^{2}+2\int_{\Omega}\Pi_{H}G(u)(u).\label{eq:Pi_bound}
\end{equation}
Combining (\ref{eq:dt_bound}), (\ref{eq:theta_bound}) and (\ref{eq:Pi_bound})
with (\ref{eq:thm12_p-1}), we have 
\begin{align*}
 & \int_{0}^{t}\left(\cfrac{1}{2}\partial_{t}\|e\|_{V_0^{*}}^{2}+(\tilde{\theta}-\theta,\tilde{S}-S)+\cfrac{\mu}{2}\|K^{\frac{1}{2}}\nabla G(e)\|_{L^{2}}^{2}\right) \\
 \leq  &  \int_{0}^{t}\left(\cfrac{1}{2}\partial_{t}\|e\|_{V_0^{*}}^{2}+(\tilde{\theta}-\theta,\tilde{S}-S)+\mu\int_{\Omega}\Pi_{H}G(e)(e)+\cfrac{\mu\bar{C}}{2}H^{2}\|e\|_{L^{2}}^{2}\right)\\
 = &   \int_{0}^{T}\int_{\Omega}\Big(\Big(\kappa_{w}(S)-\kappa_{w}(\tilde{S})\Big)K\nabla P\cdot\nabla G(e)-\kappa_{w}(\tilde{S})K\nabla(\tilde{P}-P)\cdot\nabla G(e)\\
 & \quad +(\tilde{\theta}-\theta,\pi(\tilde{S}-S))+\cfrac{\mu\bar{C}}{2}H^{2}\|e\|_{L^{2}}^{2}\Big)\\
\leq &   C_{1}\int_{0}^{t}\Big(\Big(\|S-\tilde{S}\|_{L^{p}}+\|K^{\frac{1}{p}}\nabla(\tilde{P}-P)\|_{L^{p}(\Omega)}\Big)\|K^{\frac{1}{q}}\nabla G(e)\|_{L^{q}} \\
& \quad +\|\tilde{\theta}-\theta\|_{L^{\tilde{q}}}\|\pi(\tilde{S}-S)\|_{L^{\tilde{P}}}+\cfrac{\mu\bar{C}}{2}H^{2}\|e\|_{L^{2}}^{2}\Big), 
\end{align*}
for any $p>1$ with $\cfrac{1}{p}+\cfrac{1}{q}=1$ and an arbitrary $\mu >0$ which will be determined later. With an argument similar to the
proof of Lemma \ref{lem:eq1_bound}, we can also obtain 
\[
\|K^{\frac{1}{p}}\nabla(\tilde{P}-P)\|_{L^{p}}^{p}\leq\hat{C}_{p}\|\tilde{S}-S\|_{L^{p}}^{p}. 
\]
hence, with  $p=2=q$,  we have

\begin{equation}
\begin{split}
& \int_{0}^{t}\left(\cfrac{1}{2}\partial_{t}\|e\|_{V_0^{*}}^{2}+(\tilde{\theta}-\theta,\tilde{S}-S)+\cfrac{\mu}{2}\|K^{\frac{1}{2}}\nabla G(e)\|_{L^{2}}^{2}\right)  \\
& \leq  \, C_{1}\int_{0}^{t}\Bigg(\big(\|S-\tilde{S}\|_{L^{2}}+\|K^{\frac{1}{2}}\nabla(\tilde{P}-P)\|_{L^{2}(\Omega)}\big)\|K^{\frac{1}{2}}\nabla G(e)\|_{L^{2}}\\
& \quad \quad +\|\tilde{\theta}-\theta\|_{L^{2+\tau}}\|\pi(\tilde{S}-S)\|_{L^{q_0}}+\cfrac{\mu\bar{C}}{2}H^{2}\|e\|_{L^{2}}^{2}\Bigg), 
    \end{split}
    \end{equation}
where $ q_0 = \frac{2+\tau}{1+\tau} $, and 
\[
\int_{0}^{t}\Big(\|S-\tilde{S}\|_{L^{2}}+\|K^{\frac{1}{2}}\nabla(\tilde{P}-P)\|_{L^{2}(\Omega)}\Big)\|K^{\frac{1}{2}}\nabla G(e)\|_{L^{2}}\leq\cfrac{\mu}{4C_{1}}\|K^{\frac{1}{2}}\nabla G(e)\|_{L^{2}}^{2}+\cfrac{2C_{2}}{\mu}\|S-\tilde{S}\|_{L^{2}}^{2}, 
\]
where $C_{2}:=4C_{1}(1+\hat{C}_{2})$. We next estimate the term $\|\tilde{\theta}-\theta\|_{L^{2+\tau}}\|\pi(\tilde{S}-S)\|_{L^{q_0}}$
as 
\begin{align*}
\|\tilde{\theta}-\theta\|_{L^{2+\tau}}\|\pi(\tilde{S}-S)\|_{L^{q_0}} & \leq\cfrac{\tilde{C}}{2C_{1}C_{0}}\|\tilde{\theta}-\theta\|_{L^{2+\tau}}^{2+\tau}+\cfrac{2C_{1}C_{0}(1+\tau)}{\tilde{C}(2+\tau)^{2}}\|\pi(\tilde{S}-S)\|_{L^{q_0}}^{q_0}\\
 & \leq\cfrac{\tilde{C}}{2C_{1}}\|\tilde{S}-S\|_{L^{2+\tau}}^{2+\tau}+\cfrac{C_{3}}{C_{1}}\|\pi(\tilde{S}-S)\|_{L^{q_0}}^{q_0}, 
\end{align*}
where $C_{3}=\cfrac{2C_{1}^{2}C_{0}(1+\tau)}{\tilde{C}(2+\tau)^{2}}.$  Now, from $\int_{0}^{t}(\tilde{\theta}-\theta,\tilde{S}-S)\geq\tilde{C}\int_{0}^{t}\int_{\Omega}|\tilde{S}-S|^{2+\tau}$,
we get that 
\begin{align*}
 & \int_{0}^{t}\Big(\cfrac{1}{2}\partial_{t}\|e\|_{V_0^{*}}^{2}+\cfrac{\tilde{C}}{2}\int_{0}^{t}\int_{\Omega}|\tilde{S}-S|^{2+\tau}+\cfrac{\mu}{4}\|K^{\frac{1}{2}}\nabla G(e)\|_{L^{2}}^{2}\Big)\\
 & \leq  \int_{0}^{t}\Big(\cfrac{C_{2}}{\mu}\|S-\tilde{S}\|_{L^{2}}^{2}+C_{3}\|\pi(\tilde{S}-S)\|_{L^{q_0}}^{q_0}+\cfrac{\mu\bar{C}}{2}H^{2}\|e\|_{L^{2}}^{2},
\end{align*}
Then with the help of  $\|e\|_{L^{2}}^{2}\leq\|\tilde{S}-S\|_{L^{2}}^{2}$ and $\|\pi(\tilde{S}-S)\|_{L^{q_0}}^{q_0}\leq e^{-2(q_0)t}\|\pi(\tilde{S}-S)(0,\cdot)\|_{L^{q_0}}^{q_0}$,
we obtain
\begin{align*}
 & \int_{0}^{t}\Big(\cfrac{1}{2}\partial_{t}\|e\|_{V_0^{*}}^{2}+\cfrac{\tilde{C}}{2}\int_{0}^{t}\int_{\Omega}|\tilde{S}-S|^{2+\tau}+\cfrac{\mu}{4}\|K^{\frac{1}{2}}\nabla G(e)\|_{L^{2}}^{2}\Big)\\
\leq & \Big(\cfrac{2C_{2}}{\mu}+\cfrac{\mu\bar{C}}{2}H^{2}\Big)\|S-\tilde{S}\|_{L^{2}}^{2}+C_{3}\int_{0}^{t}e^{-2(q_0)t}\|\pi(\tilde{S}-S)(0,\cdot)\|_{L^{q_0}}^{q_0}.
\end{align*}
We can now take $\mu$ to be the form $\mu=2\tilde{\gamma}\Big(\cfrac{C_{2}}{\bar{C}}\Big)^{-\frac{1}{2}}H^{-1}$ (where $\tilde{\gamma}$ is a chosen constant) to get
\[
\Big(\cfrac{2C_{2}}{\mu}+\cfrac{\mu\bar{C}}{2}H^{2}\Big)\|S-\tilde{S}\|_{L^{2}}^{2}=H\Big(\cfrac{C_{2}}{\bar{C}}\Big)^{\frac{1}{2}}(\cfrac{1}{\tilde{\gamma}}+\tilde{\gamma}).
\]
With $\int_{0}^{t}\int_{\Omega}|\tilde{S}-S|^{2}\leq\tilde{C}_{\tau}\int_{0}^{t}\Big(\int_{\Omega}|\tilde{S}-S|^{2+\tau}\Big)^{\frac{2}{2+\tau}}$,
we obtain 
\begin{align*}
 & \int_{0}^{t}\Big(\cfrac{1}{2}\partial_{t}\|e\|_{V_0^{*}}^{2}+\cfrac{\tilde{C}}{2}\int_{0}^{t}\int_{\Omega}|\tilde{S}-S|^{2+\tau}+\cfrac{\mu}{4}\|K^{\frac{1}{2}}\nabla G(e)\|_{L^{2}}^{2}\Big)\\
\leq & H\Big(\cfrac{C_{2}}{\bar{C}}\Big)^{\frac{1}{2}}(\cfrac{1}{\tilde{\gamma}}+\tilde{\gamma})\int_{0}^{t}\Big(\int_{\Omega}|\tilde{S}-S|^{2+\tau}\Big)^{\frac{2}{2+\tau}}+C_{3}\int_{0}^{t}e^{-2(q_0)t}\|\pi(\tilde{S}-S)(0,\cdot)\|_{L^{q_0}}^{q_0}.
\end{align*}

Now consider 
$$I:= \left\{0<t<T,  \quad  \text{such that} \quad \;c^*\, H^{2+\tau}\leq\Big(\int_{\Omega}|\tilde{S}(t,\cdot)-s(t,\cdot)|^{2+\tau}\Big)^{\tau}\right\},$$
where 
\begin{equation}\label{cStar}
    c^{*} = \Big(\Big(\cfrac{C_{2}}{\bar{C}}\Big)^{\frac{1}{2}}(\cfrac{1}{\tilde{\gamma}}+\tilde{\gamma})\Big)^{2+\tau}\Big(\cfrac{\tilde{C}}{4}\Big)^{-2-\tau}. 
\end{equation}
Now for all $t\in I$, we have $$\Big(\cfrac{C_{2}}{\bar{C}}\Big)^{\frac{2+\tau}{2}}(\cfrac{1}{\tilde{\gamma}}+\tilde{\gamma})^{2+\tau}\Big(\cfrac{\tilde{C}}{4}\Big)^{-2-\tau}H^{2+\tau}\leq\Big(\int_{\Omega}|\tilde{S}-S|^{2+\tau}\Big)^{\tau}, $$
and 
\begin{align*}
& H\Big(\cfrac{C_{2}}{\bar{C}}\Big)^{\frac{1}{2}}(\cfrac{1}{\tilde{\gamma}}+\tilde{\gamma})\int_{0}^{t}\Big(\int_{\Omega}|\tilde{S}-S|^{2+\tau}\Big)^{\frac{2}{2+\tau}}\\
& =\Big(\Big(\cfrac{C_{2}}{\bar{C}}\Big)^{\frac{2+\tau}{2}}(\cfrac{1}{\tilde{\gamma}}+\tilde{\gamma})^{2+\tau}H^{2+\tau}\Big(\int_{\Omega}|\tilde{S}-S|^{2+\tau}\Big)^{-\tau}\Big)\Big)^{\frac{1}{2+\tau}}\int_{0}^{t}\Big(\int_{\Omega}|\tilde{S}-S|^{2+\tau}\Big)\\
 & \leq \, \cfrac{\tilde{C}}{4}\int_{0}^{t}\int_{\Omega}|\tilde{S}-S|^{2+\tau}.
\end{align*}
Hence, we have 
\begin{align*}
  \int_{0}^{t}\Big(\cfrac{1}{2}\partial_{t}\|e\|_{V_0^{*}}^{2}+\cfrac{\tilde{C}}{4}\int_{0}^{t}\int_{\Omega}|\tilde{S}-S|^{2+\tau}+\cfrac{\mu}{4}\|e\|_{V_0^{*}}^{2}\Big)
& \leq  C_{3}\, \int_{0}^{t}e^{-2(q_0)t}\|\pi(\tilde{S}-S)(0,\cdot)\|_{L^{q_0}}^{q_0}\\
& \leq  C_{5} \, \|\pi(\tilde{S}-S)(0,\cdot)\|_{L^{q_0}}^{q_0}, 
\end{align*}
which proves the theorem as
\[
\|e(t,\cdot)\|_{V_0^{*}}^{2}\leq e^{-\frac{\mu}{2}t}\Big(\|e(0,\cdot)\|_{V_0^{*}}^{2}+2C_{5}\|\pi(\tilde{S}-S)(0,\cdot)\|_{L^{q_0}}^{q_0}\Big).
\]
\end{proof}
\begin{re}
We remark that our method can be likewise extended to a general multi-phase model with no obstruction; however, for the sake of simplicity, we will postpone that to a future work.
\end{re}

\section{Computational study} \label{sec5}

We now present results of two numerical tests that illustrate the theory given in the last section.  We consider Algorithm \eqref{eq: DA2phase} for two scenarios with two different permeability profiles.   In both tests,  the domain is $\Omega = (0,100)^2$.   In all of our numerical tests, we consider Neumann boundary conditions for $P$ and $S$. The reference and the approximate solutions are calculated in a fine square mesh with fine mesh size $h=\sqrt{2}/100$ with time step size $dt = 0.05$. For the numerical method, we follow the method in \cites{huber1999multiphase}. The pressure equation is solved by standard mixed finite element method with $RT_0$ finite element space and the saturation equation is solved by explicit upwinding finite volume method. 

The data is obtained in a coarse square mesh with coarse mesh size $H=\sqrt{2}/10$, while the nudge parameter $\mu = 200$.

Since we do not have access to  true (reference) solutions for these  problems, we instead use a computed solution. The reference solutions were  evolved from a zero initial value, and is run to $t = 325$ using the above setting. For the DA computation, we start from zero initial conditions   use the same spatial and temporal discretization parameters as the reference solution, and start assimilation with the reference solution at $t = 25$  (i.e., time $0$ for DA corresponds to $t = 25$ for the reference).

\subsection{Computational study \RN{1}}\label{Comp-Ex1}
In our first experiment,
we consider the relative permeability $k_{r\alpha}$ defined as
\[
k_{ro}=(1-S_{w})^{2},\hspace{0.5cm}k_{rw}=(S_{w})^{2}\, , 
\]
with the injection and production located at the top-left and  bottom-right corner of the domain respectively.  We consider the source terms $q_t,q_w$ are defined as
\begin{equation*}
q_t(t,x) = \begin{cases}
\cfrac{10}{h^{2}} & x\in \text{ the top-left corner finite grid finite element}\\
-\cfrac{10}{h^{2}} & x\in \text{the bottom-right corner finite grid element}
\end{cases},
\end{equation*}
and
\begin{equation*}
q_w(t,x) = \begin{cases}
\cfrac{10}{h^{2}} & x\in \text{ the top-left corner finite grid finite element}\\
-S(t,w)\cfrac{10}{h^{2}} & x\in \text{the bottom-right corner finite grid element}
\end{cases}.
\end{equation*}


In this experiment, we carried out three tests.  The first one is with data taken on full domain $\Omega$, and the other two tests are carried out with data only taken on one of the following square sub-domains at the top - left corner:
$$\Omega_1 =  [0,50]^2 \hspace{0.1 cm} , \hspace{1 cm}  \Omega_2 =  [0,25]^2 .$$

The plots in Figure \ref{fig:error_case1} shows the saturation error in $L^2$ vs. time, where the observational
data are collected from different fractions of the domain. The solution without data assimilation  $\mu = 0$  has only negligible drop (in blue) in its residual error. This underscores the significance of a nudged solution synchronizing  with the reference solution. Full nudging (in red) indicates  synchronization with  the reference solution  roughly at an exponential rate.  We then  continue  by testing the effect of the size of the sub-domain. Machine precision is reached for data collected over the whole domain.  By then, in the case of $\Omega_1$ and $\Omega_2$, the error is within $10^{-10}$ and $10^{-5}$ respectively.  Particularly in the case of subdomains $\Omega_1$ and $\Omega_2$, we see from the
snapshot plots in Figures \ref{Sub1}, \ref{Sub2}, \ref{Sub3}, and  \ref{Sub4} that the main spatial features over the full domain  $\Omega$ are nevertheless captured as time evolves. 

The convergence of the DA solution with deceasing $H$ is shown in table \ref{tab:error_case1_vary_H}. The convergence of the DA solution
to the true solution in time can also be seen in the snapshot  plots of the  solutions  in Figure \ref{fig:error_case1-1-3}. Here at $t = 0.1$,  there is of course a major difference, since the DA simulation starts at $t=0$. The accuracy of DA is
seen to increase by $t=1$ and further by $t=10$. Finally by $t=100$ there is no visual difference between DA and reference solution, which we expect since the $L^2$ difference between the solutions at $t=100$ is seen in Figure \ref{fig:error_case1} to be near $10^{-15}$.  Moreover, the snapshot plots of the true solution starting from zero initial value (bottom plots  in Figure \ref{fig:error_case1-1-3}) indicate the sensitivity of the solution to the initial conditions.


\begin{figure}[!h]
\centering\includegraphics[width=5.5cm,height=3.5cm]{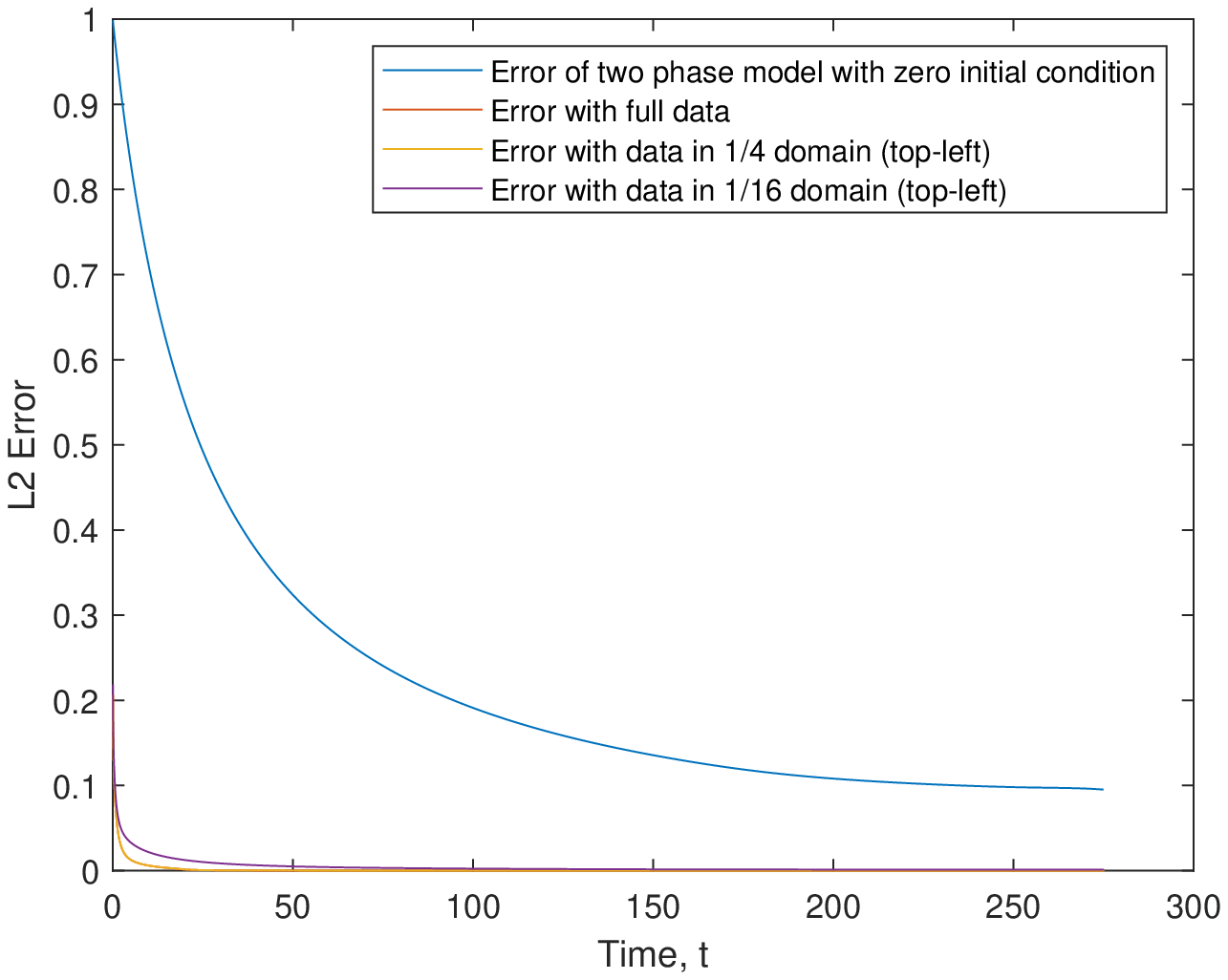} \includegraphics[width=5.5cm,height=3.5cm]{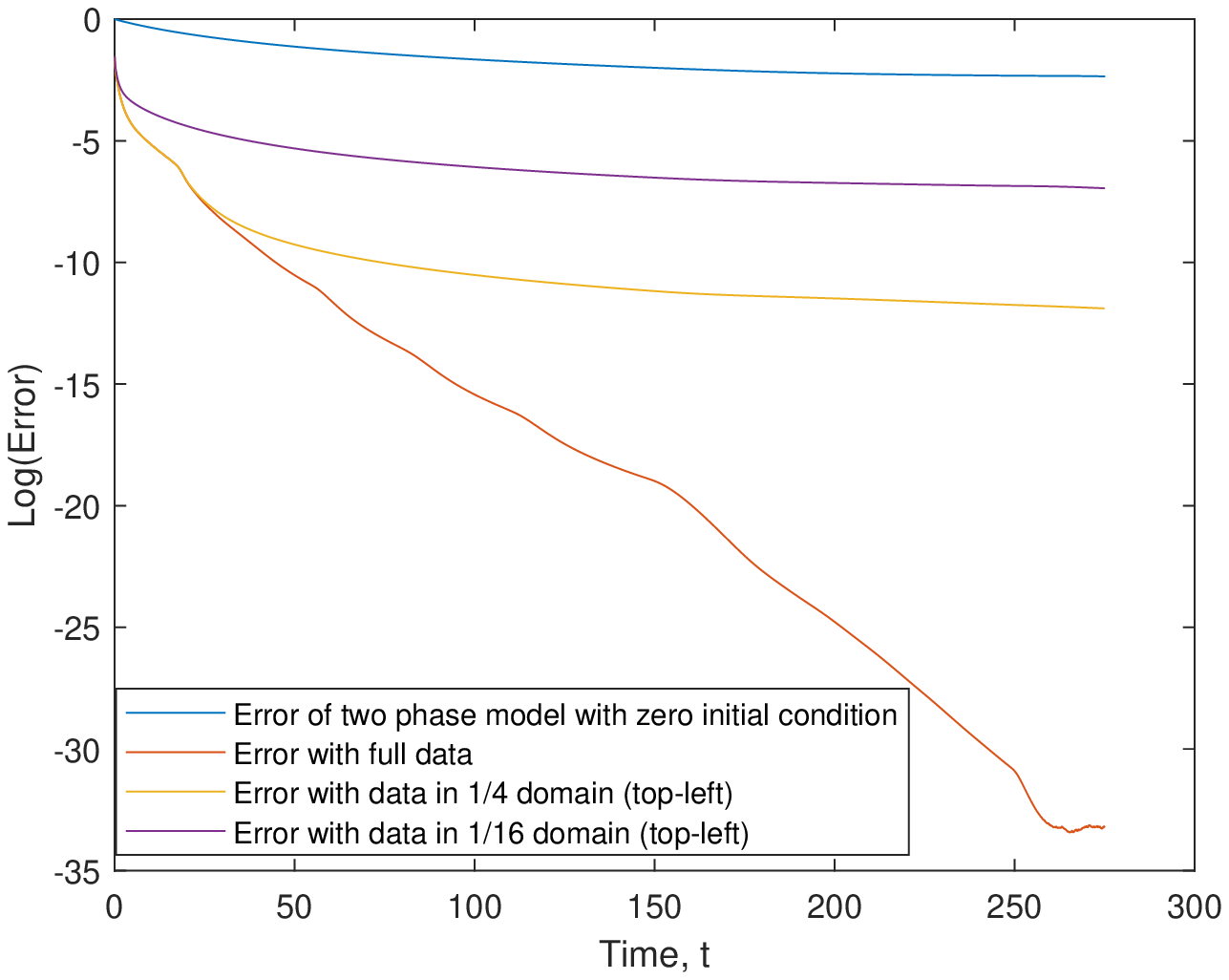}

\caption{Error comparison $\|S(t) - \tilde{S}(t)\|_{L^2}$. Left: $L_2$ error. Right: $L_2$ error in log-scale.}
\label{fig:error_case1}
\end{figure}

\begin{table}
\begin{tabular}{|c|c|c|c|}
\hline 
$(H,\mu)$ & $(1/5,100)$ & $(1/10,200)$ & $(1/20,400)$\tabularnewline
\hline 
$L^{2}$-error at $T=10$ & $0.2499$ & $0.0619$ & $0.0080$\tabularnewline
\hline 
\end{tabular}
\caption{Error convergence with varying $H$}
\label{tab:error_case1_vary_H}
\end{table}

\begin{figure}[!h]
\begin{minipage}[t]{0.2\textwidth}
         \centering
         \includegraphics[width=3cm,height=2cm]{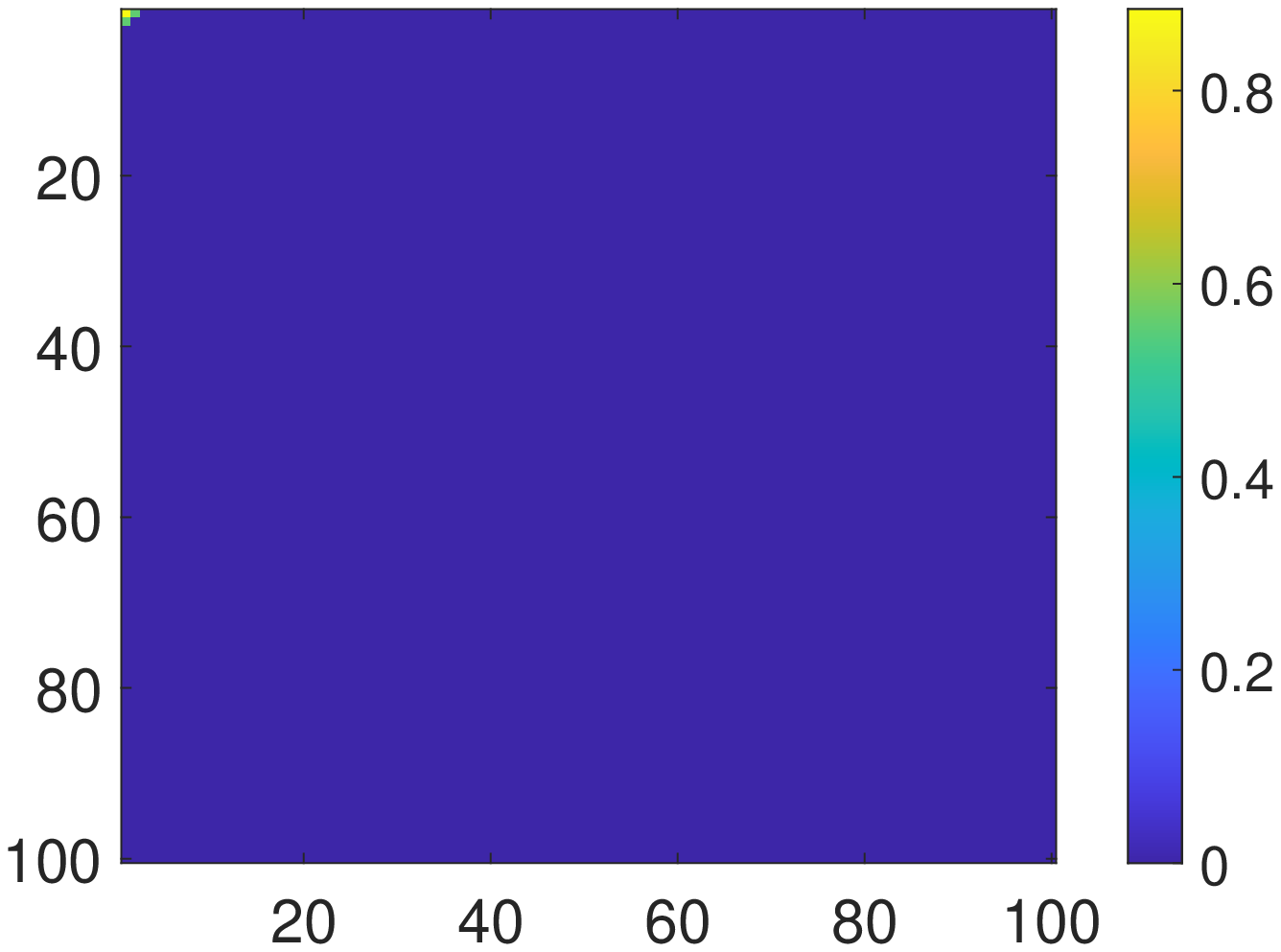}

        \includegraphics[width=3cm,height=2cm]{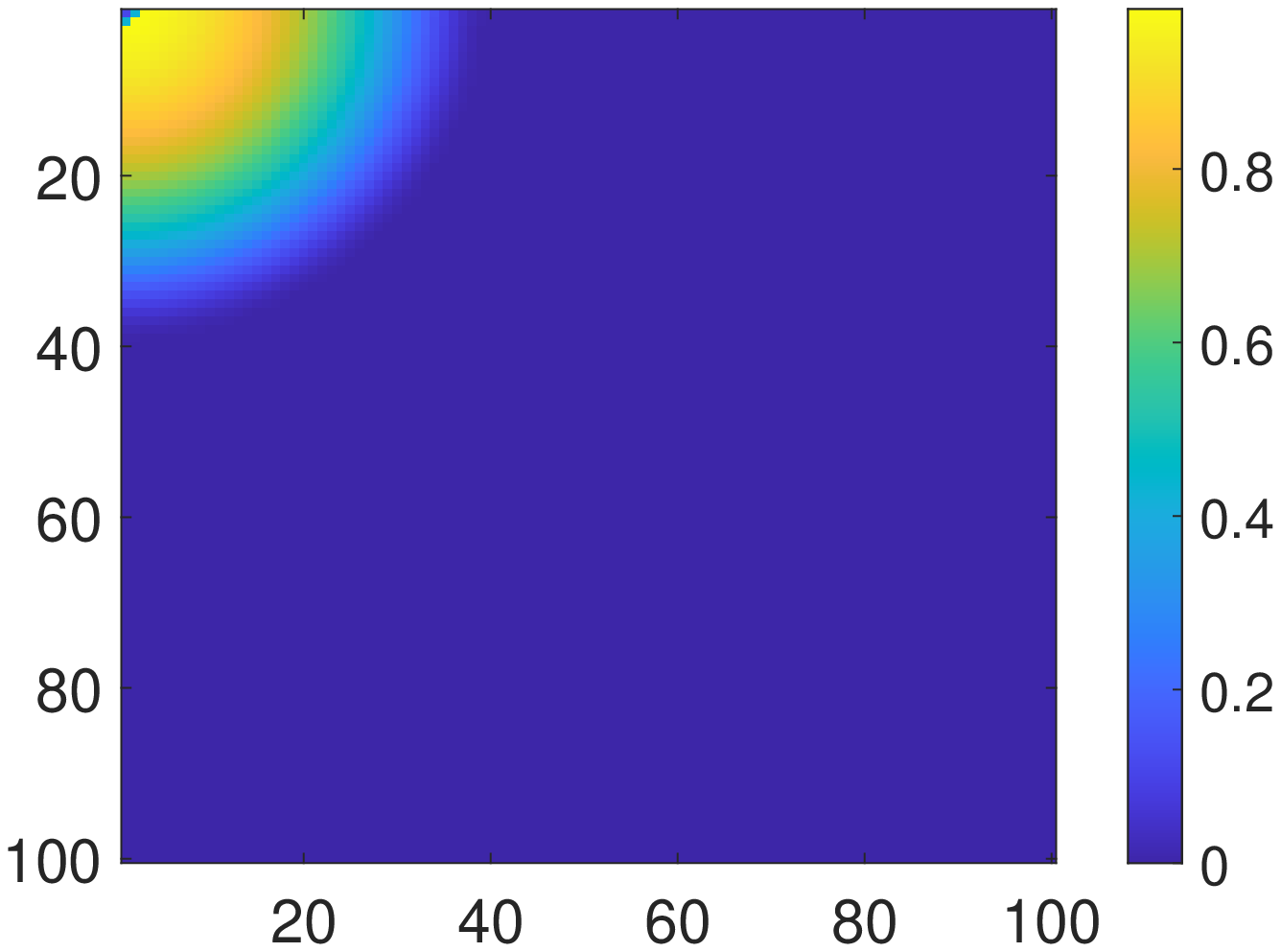}
\end{minipage}
\begin{minipage}[t]{0.2\textwidth}
         \centering
         \includegraphics[width=3cm,height=2cm]{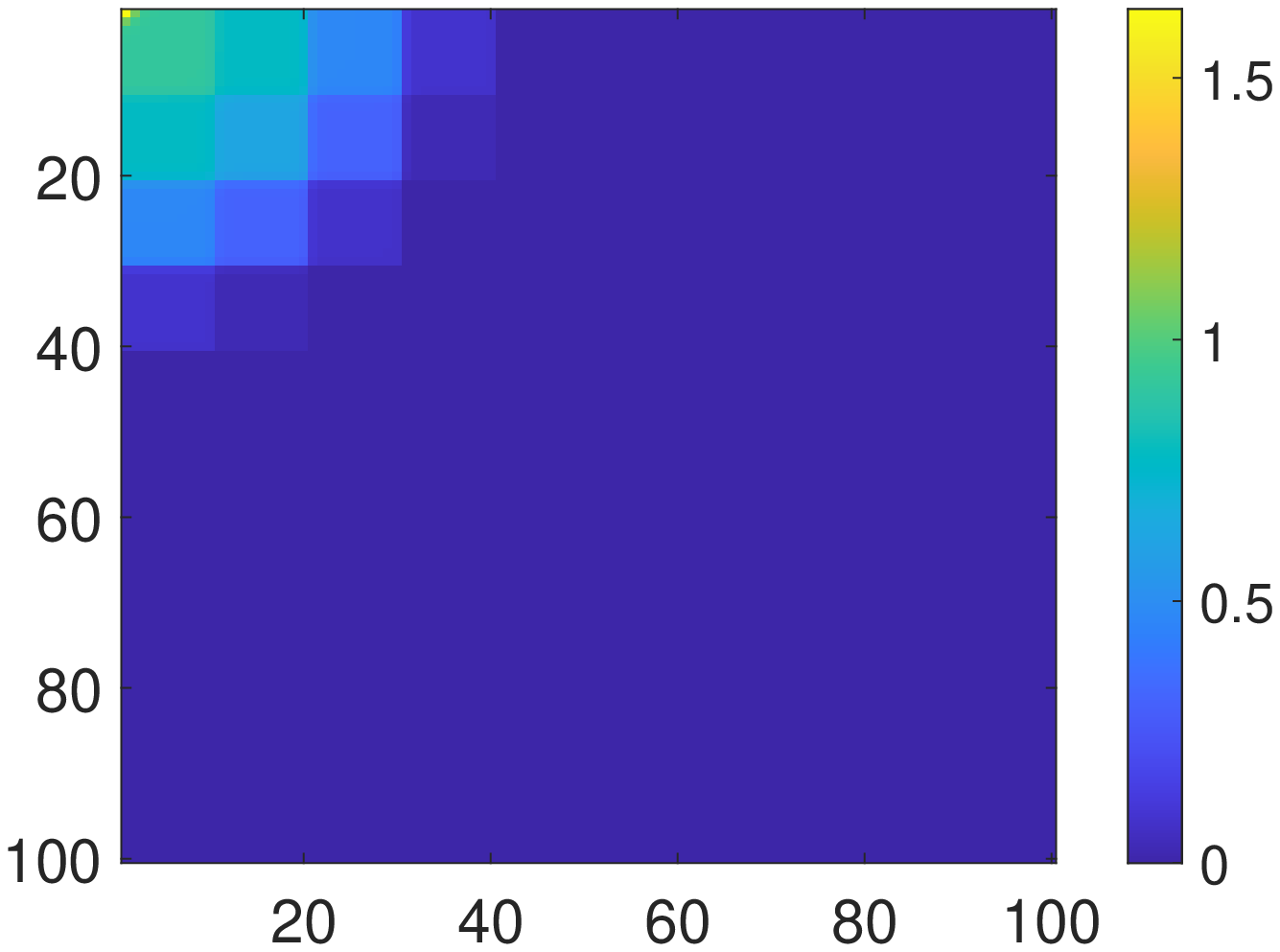}

         \includegraphics[width=3cm,height=2cm]{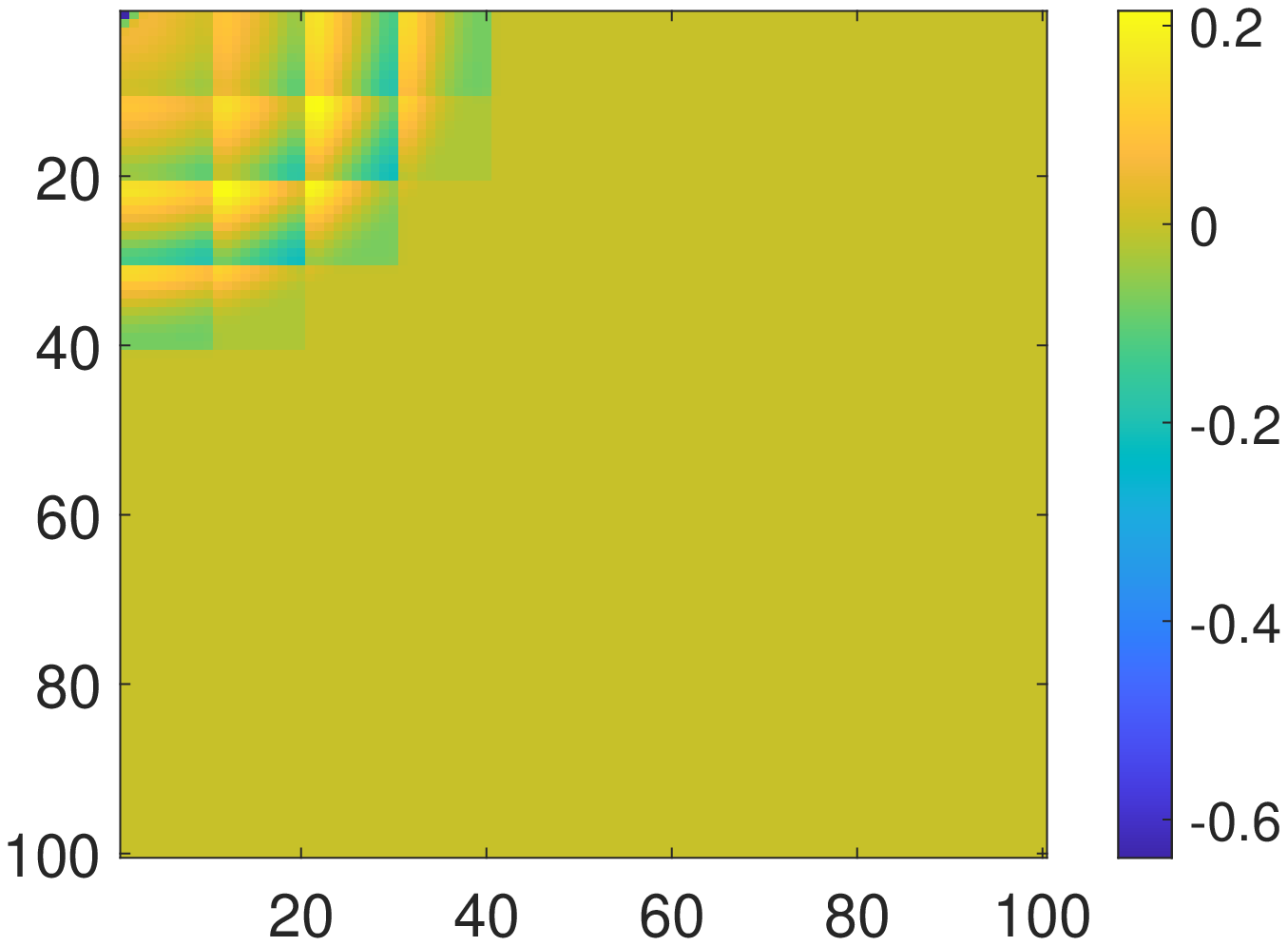}
\end{minipage}
\begin{minipage}[t]{0.2\textwidth}
         \centering
         \includegraphics[width=3cm,height=2cm]{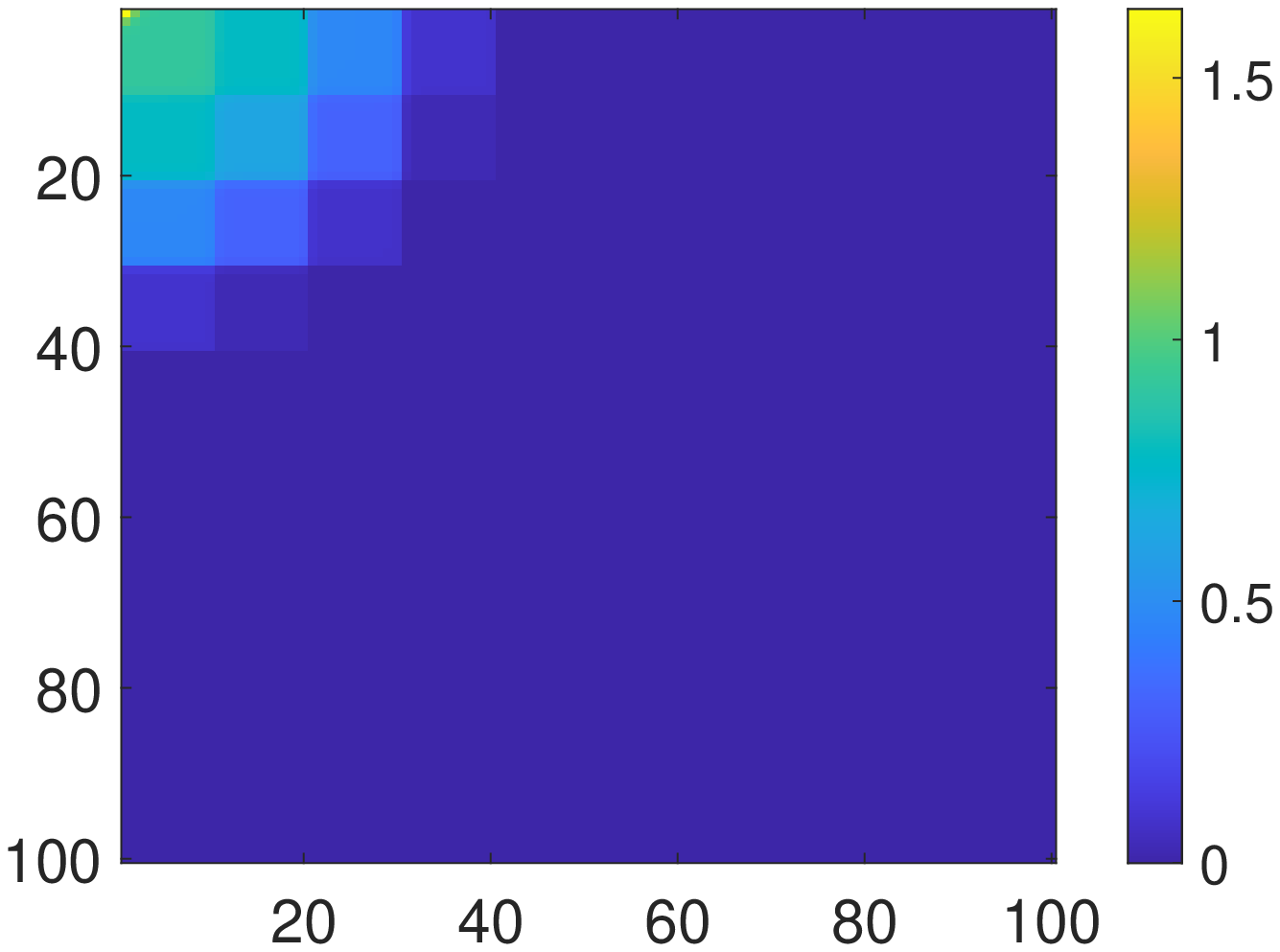}
         
         \includegraphics[width=3cm,height=2cm]{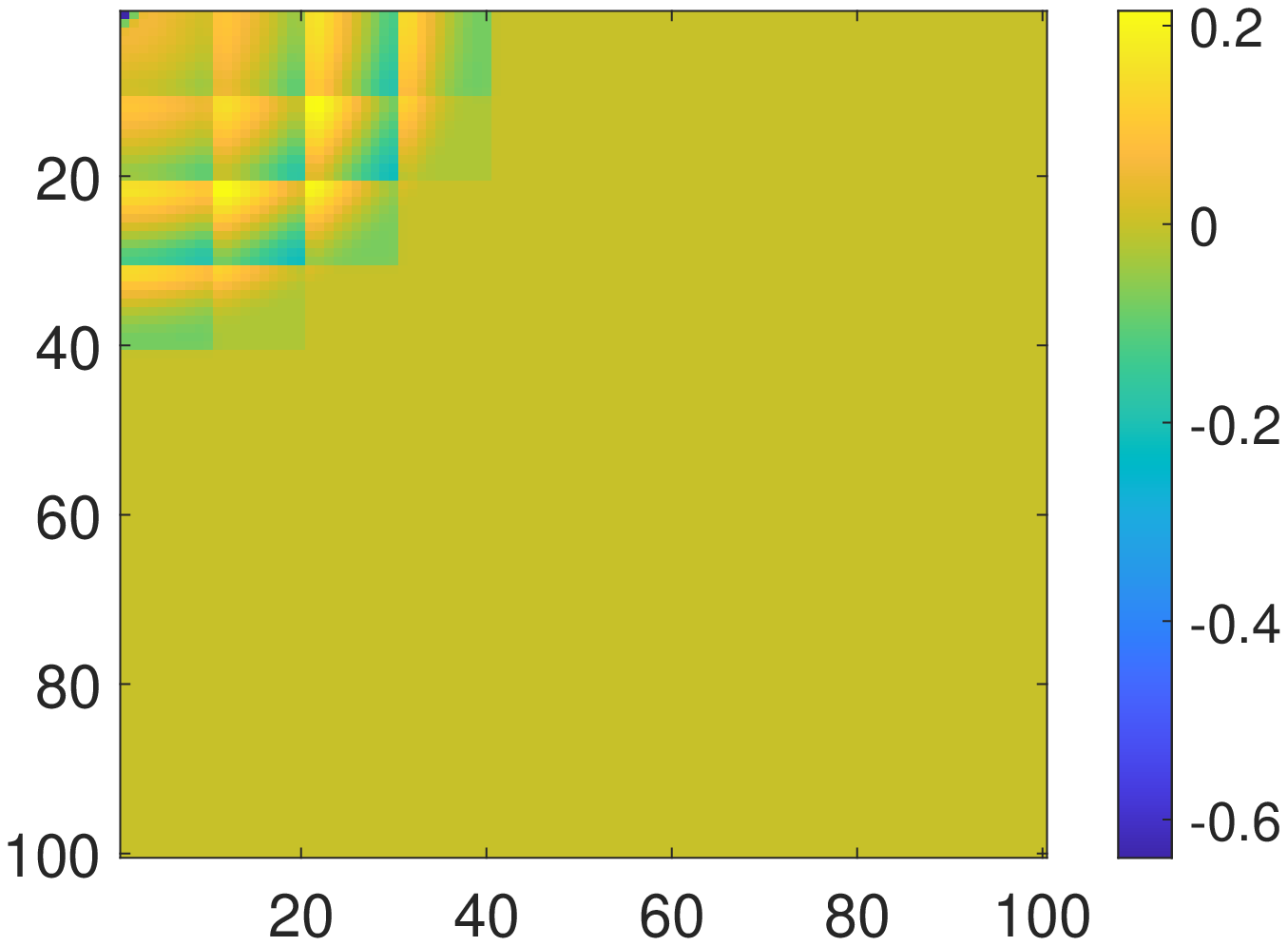}
\end{minipage}
\begin{minipage}[t]{0.2\textwidth}
         \centering
         \includegraphics[width=3cm,height=2cm]{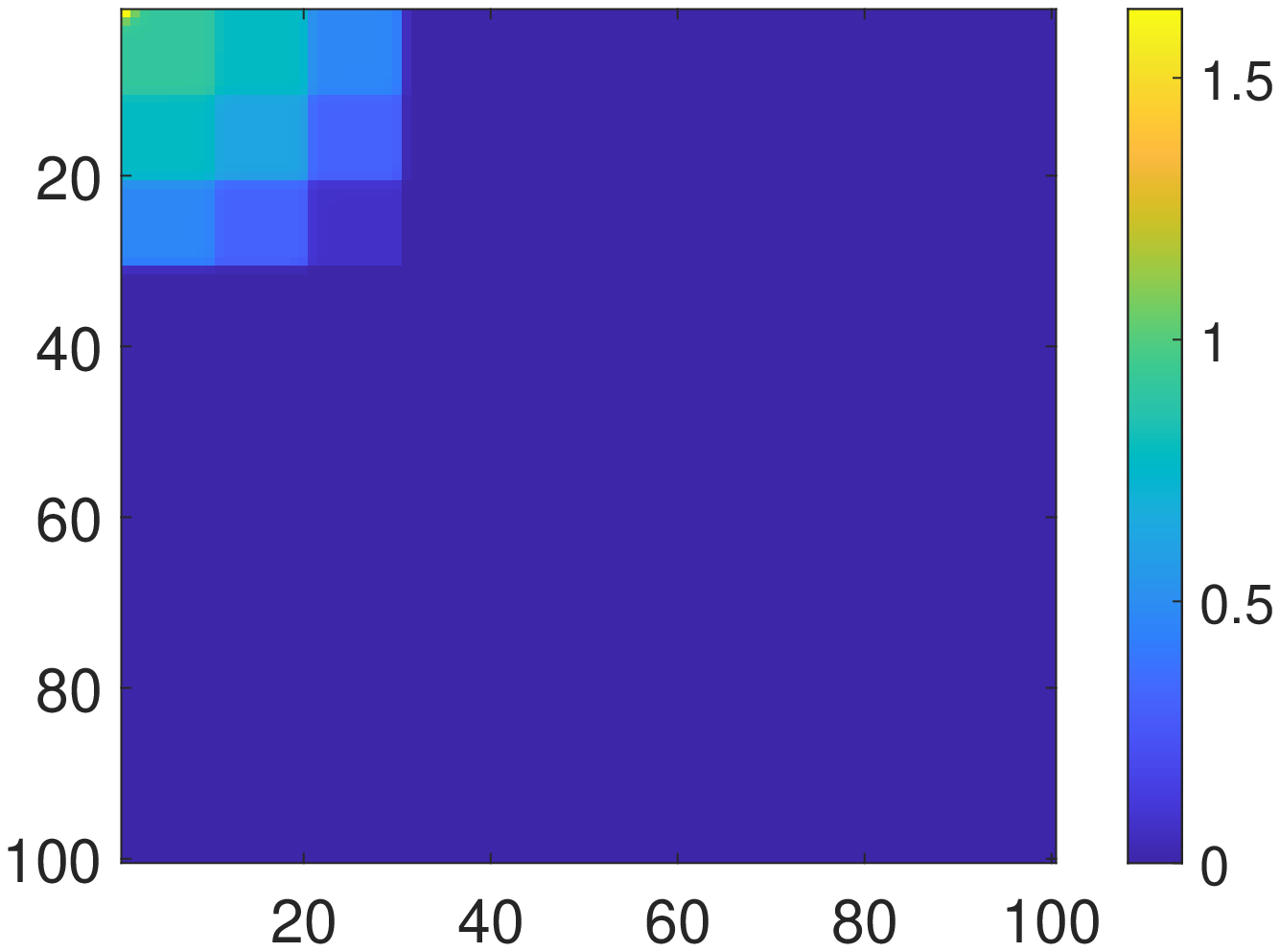}
         
         \includegraphics[width=3cm,height=2cm]{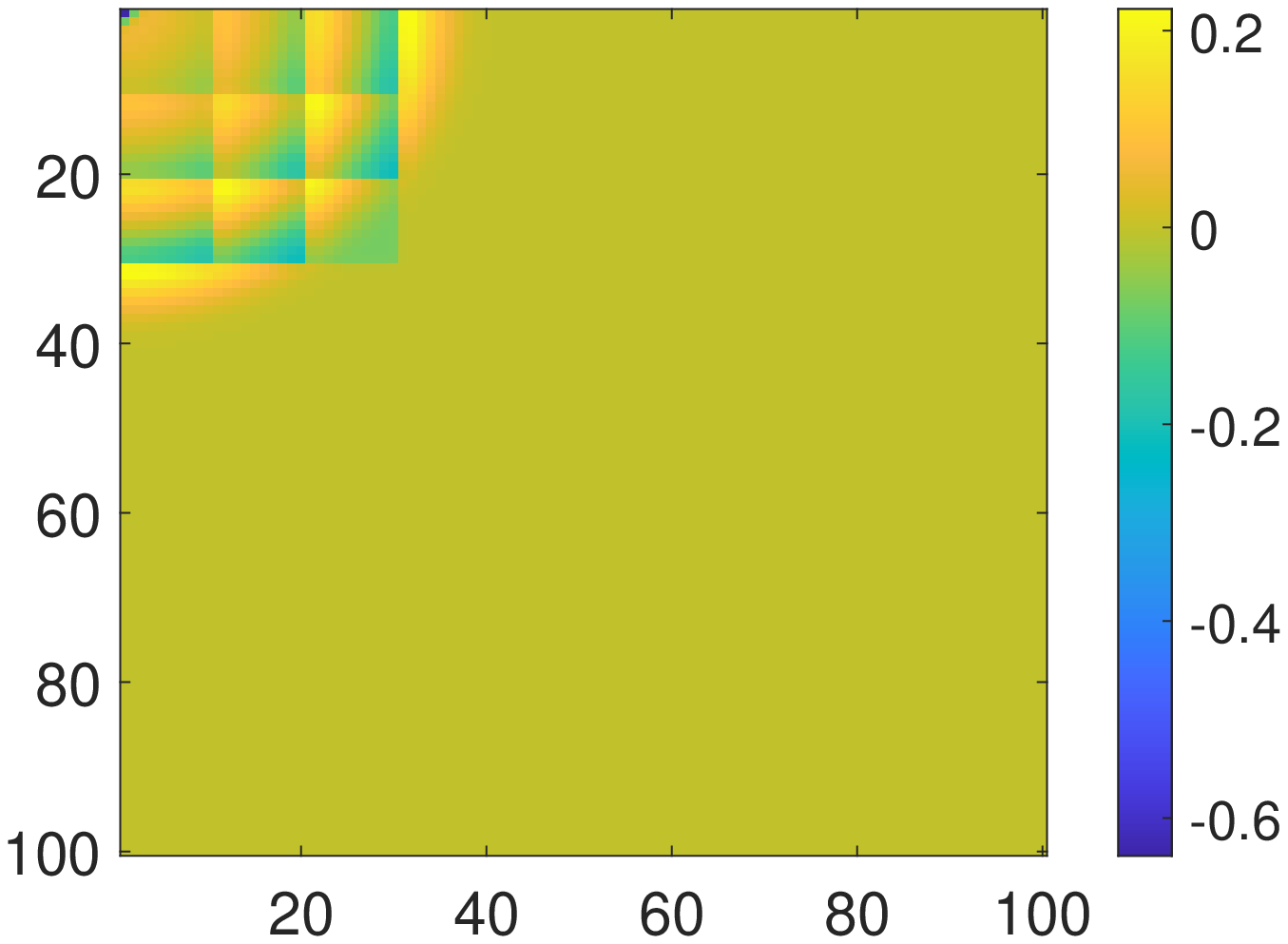}
\end{minipage}
\caption{Snapshots of $S$ and comparison of the solutions with different data at $T=0.1$.  Top: snapshot  of $S$. Bottom: pointwise error of the solution. Left: two-phase  solution. Middle-Left:
full data. Middle-Right: $\frac{1}{4}$ data.
 Right: $\frac{1}{16}$ data}
 \label{Sub1}
\end{figure}

\begin{figure}[!h]
\begin{minipage}[t]{0.2\textwidth}
         \centering
         \includegraphics[width=3cm,height=2cm]{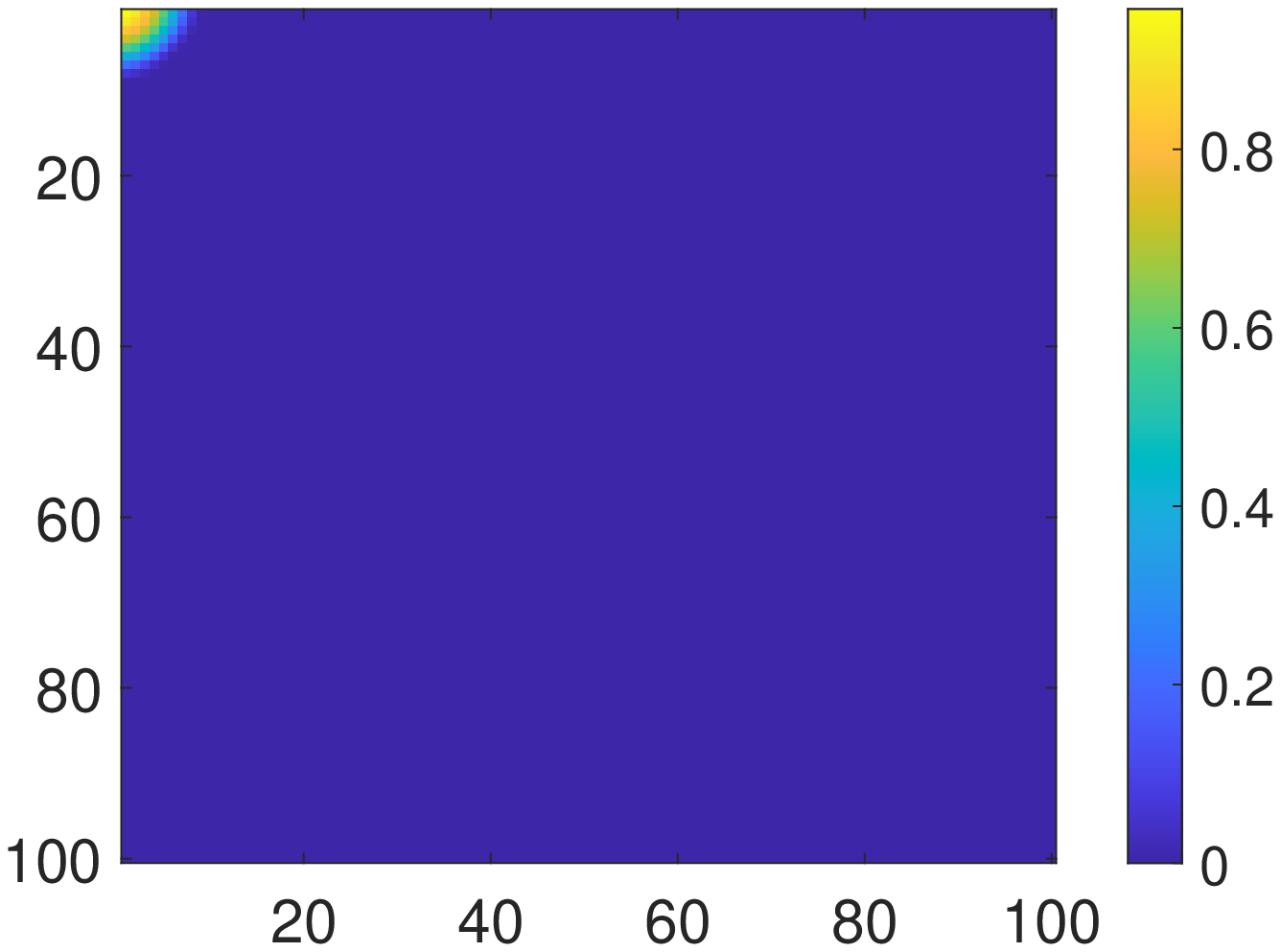}

        \includegraphics[width=3cm,height=2cm]{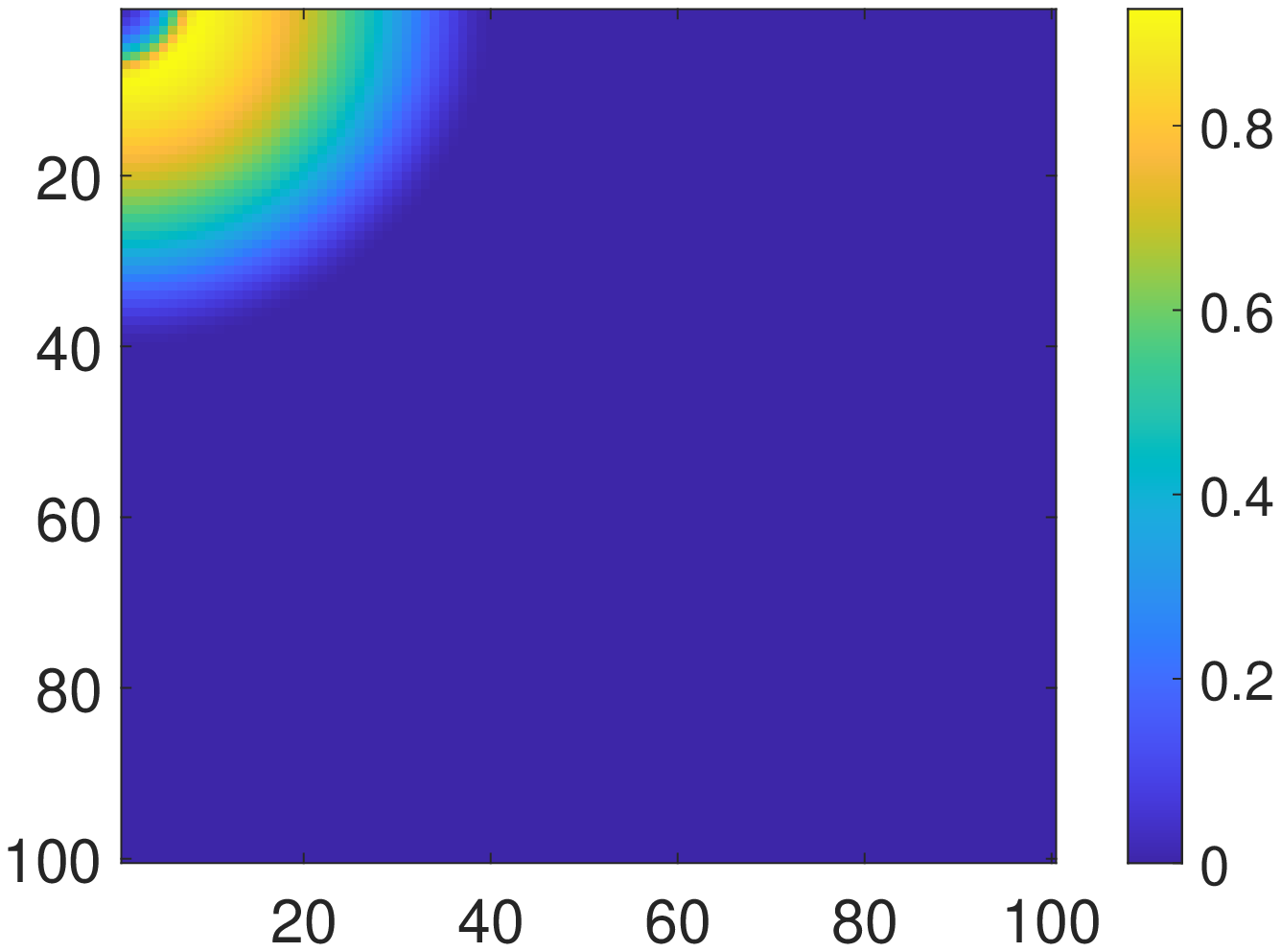}
\end{minipage}
\begin{minipage}[t]{0.2\textwidth}
         \centering
         \includegraphics[width=3cm,height=2cm]{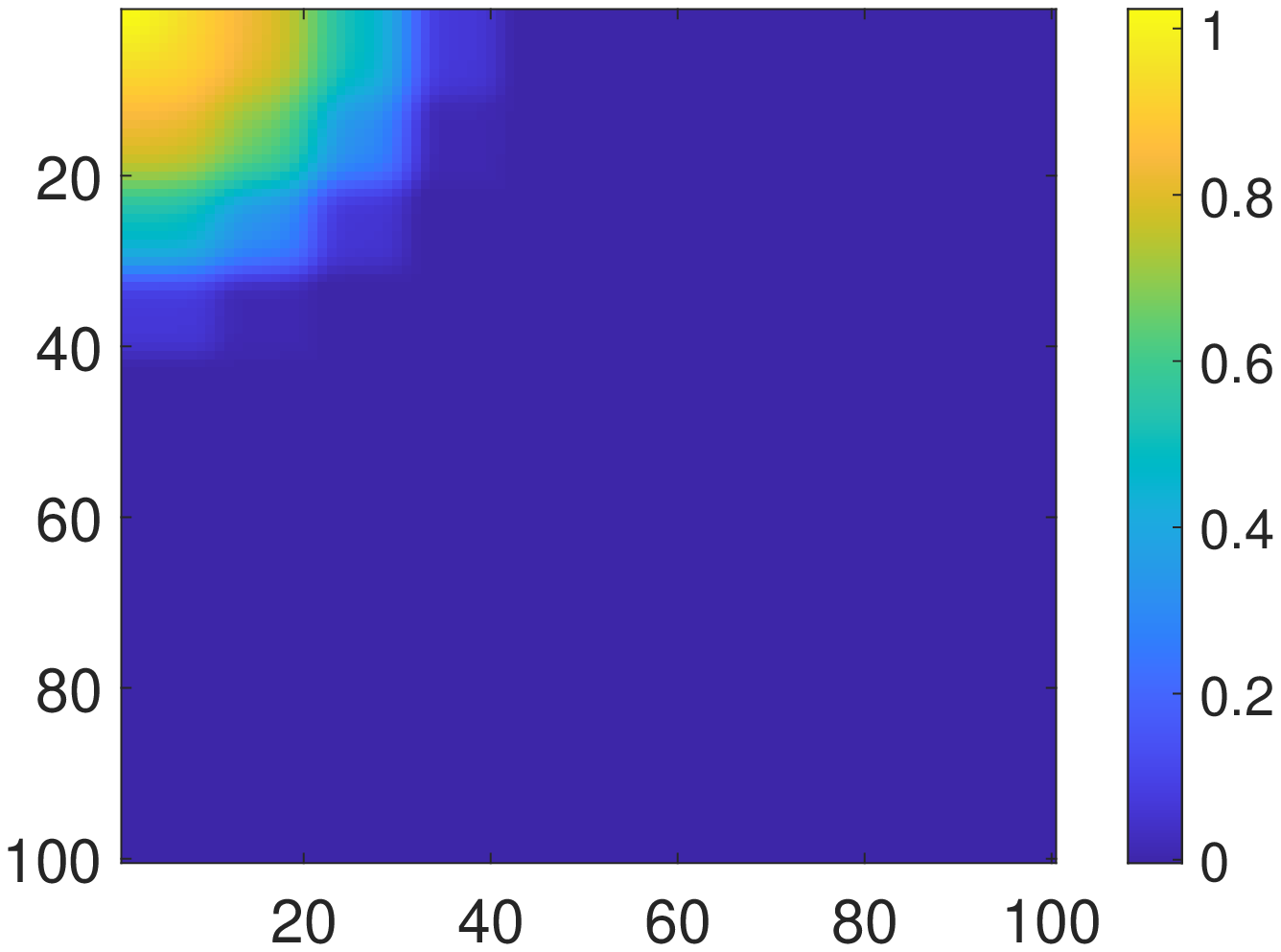}

         \includegraphics[width=3cm,height=2cm]{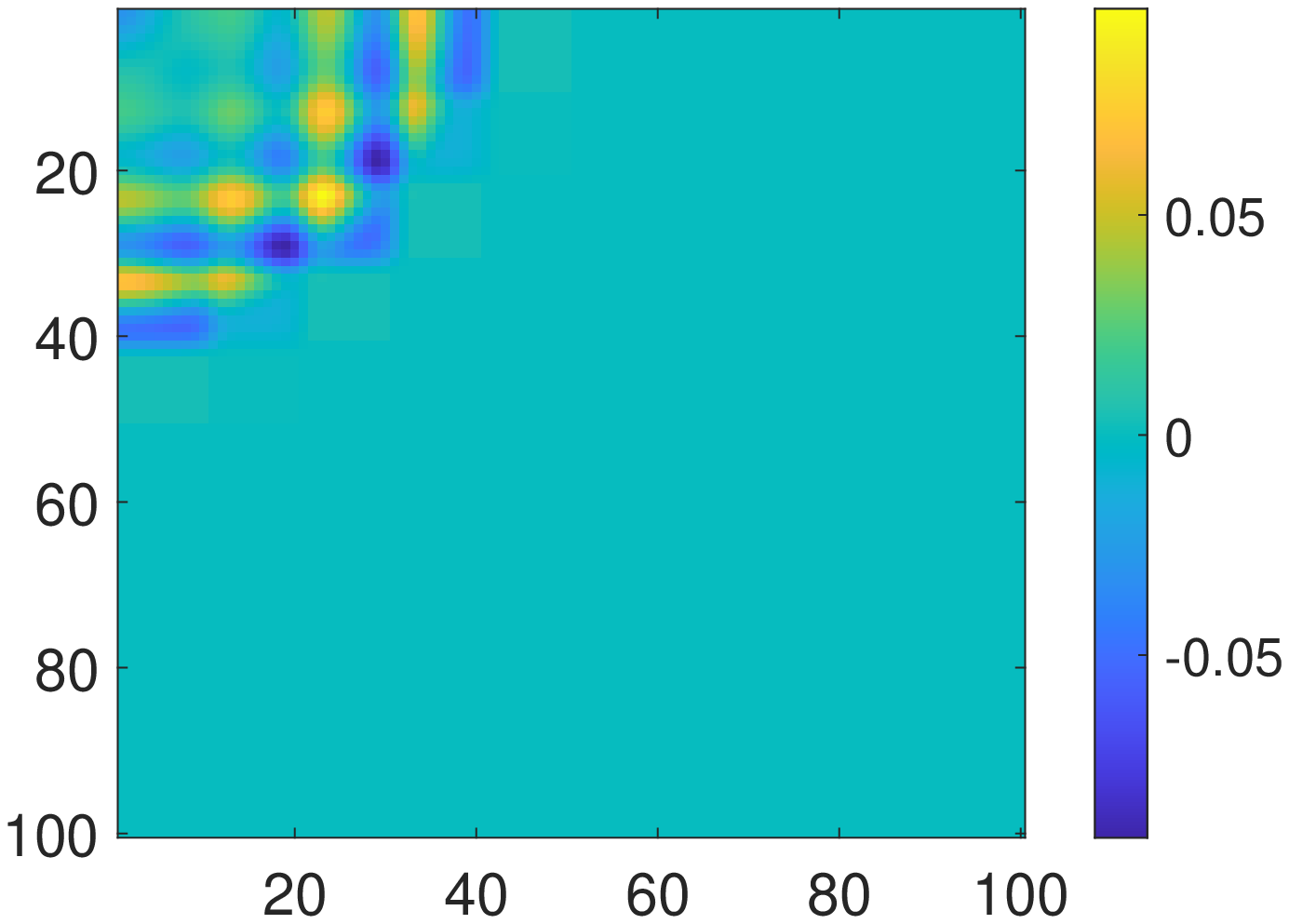}
\end{minipage}
\begin{minipage}[t]{0.2\textwidth}
         \centering
         \includegraphics[width=3cm,height=2cm]{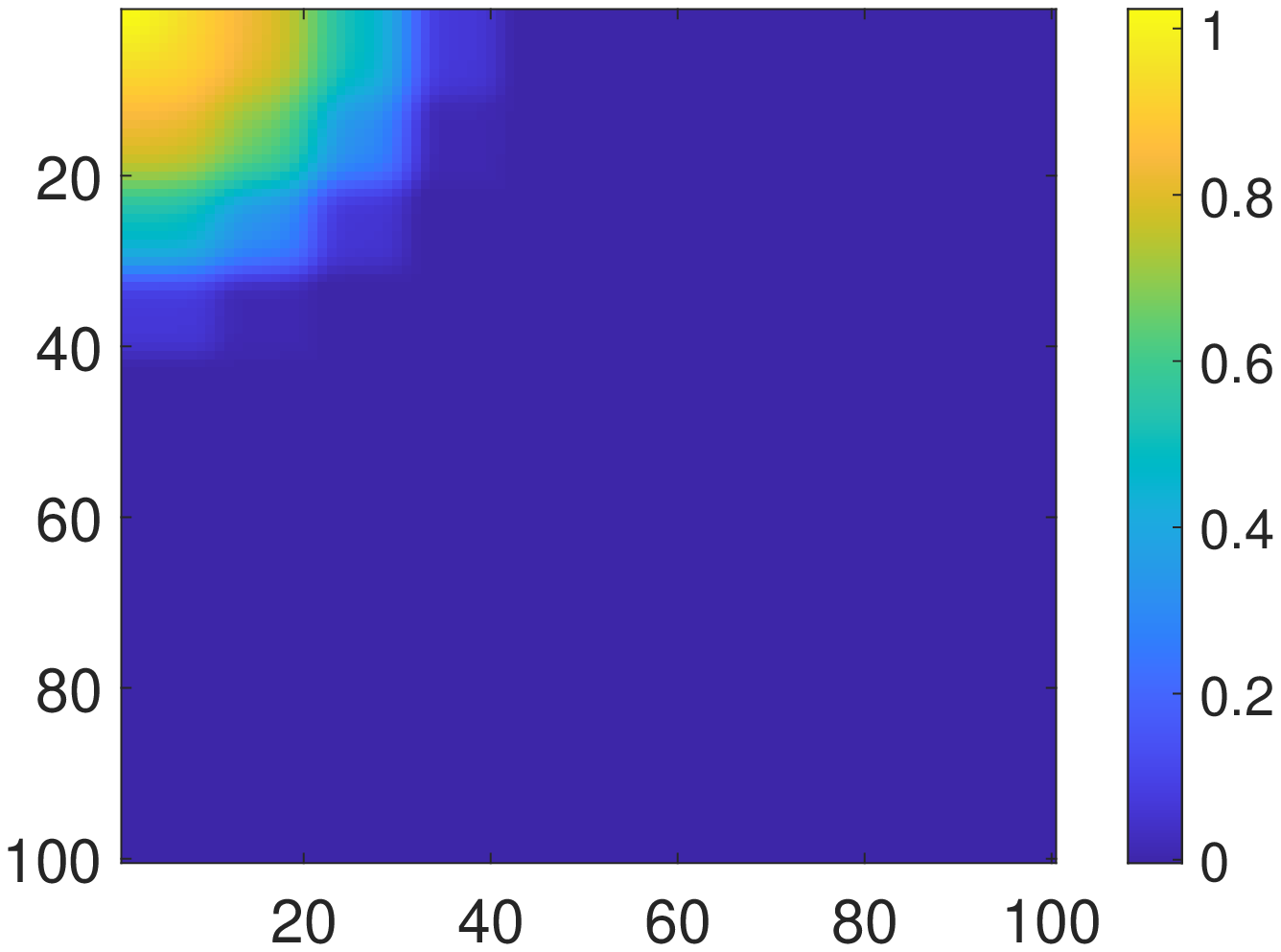}
         
         \includegraphics[width=3cm,height=2cm]{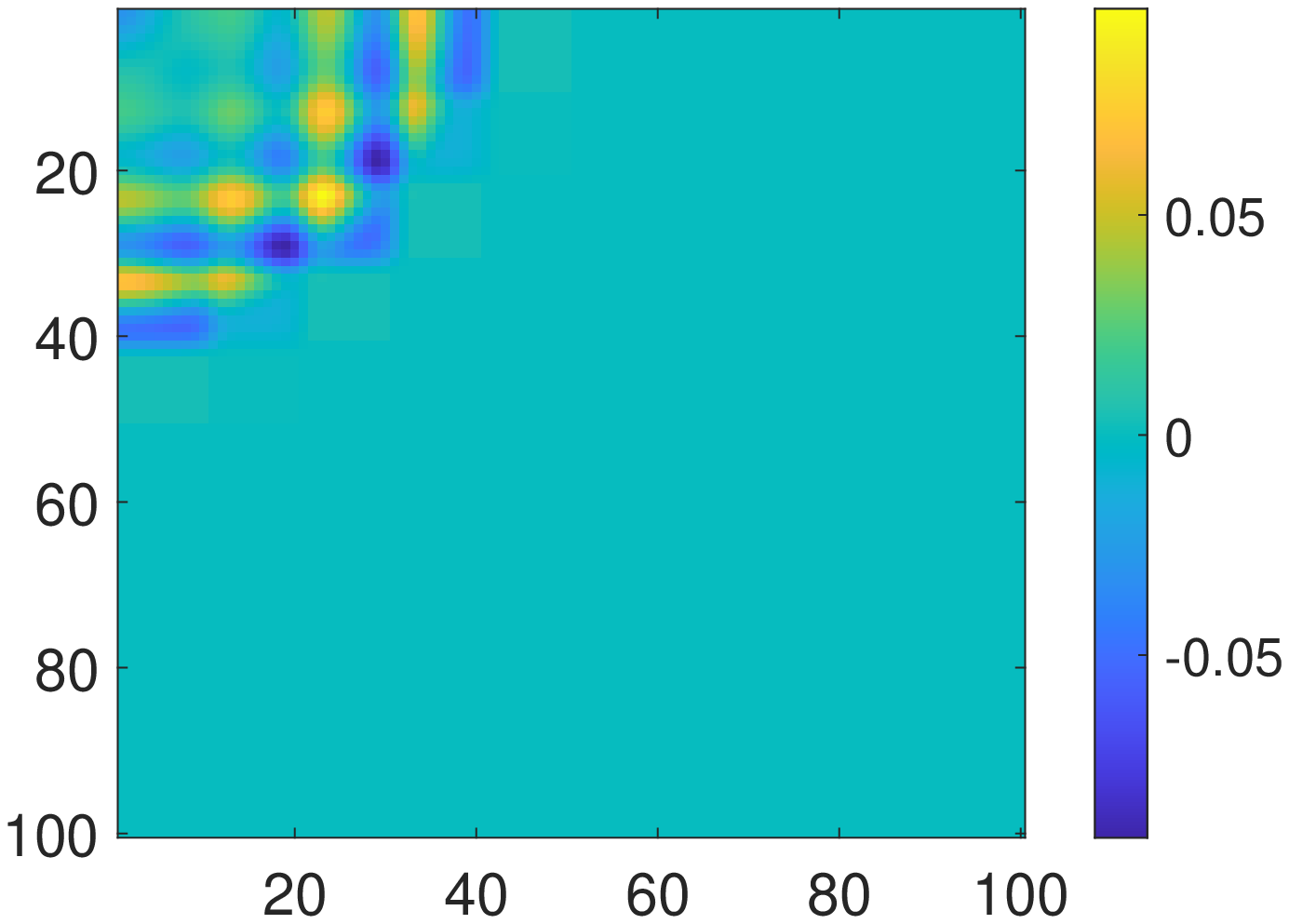}
\end{minipage}
\begin{minipage}[t]{0.2\textwidth}
         \centering
         \includegraphics[width=3cm,height=2cm]{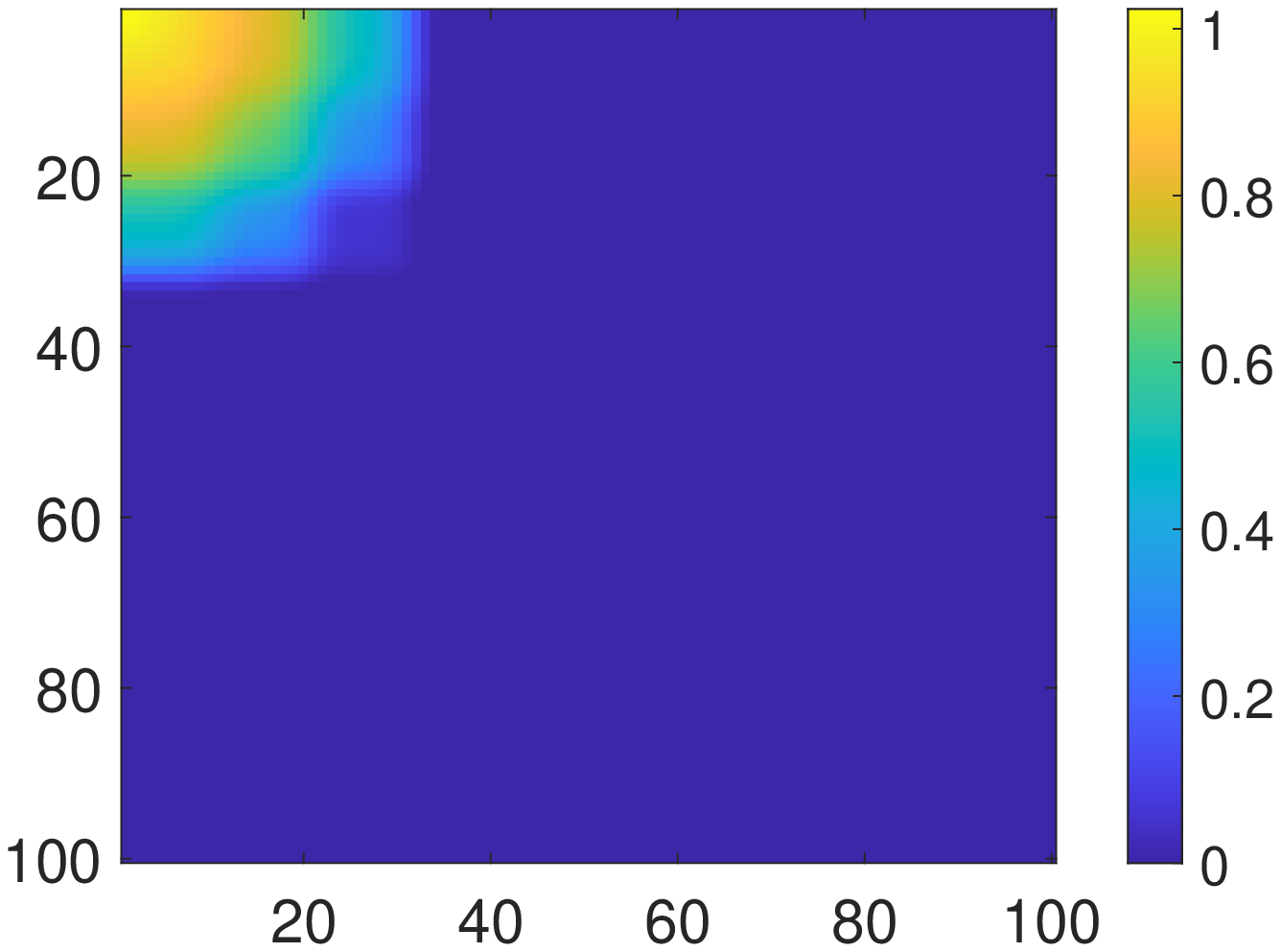}
         
         \includegraphics[width=3cm,height=2cm]{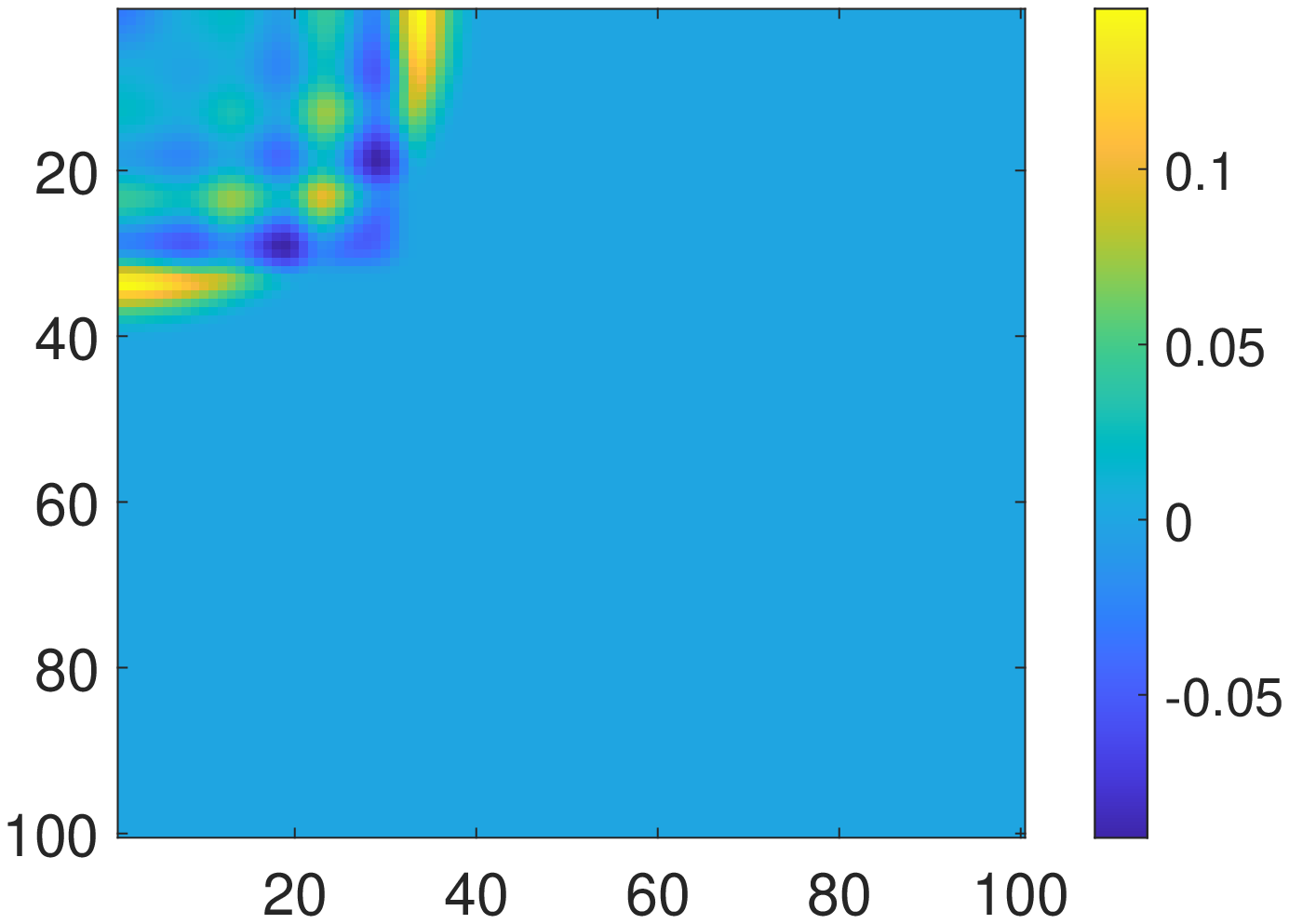}
\end{minipage}
\caption{Snapshots  of $S$ and comparison of the solutions with different data at $T=1$. Top: snapshot  of $S$ . Bottom: pointwise error of the solution. Left: two-phase  solution. Middle-Left:
full data. Middle-Right: $\frac{1}{4}$ data.
 Right: $\frac{1}{16}$ data}
 \label{Sub2}
\end{figure}

\begin{figure}[!h]
\begin{minipage}[t]{0.2\textwidth}
         \centering
         \includegraphics[width=3cm,height=2cm]{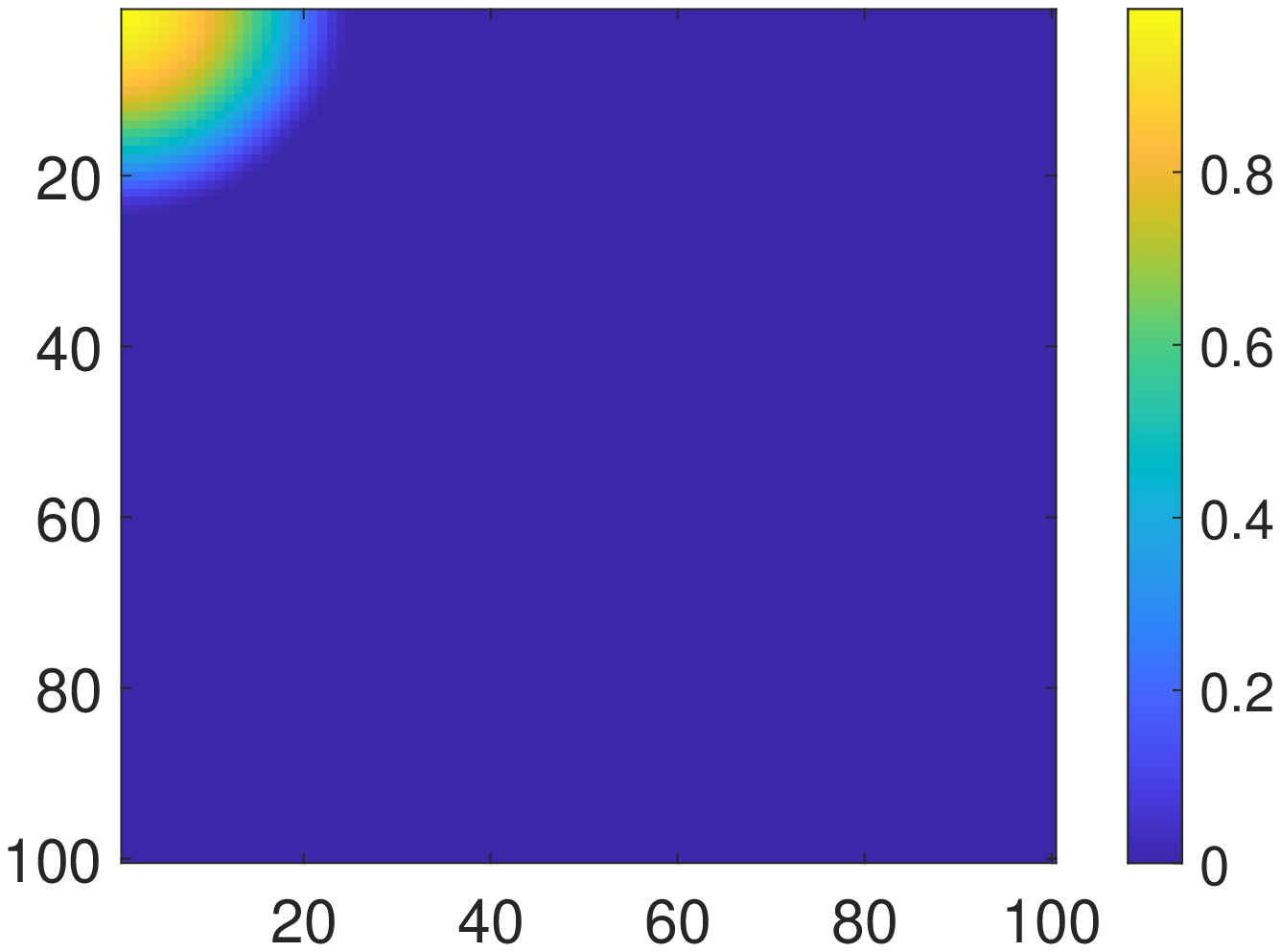}

        \includegraphics[width=3cm,height=2cm]{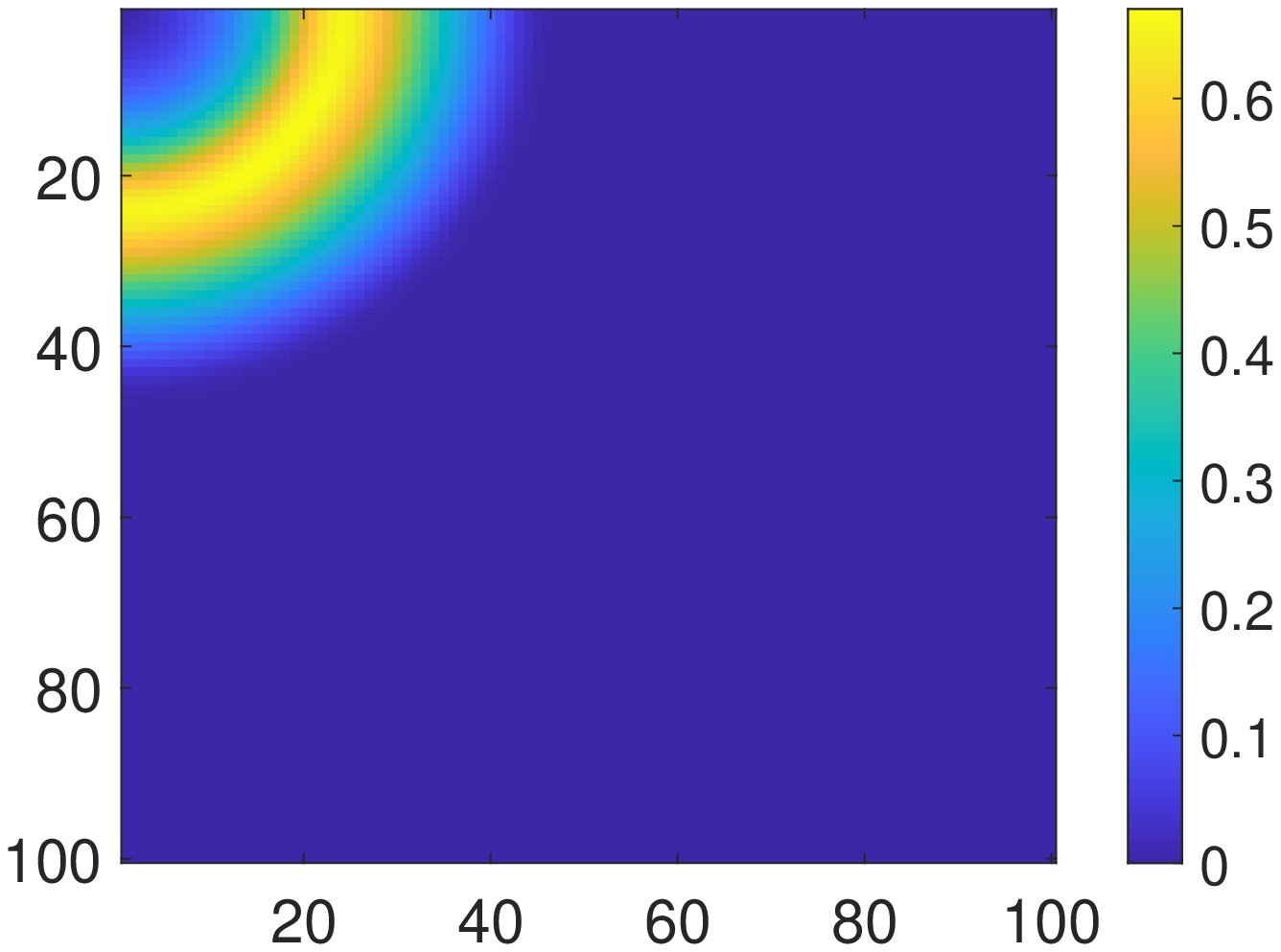}
\end{minipage}
\begin{minipage}[t]{0.2\textwidth}
         \centering
         \includegraphics[width=3cm,height=2cm]{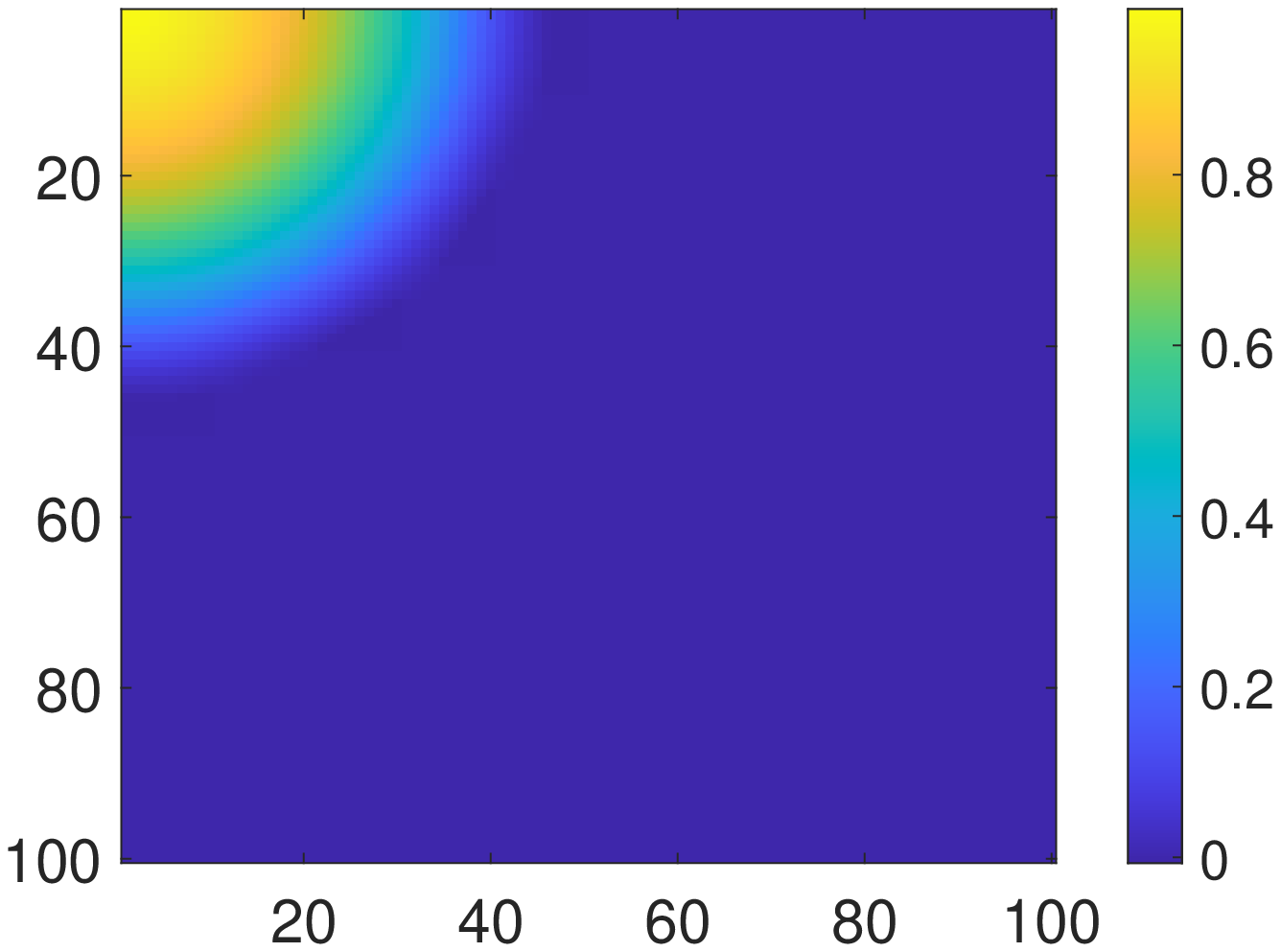}

         \includegraphics[width=3cm,height=2cm]{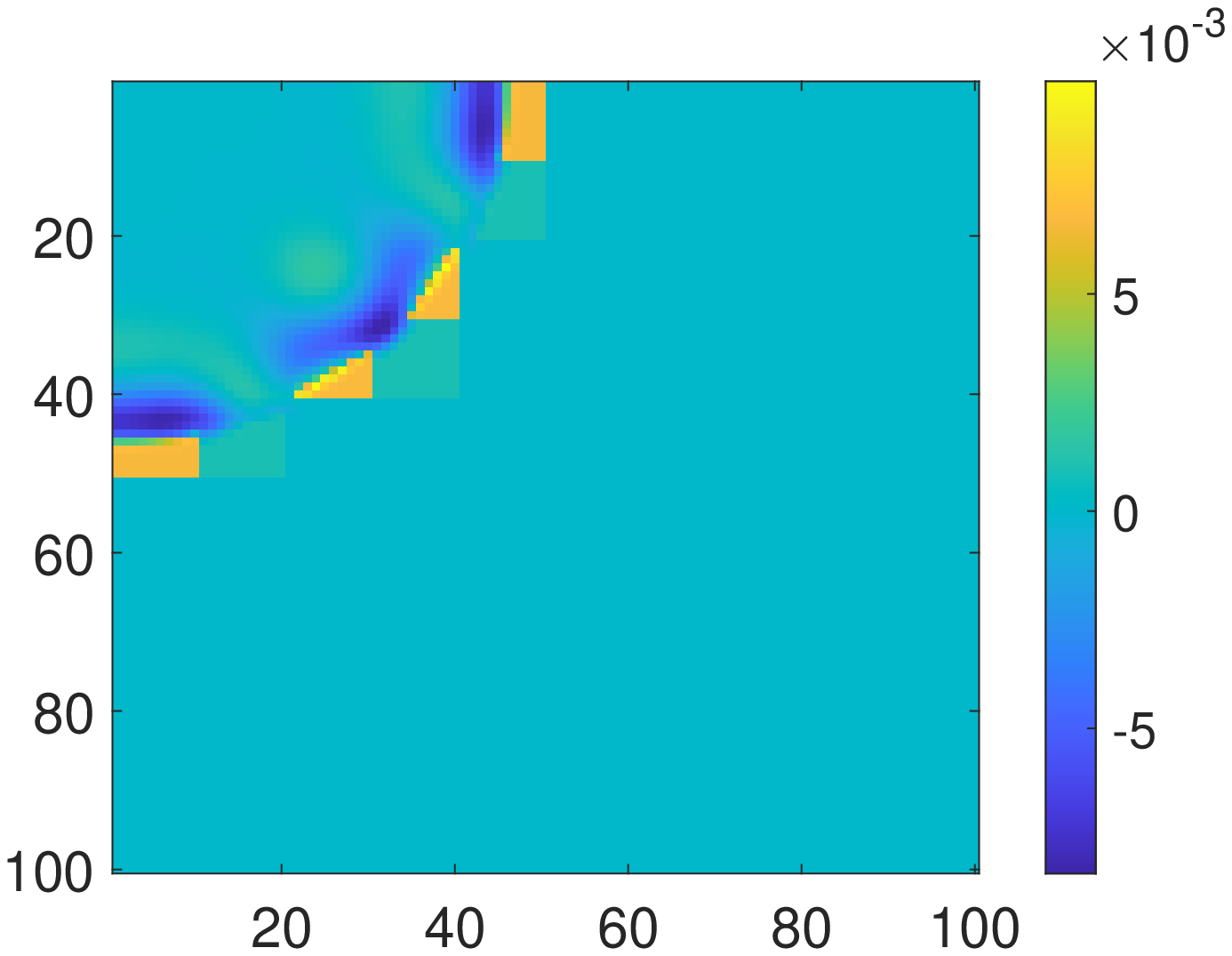}
\end{minipage}
\begin{minipage}[t]{0.2\textwidth}
         \centering
         \includegraphics[width=3cm,height=2cm]{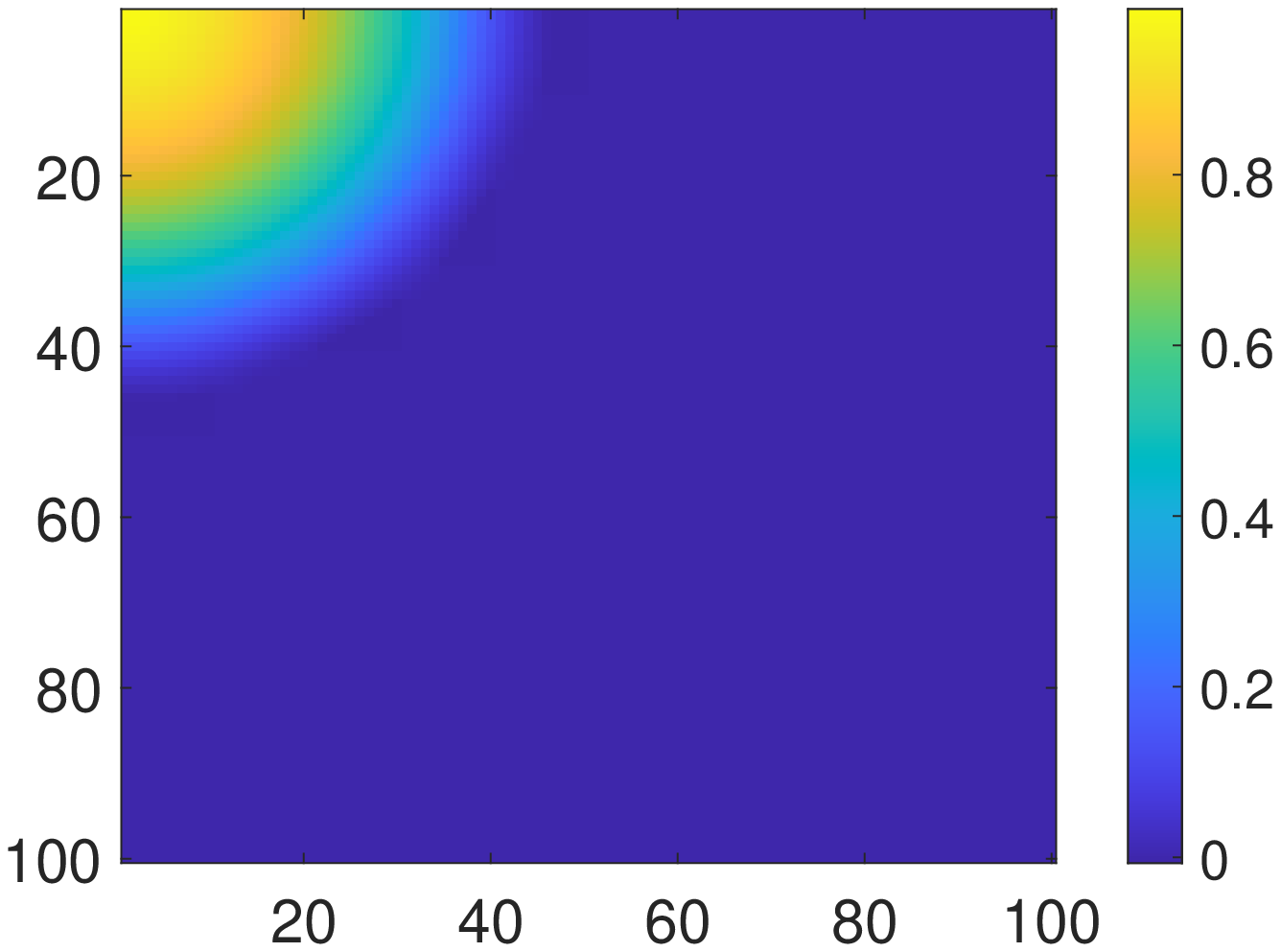}
         
         \includegraphics[width=3cm,height=2cm]{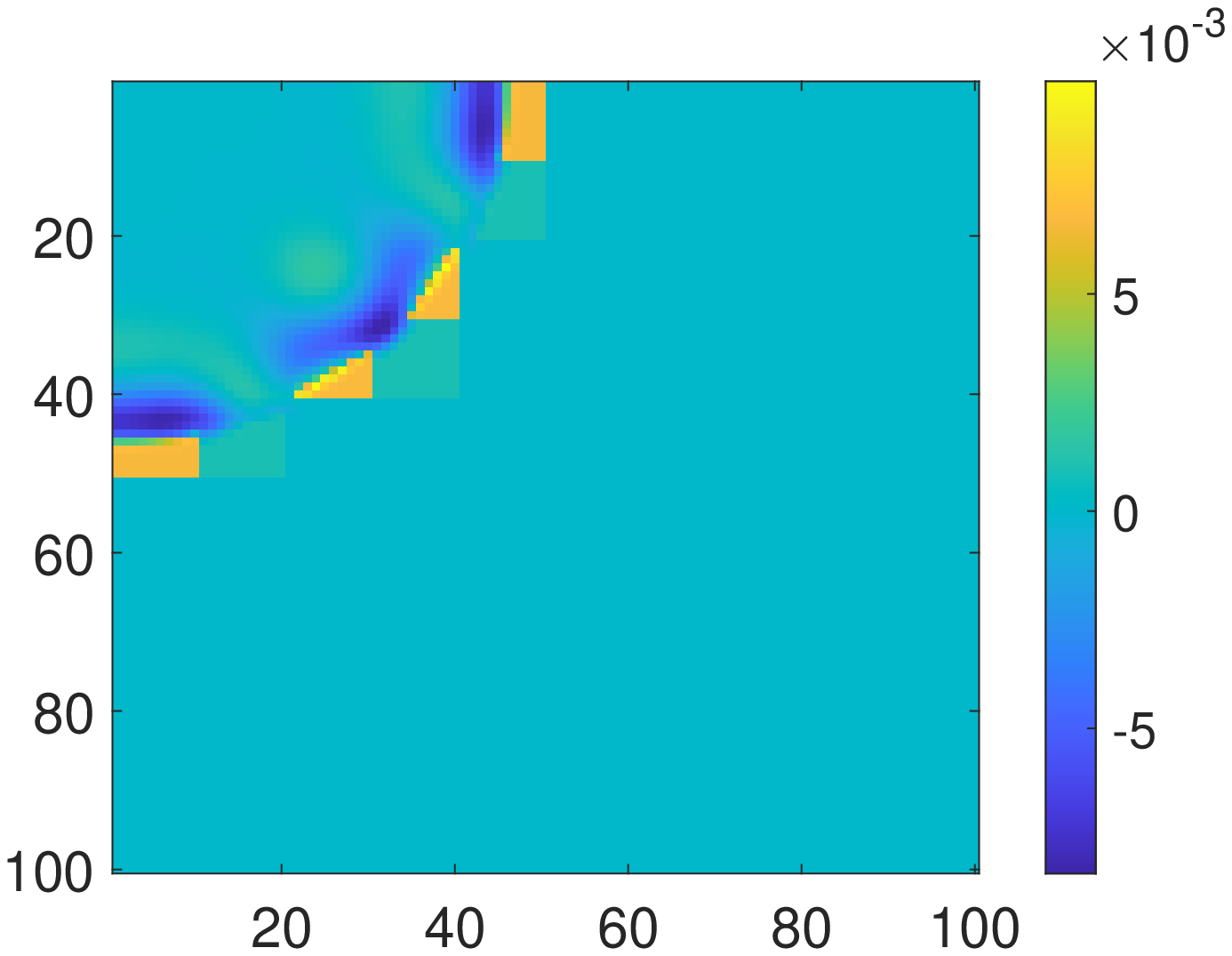}
\end{minipage}
\begin{minipage}[t]{0.2\textwidth}
         \centering
         \includegraphics[width=3cm,height=2cm]{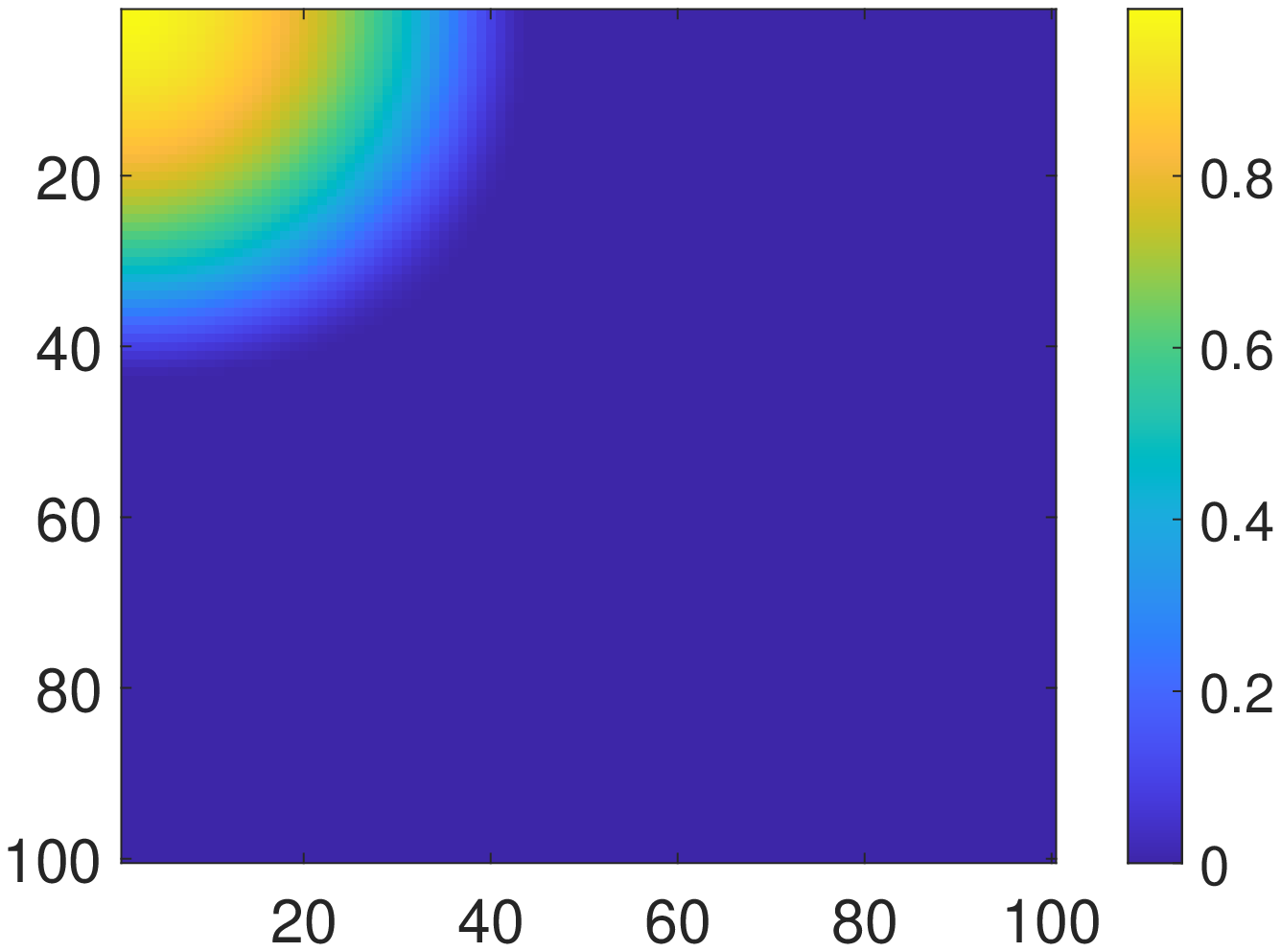}
         
         \includegraphics[width=3cm,height=2cm]{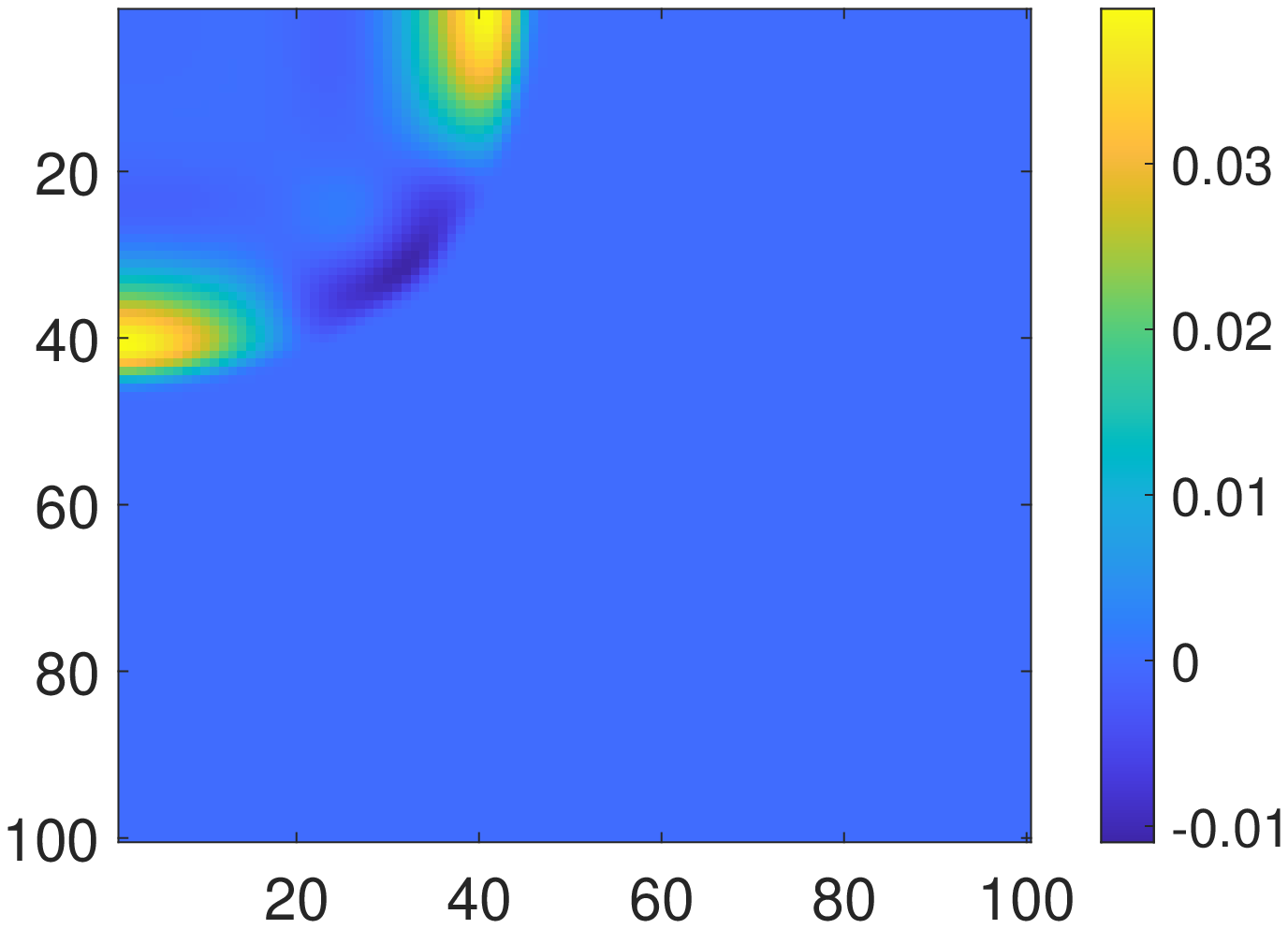}
\end{minipage}
\caption{Snapshots  of $S$ and comparison of the solutions with different data at $T=10$. Top: snapshot  of $S$ . Bottom: pointwise error of the solution.  Left: two-phase  solution. Middle-Left:
full data. Middle-Right: $\frac{1}{4}$ data.
 Right: $\frac{1}{16}$ data}
 \label{Sub3}
\end{figure}

\begin{figure}[!h]
\begin{minipage}[t]{0.2\textwidth}
         \centering
         \includegraphics[width=3cm,height=2cm]{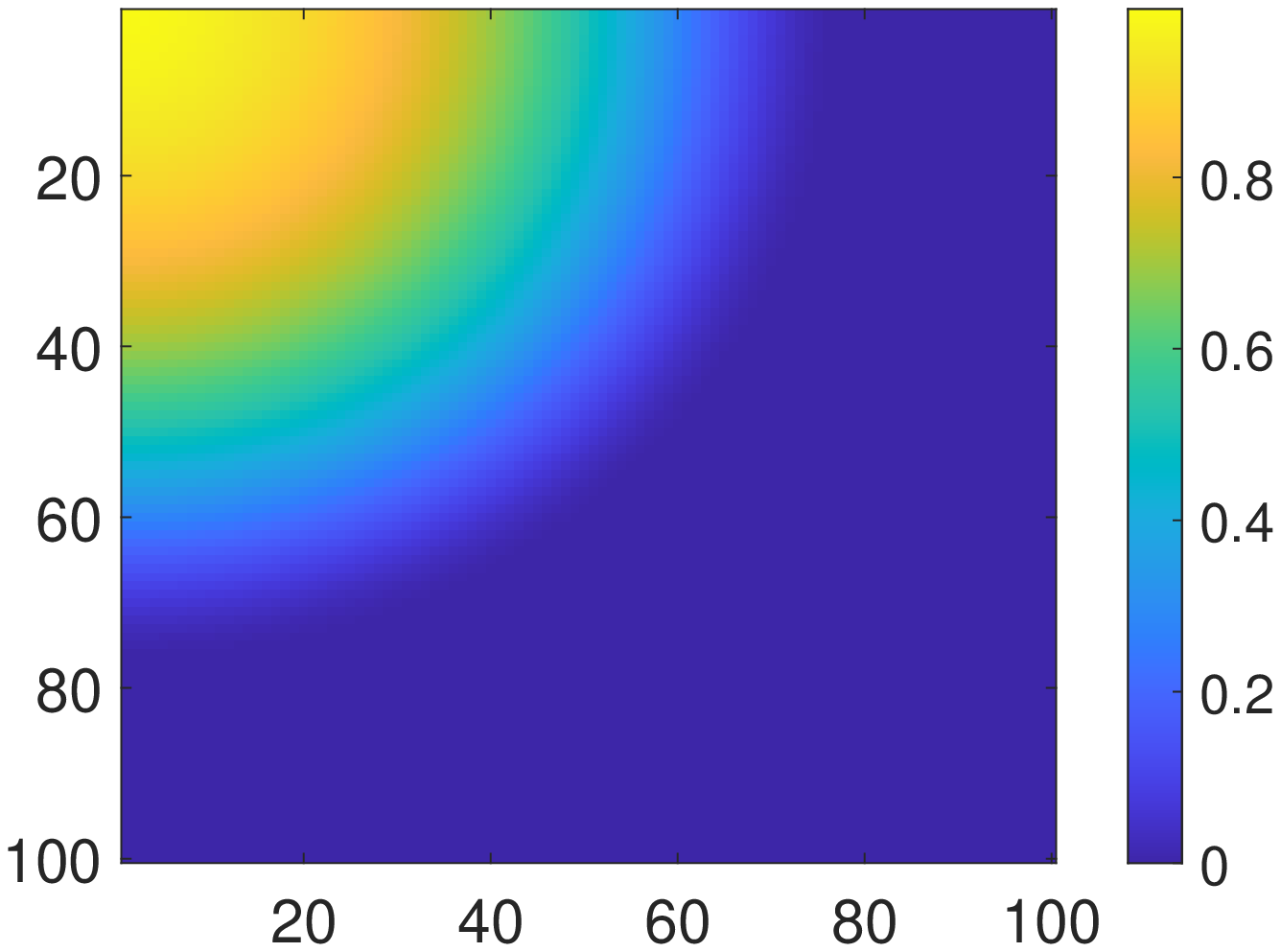}

        \includegraphics[width=3cm,height=2cm]{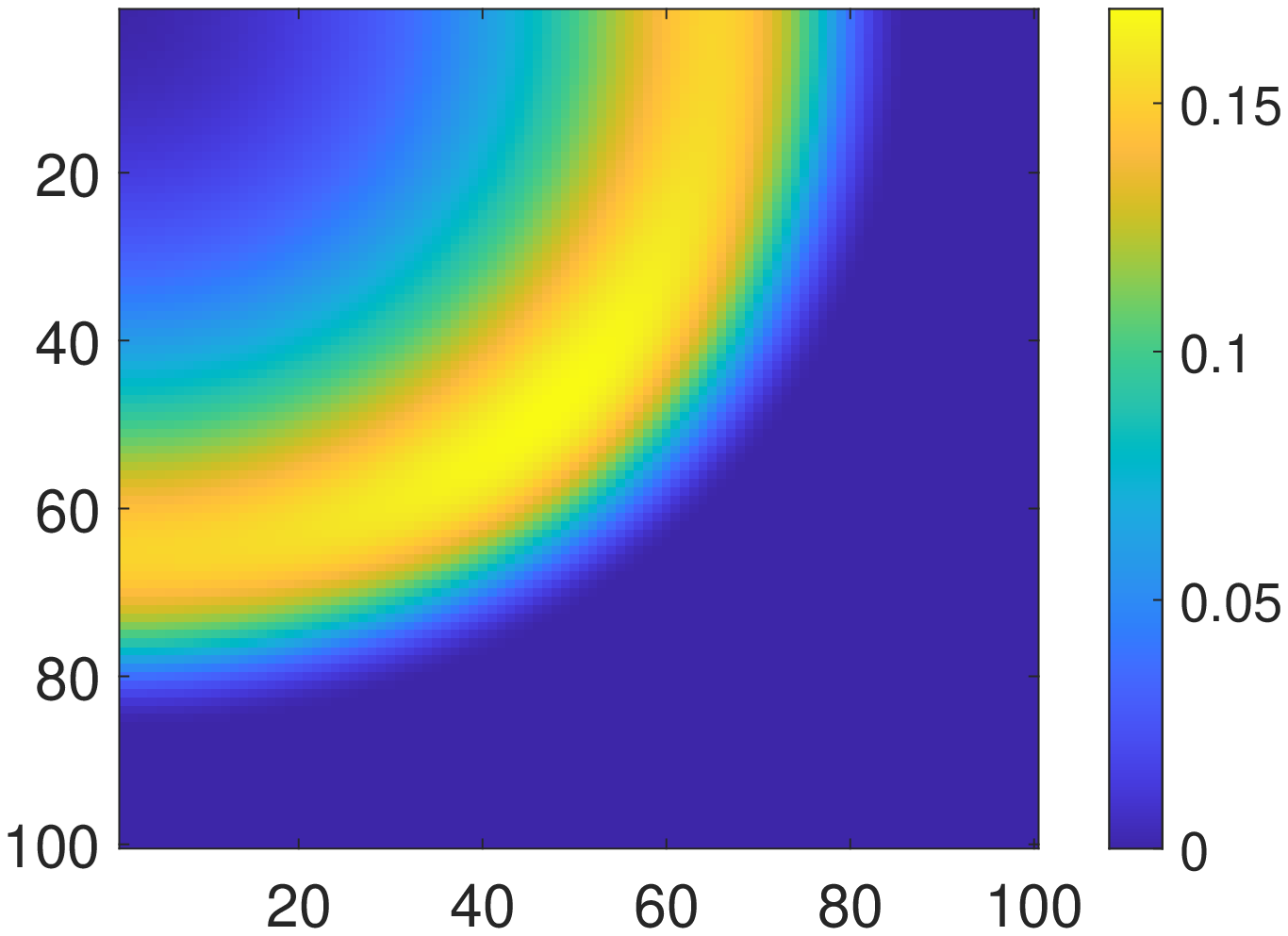}
\end{minipage}
\begin{minipage}[t]{0.2\textwidth}
         \centering
         \includegraphics[width=3cm,height=2cm]{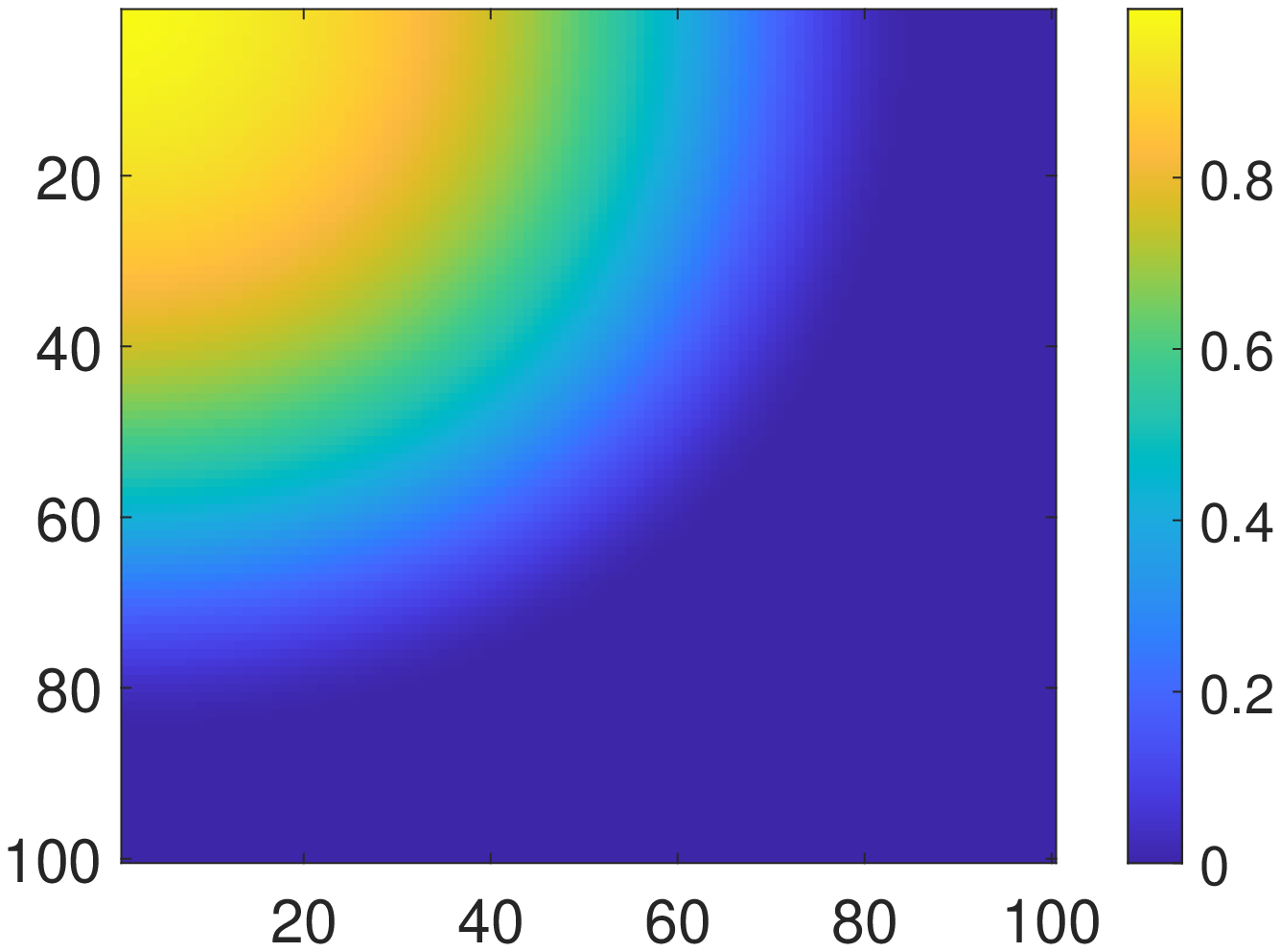}

         \includegraphics[width=3cm,height=2cm]{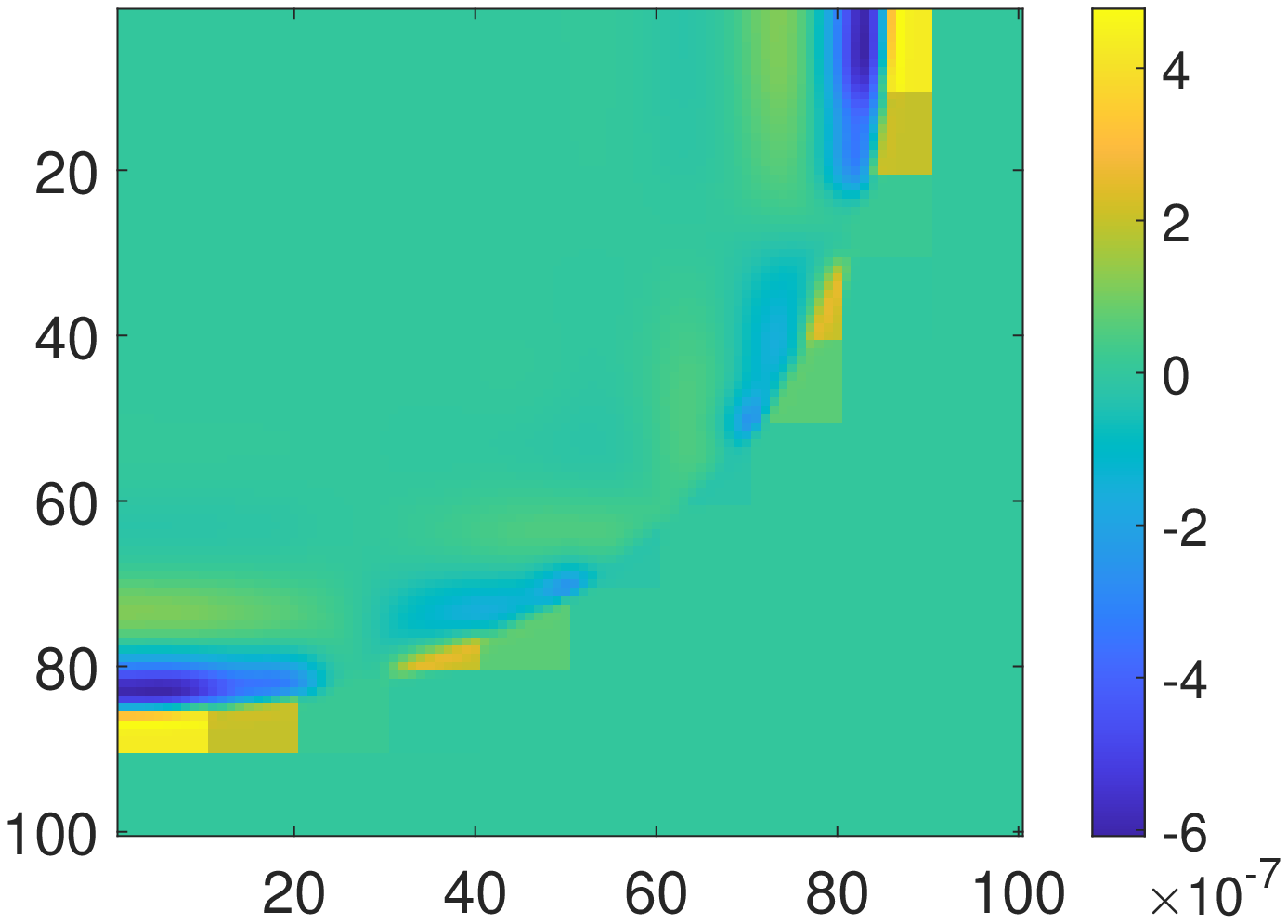}
\end{minipage}
\begin{minipage}[t]{0.2\textwidth}
         \centering
         \includegraphics[width=3cm,height=2cm]{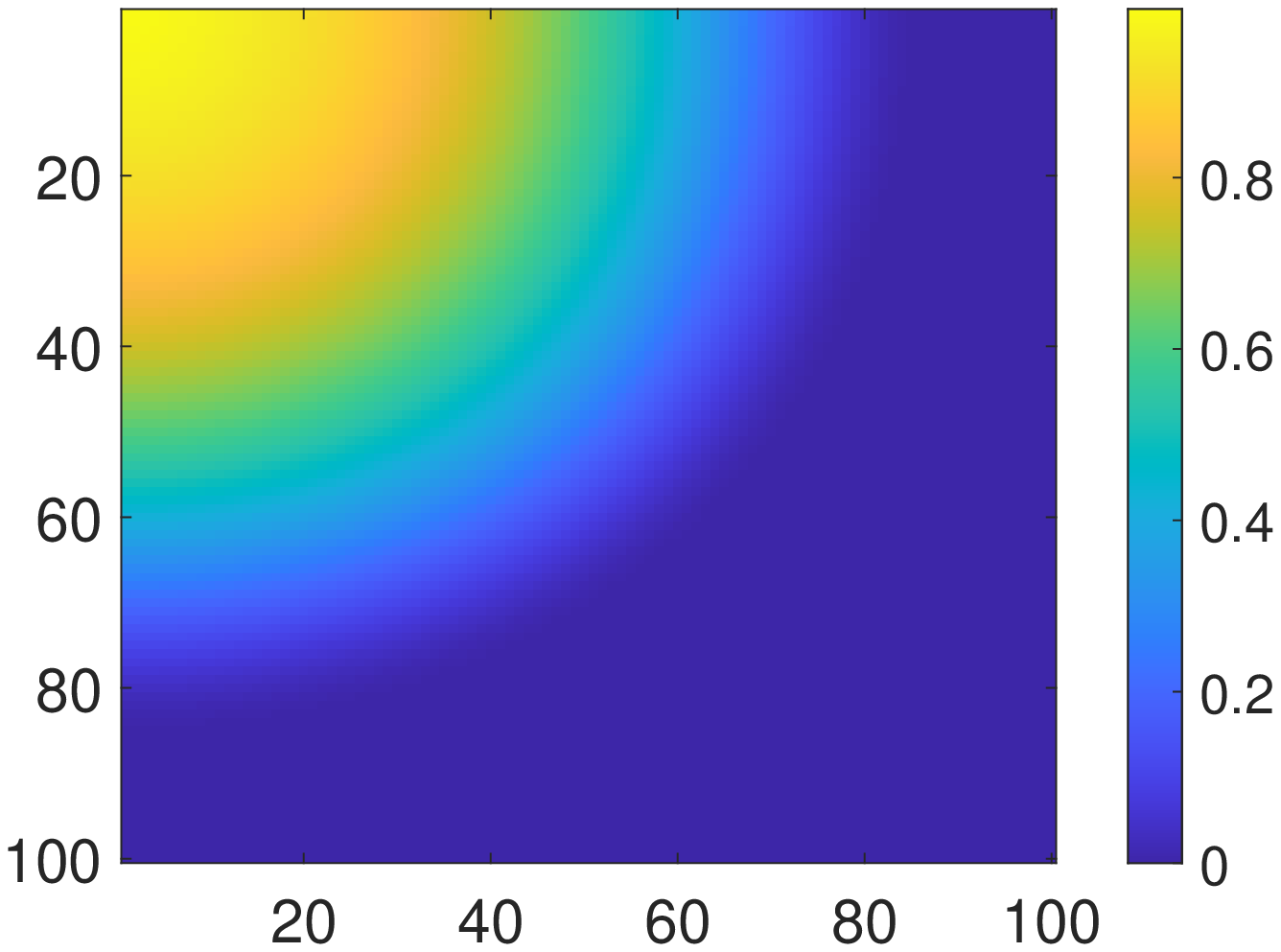}
         
         \includegraphics[width=3cm,height=2cm]{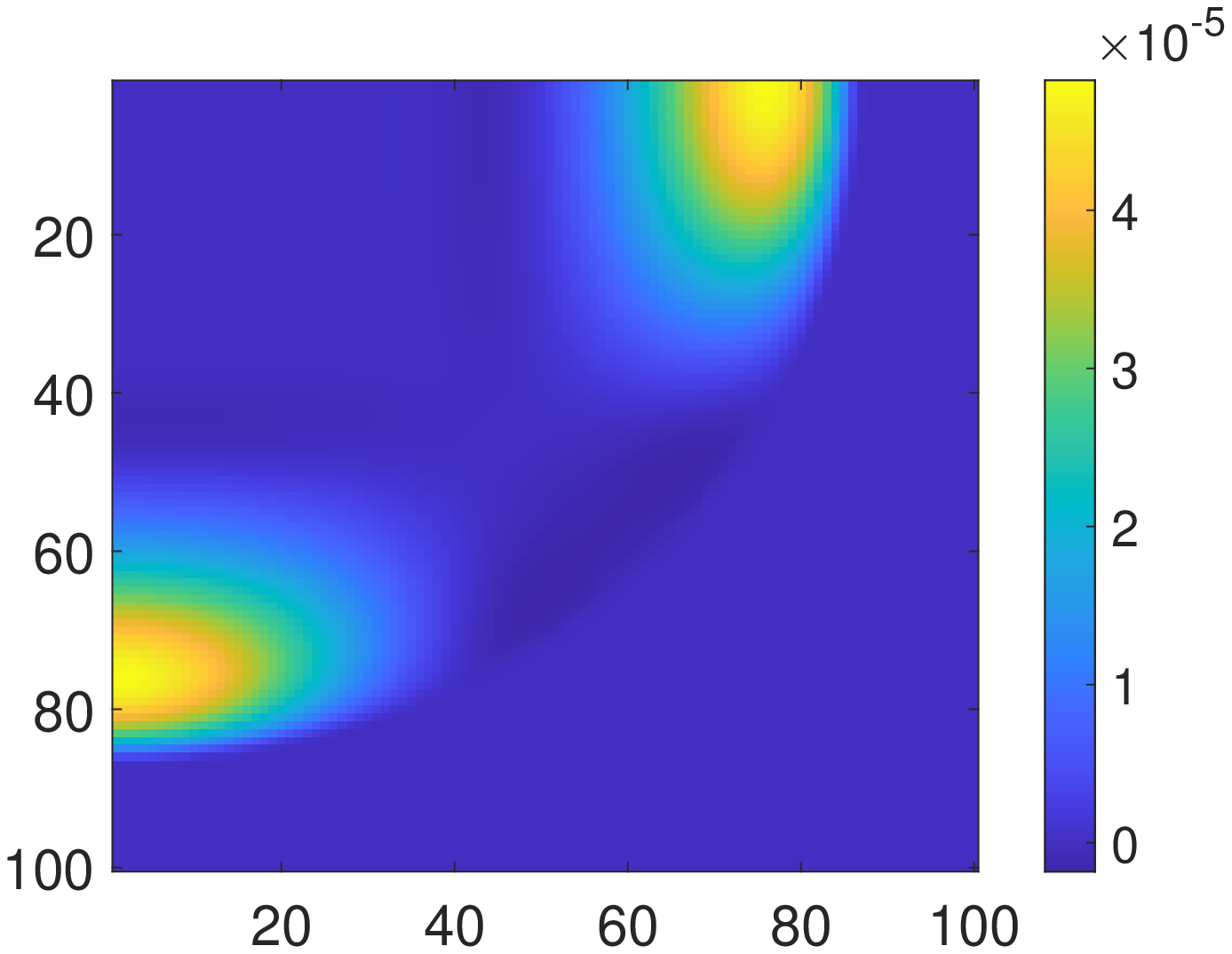}
\end{minipage}
\begin{minipage}[t]{0.2\textwidth}
         \centering
         \includegraphics[width=3cm,height=2cm]{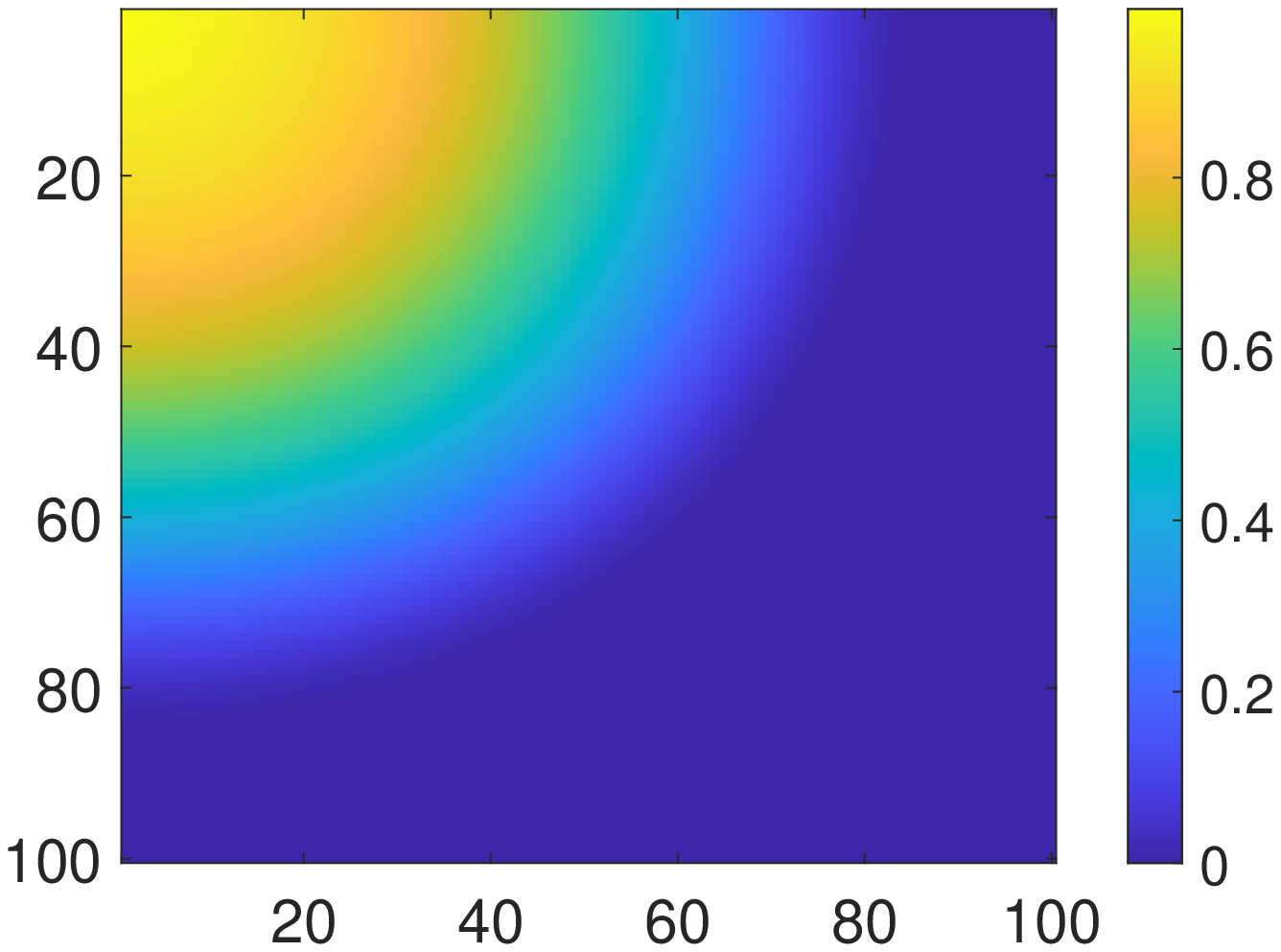}
         
         \includegraphics[width=3cm,height=2cm]{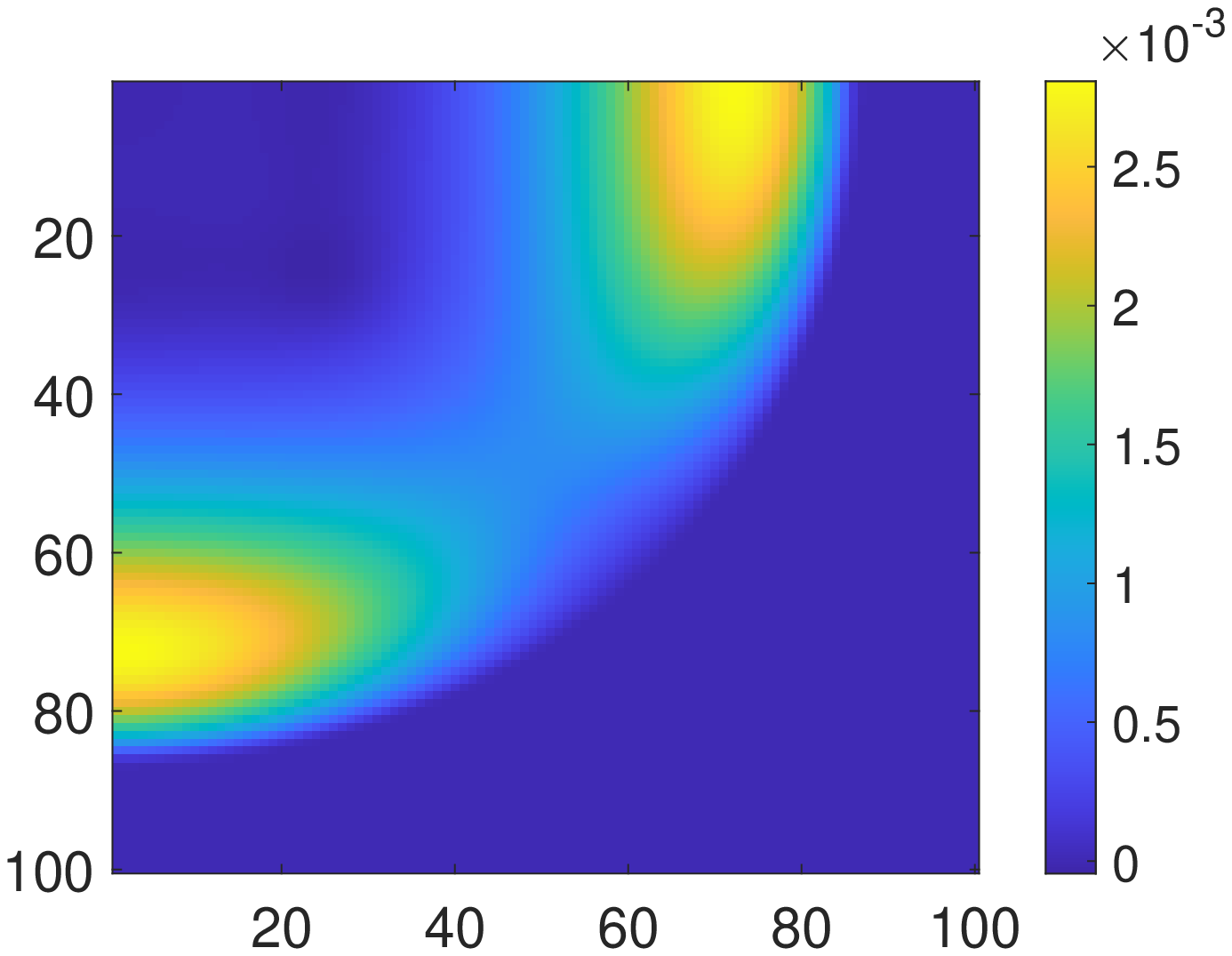}
\end{minipage}
\caption{Snapshots  of $S$ and comparison of the solutions with different data at $T=100$. Top: snapshot  of $S$ . Bottom: pointwise error of the solution. Left: two-phase  solution. Middle-Left:
full data. Middle-Right: $\frac{1}{4}$ data.
 Right: $\frac{1}{16}$ data}
 \label{Sub4}
\end{figure}
\clearpage{}

\begin{figure}[!h]
\centering

     \begin{subfigure}[b]{0.4\textwidth}
         \centering
         \includegraphics[width=3cm,height=2cm]{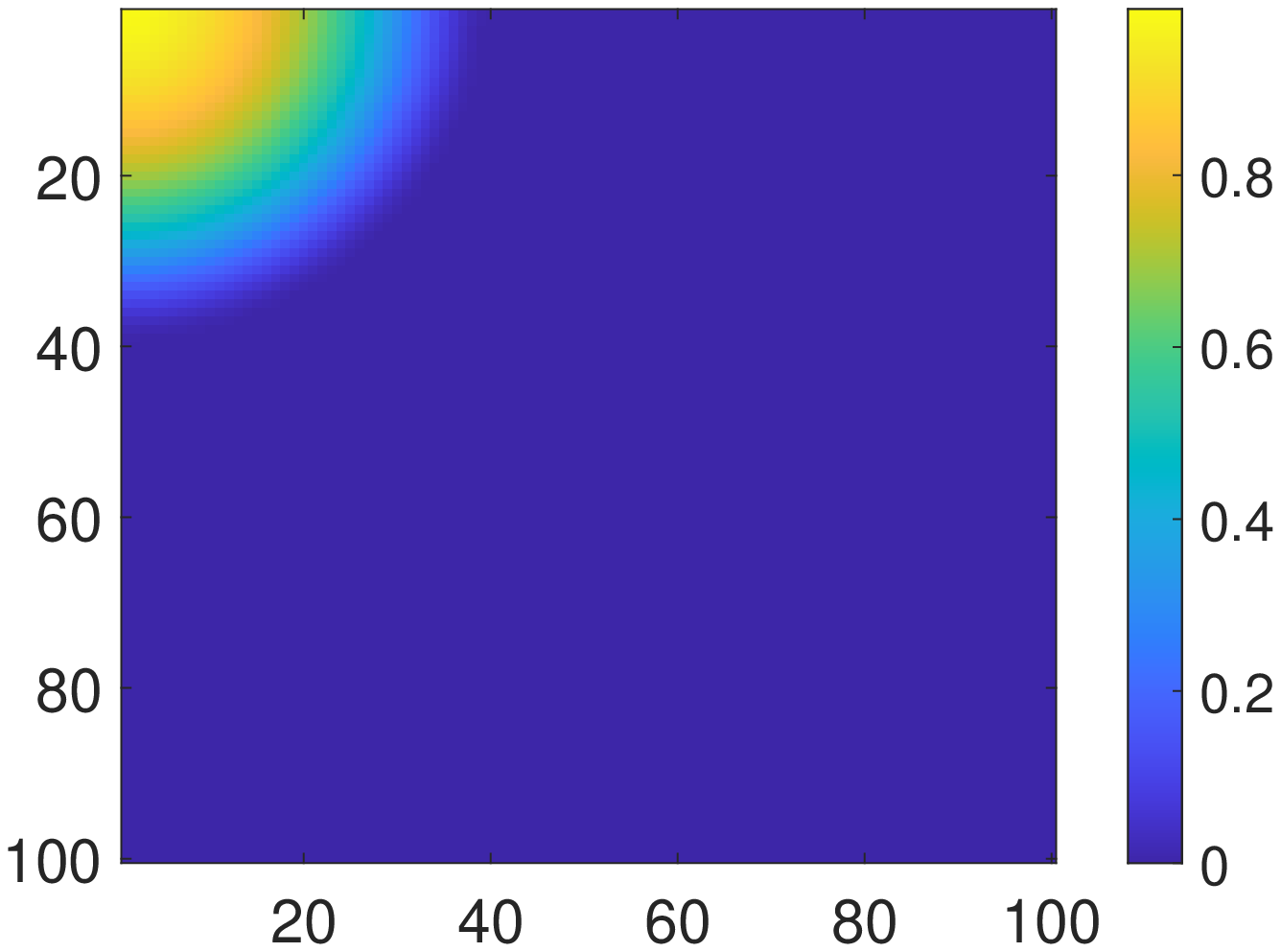}
         
         \includegraphics[width=3cm,height=2cm]{figure/data_sol_1_k2}
         \includegraphics[width=3cm,height=2cm]{figure/data_different_1_k2}
         
         \includegraphics[width=3cm,height=2cm]{figure/twophase_sol_1_k2}
         \includegraphics[width=3cm,height=2cm]{figure/twophase_different_1_k2}
         \caption{$T=0.1$}
     \end{subfigure}
     \begin{subfigure}[b]{0.4\textwidth}
         \centering
         \includegraphics[width=3cm,height=2cm]{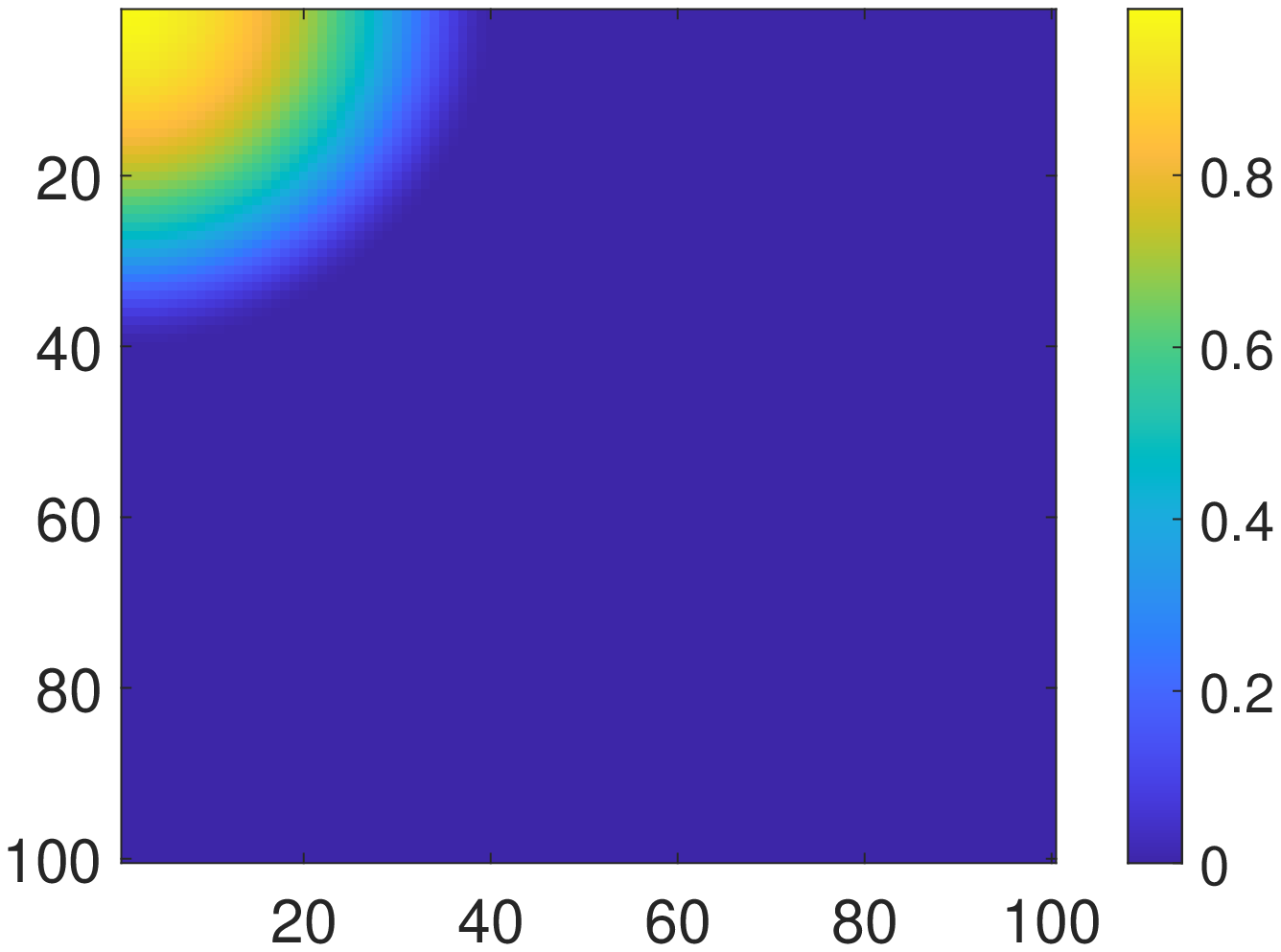}
         
         \includegraphics[width=3cm,height=2cm]{figure/data_sol_1_k20}
         \includegraphics[width=3cm,height=2cm]{figure/data_different_1_k20}
         
         \includegraphics[width=3cm,height=2cm]{figure/twophase_sol_1_k20}
         \includegraphics[width=3cm,height=2cm]{figure/twophase_different_1_k20}
         \caption{$T=1$}
     \end{subfigure}    
     \begin{subfigure}[b]{0.4\textwidth}
         \centering
         \includegraphics[width=3cm,height=2cm]{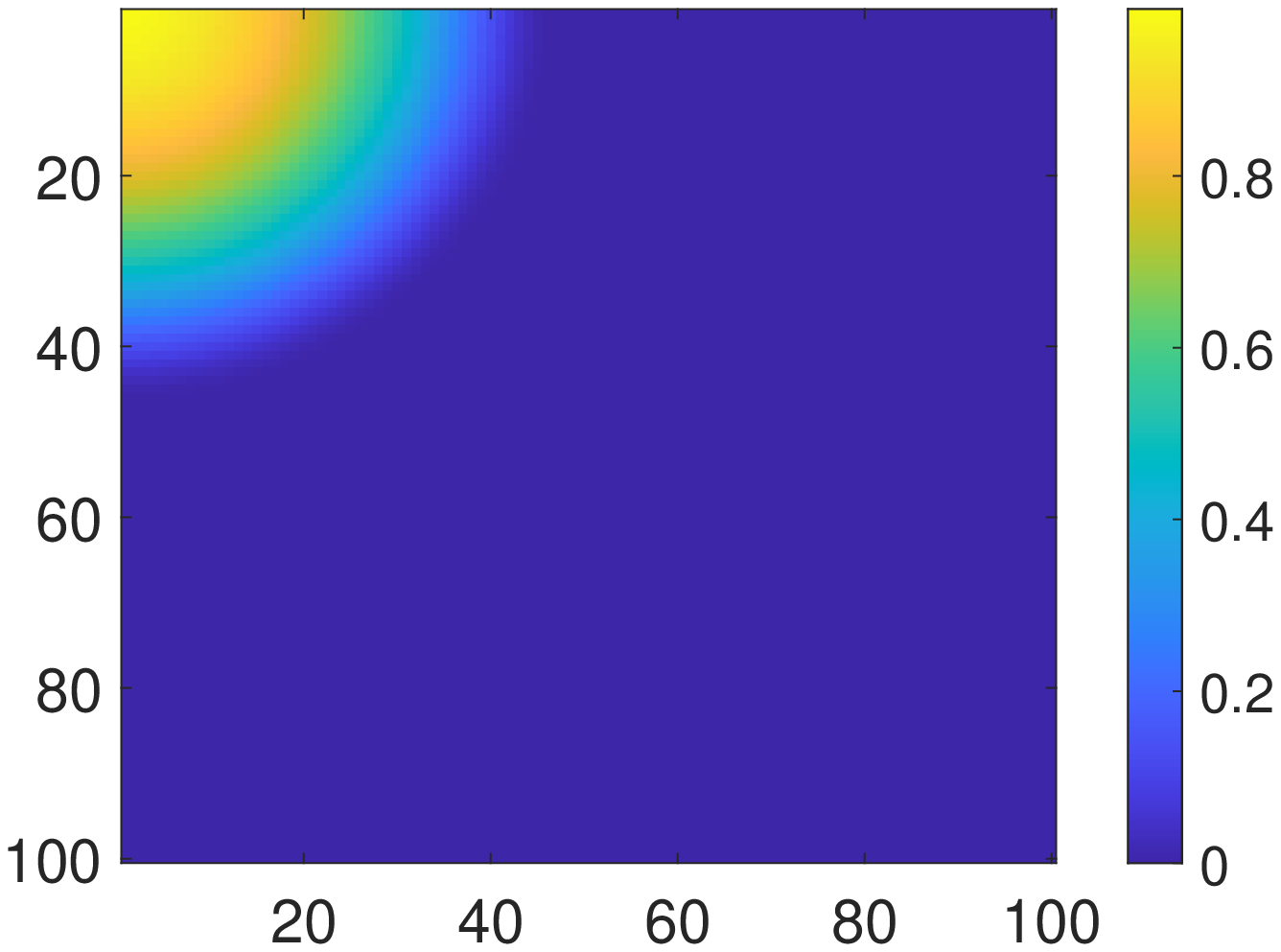}
         
         \includegraphics[width=3cm,height=2cm]{figure/data_sol_1_k200}
         \includegraphics[width=3cm,height=2cm]{figure/data_different_1_k200}
         
         \includegraphics[width=3cm,height=2cm]{figure/twophase_sol_1_k200}
         \includegraphics[width=3cm,height=2cm]{figure/twophase_different_1_k200}
         \caption{$T=10$}
     \end{subfigure}
     \begin{subfigure}[b]{0.4\textwidth}
         \centering
         \includegraphics[width=3cm,height=2cm]{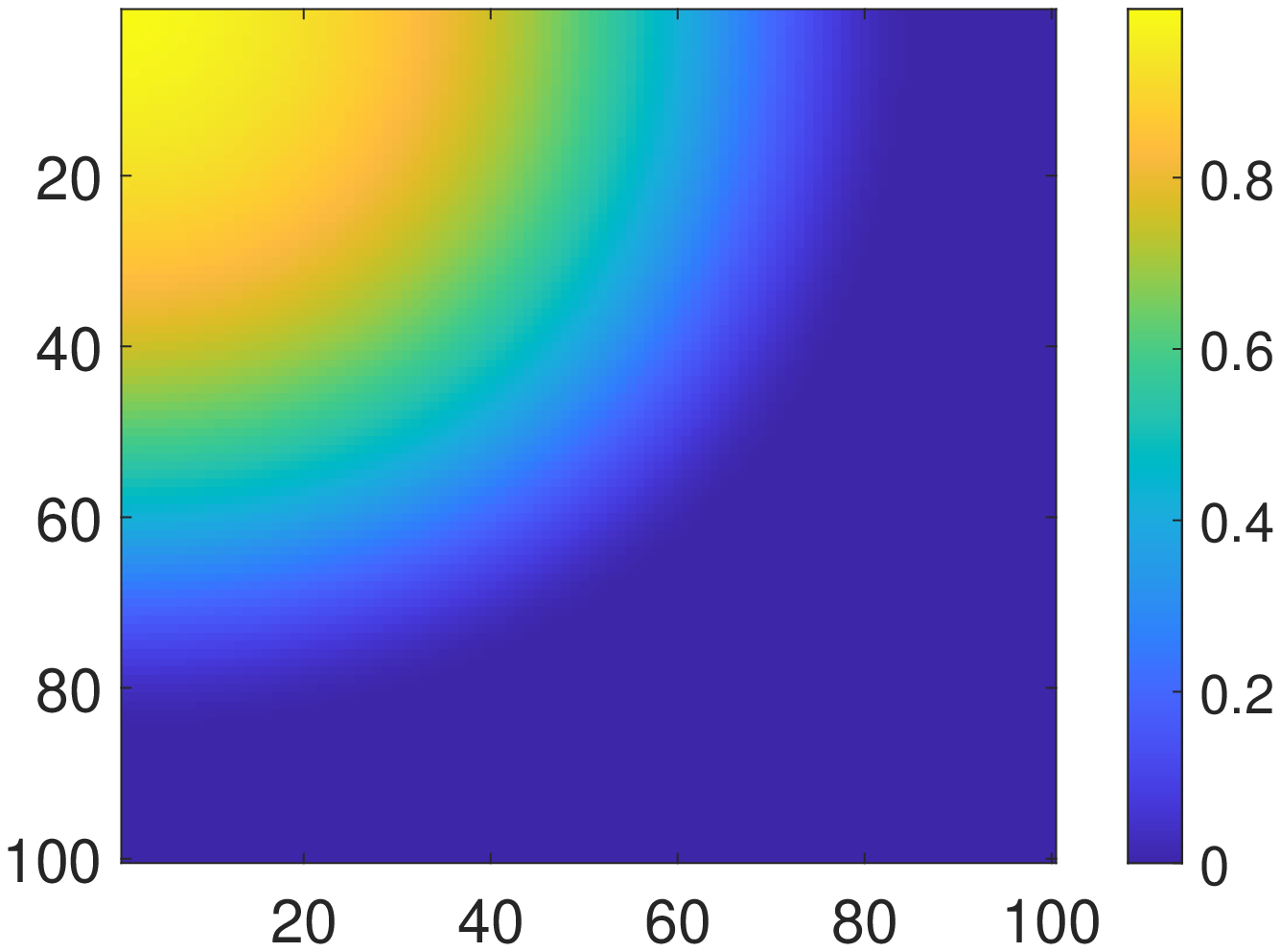}
         
         \includegraphics[width=3cm,height=2cm]{figure/data_sol_1_k2000}
         \includegraphics[width=3cm,height=2cm]{figure/data_different_1_k2000}
         
         \includegraphics[width=3cm,height=2cm]{figure/twophase_sol_1_k2000}
         \includegraphics[width=3cm,height=2cm]{figure/twophase_different_1_k2000}
         \caption{$T=100$}
     \end{subfigure}

\caption{Snapshots  of the solutions $S$ . Top: Reference solution. Middle-Left:
DA (approximate) solution. Middle-Right: difference between the approximate and reference solution. Bottom-Left: two-phase  solution starting from zero initial value. Bottom-Right:
difference between the two-phase  solutions with different initial values.}
\label{fig:error_case1-1-3}
\end{figure}

\subsection{Computational study \RN{2}}\label{Comp-Ex2}
 We present our second numerical test, with its source terms coinciding with that of the first example. On the other hand, we have a medium that have a totally different permeability profile as shown in figure \ref{fig:case_2_kappa}. 

\begin{figure}[!h]
\centering\includegraphics[width=5.5cm,height=3.5cm]{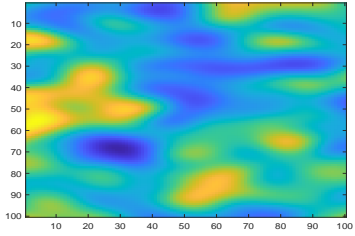}
\caption{Absolute permeability for case 2.}
\label{fig:case_2_kappa} 
\end{figure}
The same initialization process  as in the  first example is considered here, which provides us with the initial condition of the exact solution as $S(0) = \tilde{S}(25)$.
With this new medium profile, we compute an approximate solution without any prior knowledge of our initial value, and apply our data assimilation algorithm with $\mu = 200$ as in Example 1
In this experiment, we carried out two tests.  The first one is with data taken on full domain $\Omega$, and the other test is carried out with data only taken on one of the following sub-domains at the left half of the domain:
$$\Omega_3 =  [0,50]\times[0,100].$$
The convergence history of this test is given in Figure \ref{fig:error_case2},
and the snapshots of the solutions $S$ are now shown in Figures \ref{fig:error_case1-1-4}.  We  again observe an exponential decay of the residual error in Figure \ref{fig:error_case2}, with a clear linear fit in the log scale,  This has numerically validated the theoretical result that we have proved in Theorem \ref{main_theorem}, and the effectiveness of our proposed data assimilation algorithm.  We then  continue  by testing the effect of the size of the sub-domain. In the case of subdomain $\Omega_3$, we see from the
snapshot plots in Figures \ref{Sub1_2}, \ref{Sub2_2}, \ref{Sub3_2}, and  \ref{Sub4_2} that the main spatial features over the full domain  $\Omega$ are nevertheless captured as time evolves.

%
\begin{figure}[!h]
\centering\includegraphics[width=5.5cm,height=3.5cm]{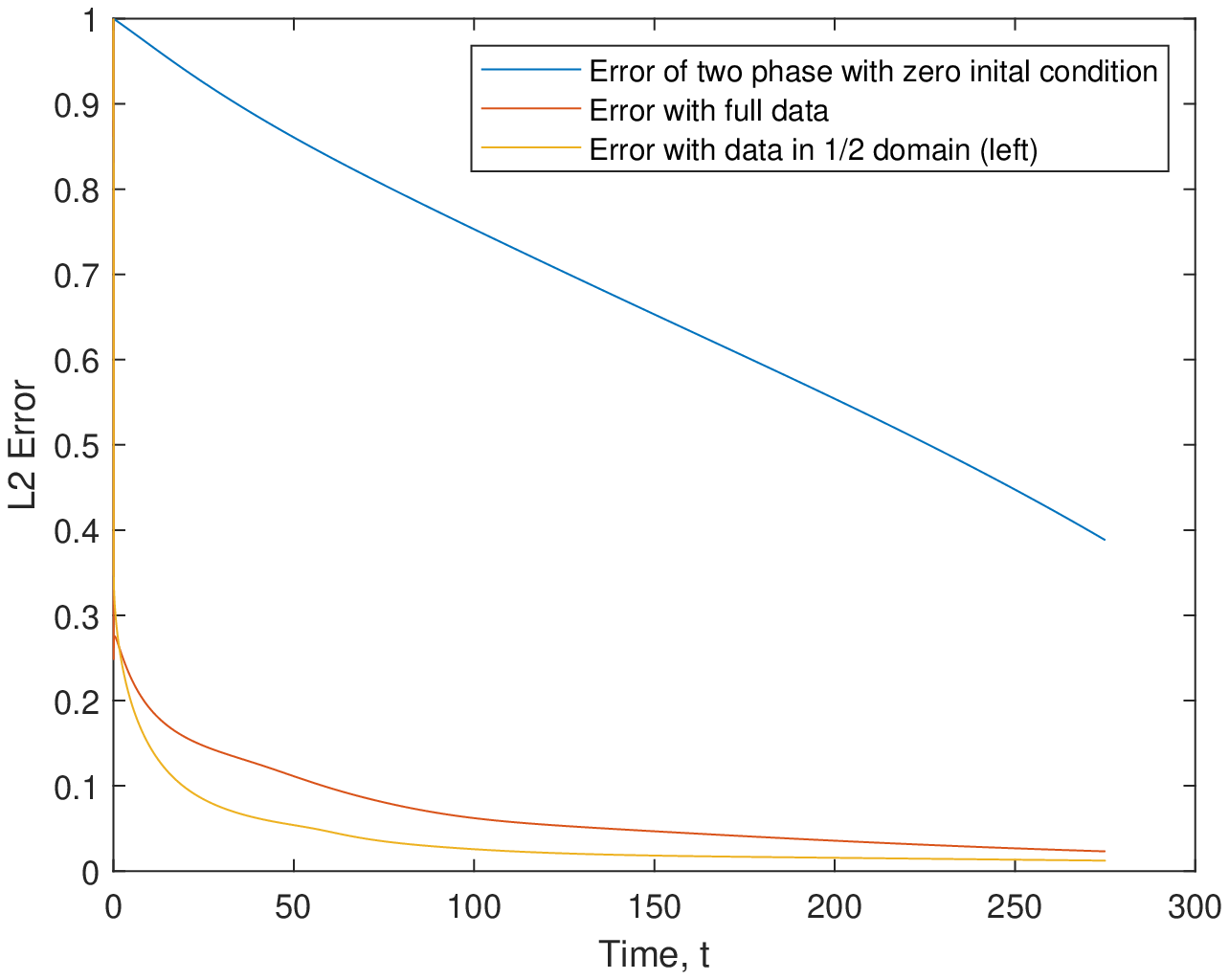} \includegraphics[width=5.5cm,height=3.5cm]{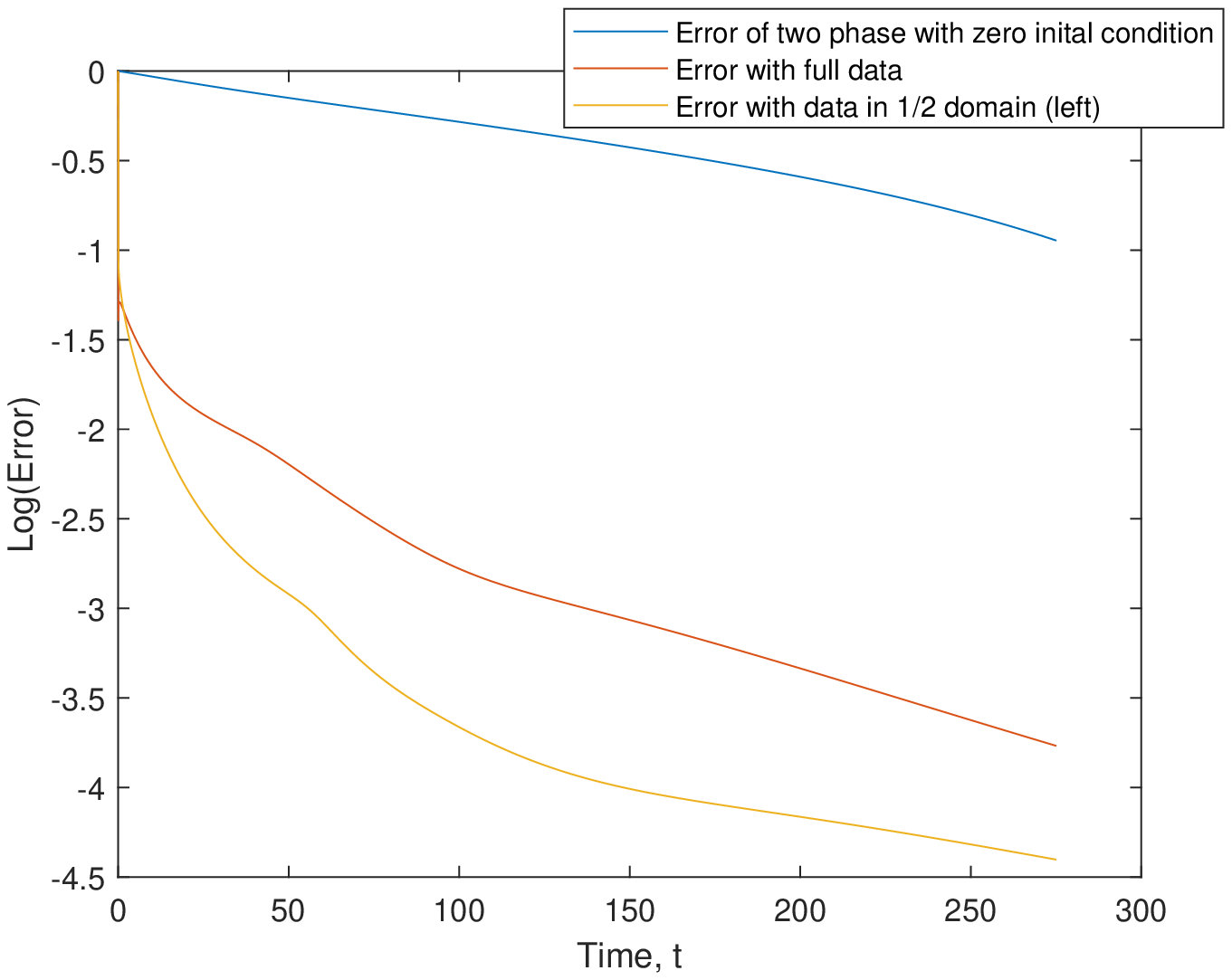}

\caption{Error comparison $\|S(t) - \tilde{S}(t)\|_{L^2}$. Left: $L_2$ error. Right: $L_2$ error in log-scale.}
\label{fig:error_case2}
\end{figure}

\begin{figure}[!h]
\begin{minipage}[t]{0.2\textwidth}
         \centering
         \includegraphics[width=3cm,height=2cm]{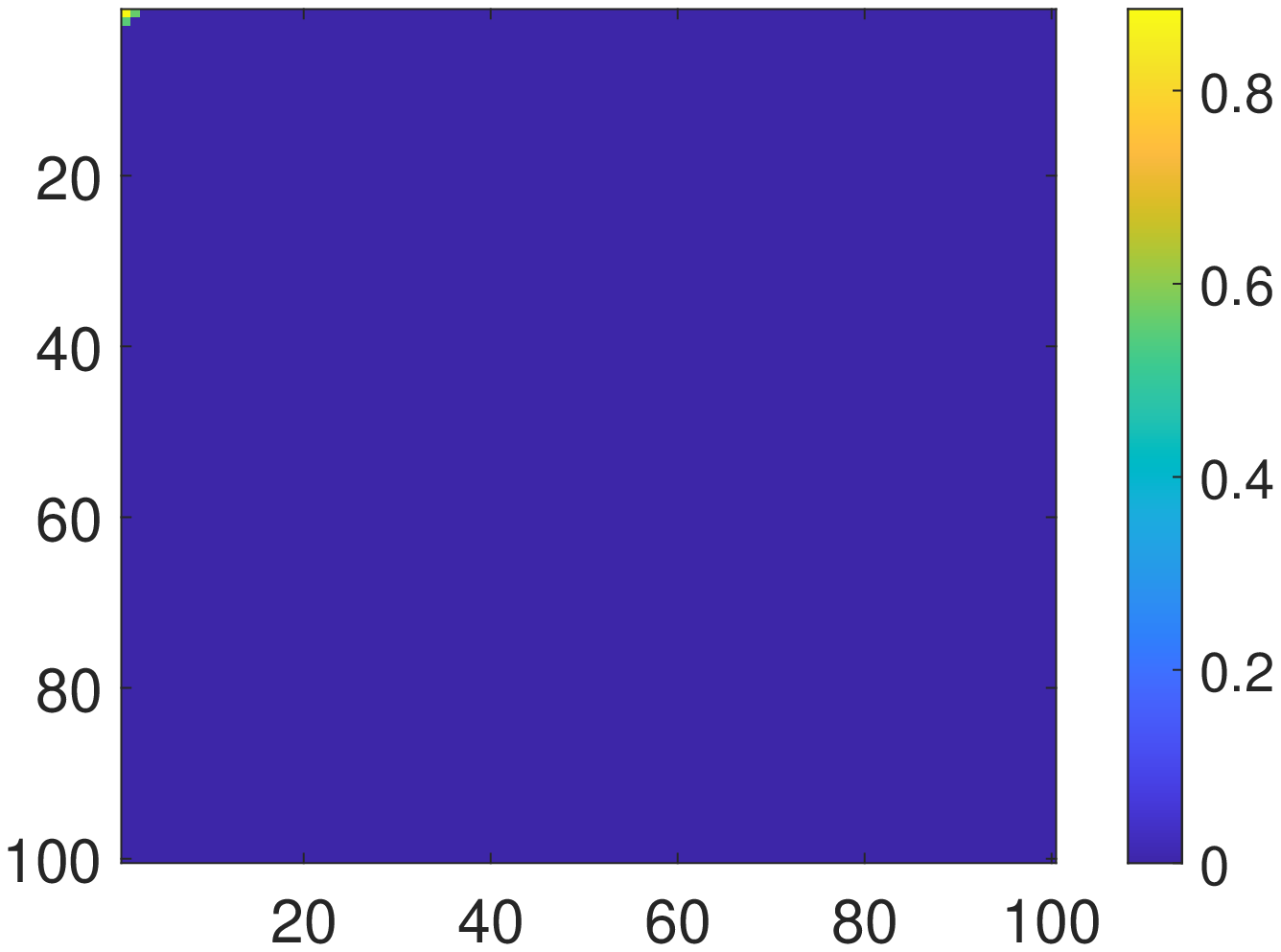}

        \includegraphics[width=3cm,height=2cm]{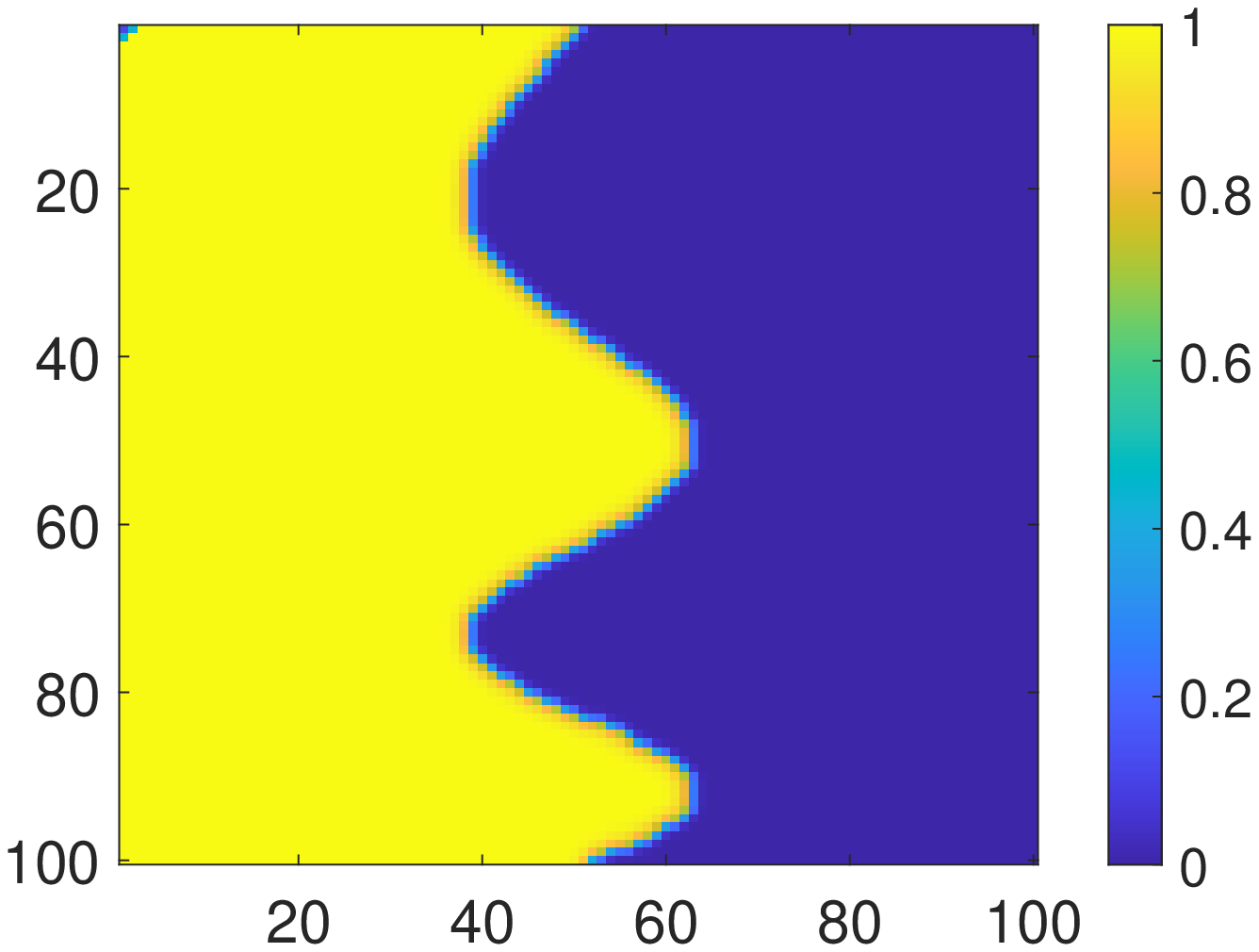}
\end{minipage}
\begin{minipage}[t]{0.2\textwidth}
         \centering
         \includegraphics[width=3cm,height=2cm]{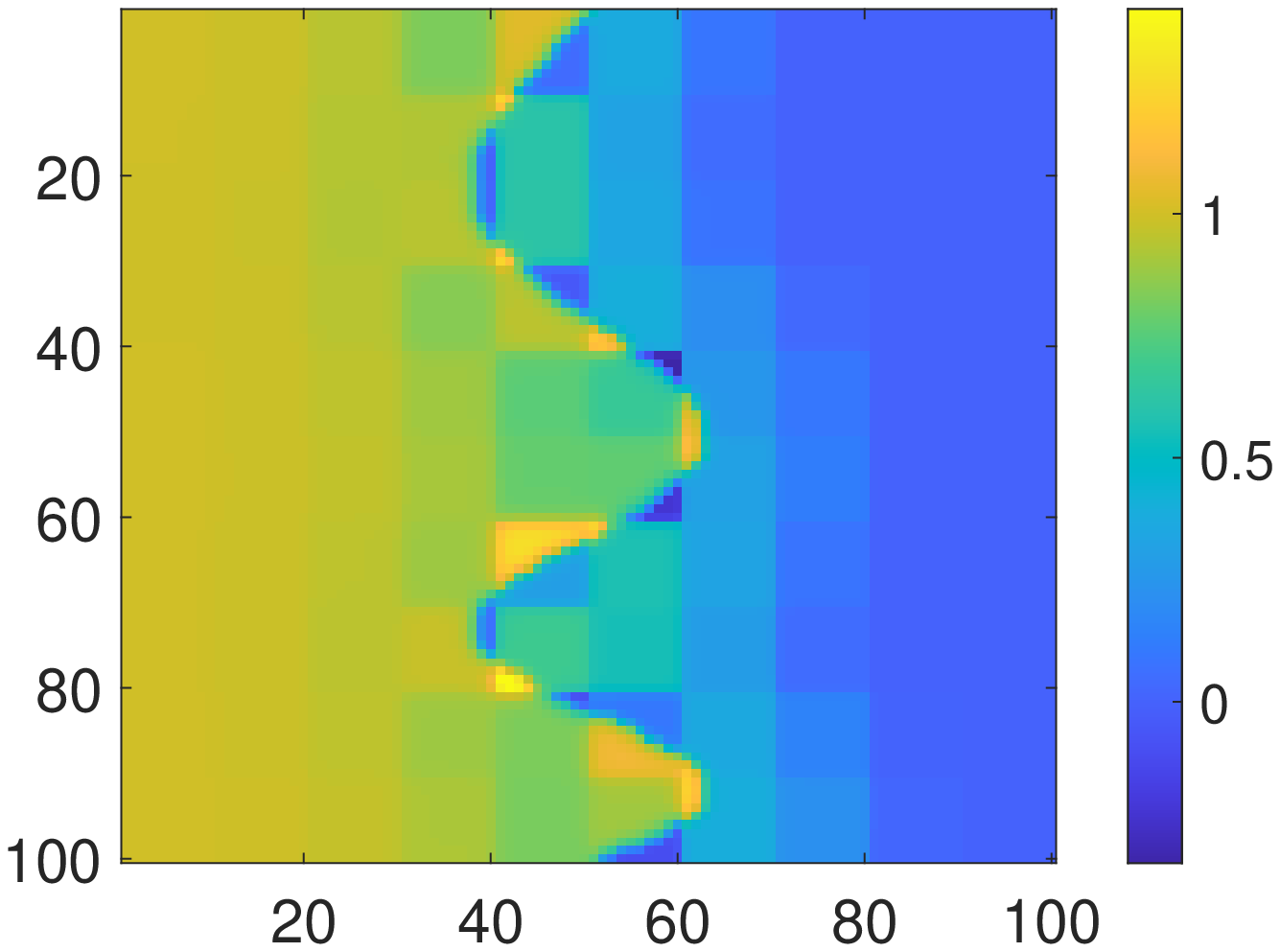}

         \includegraphics[width=3cm,height=2cm]{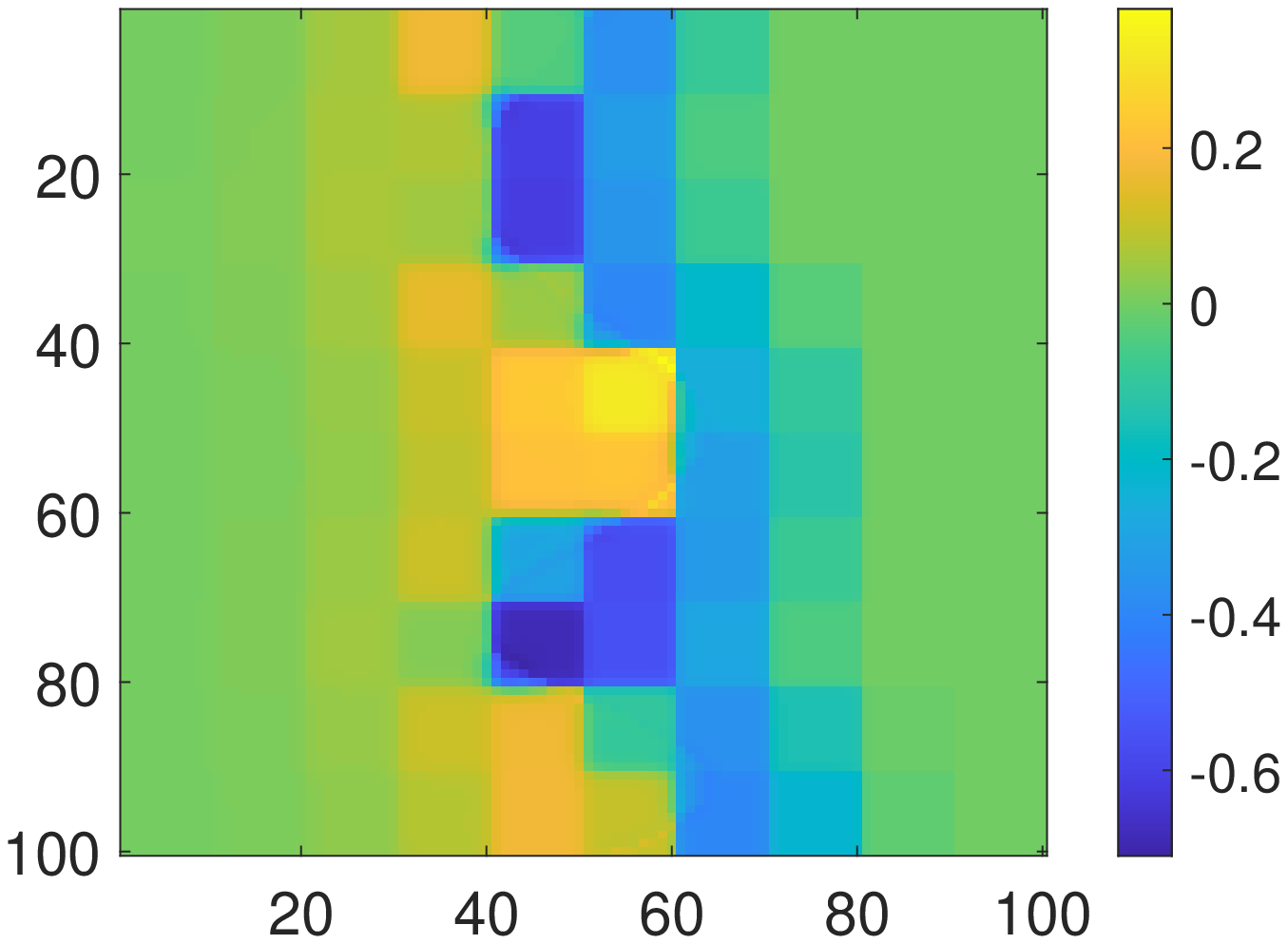}
\end{minipage}
\begin{minipage}[t]{0.2\textwidth}
         \centering
         \includegraphics[width=3cm,height=2cm]{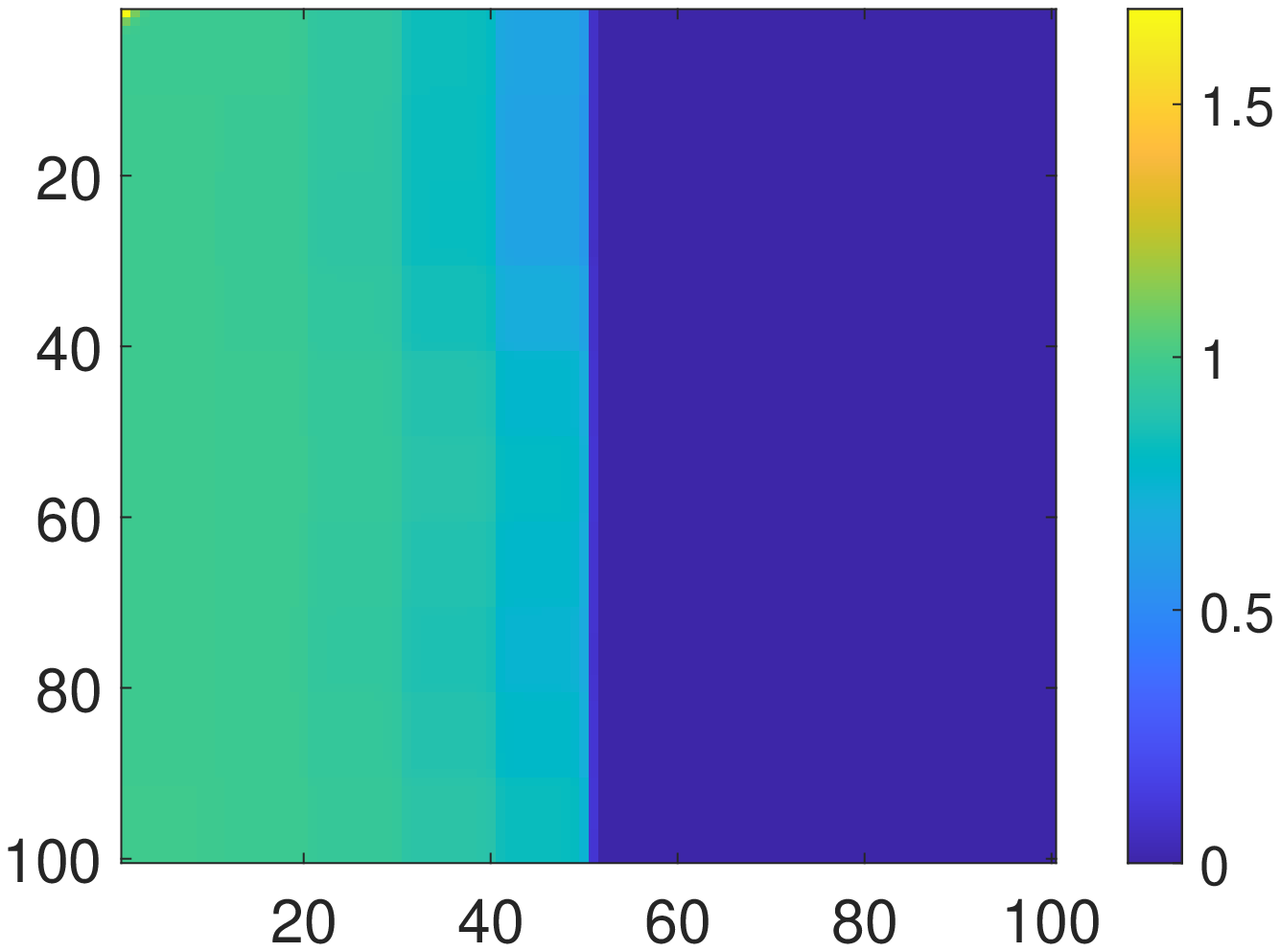}
         
         \includegraphics[width=3cm,height=2cm]{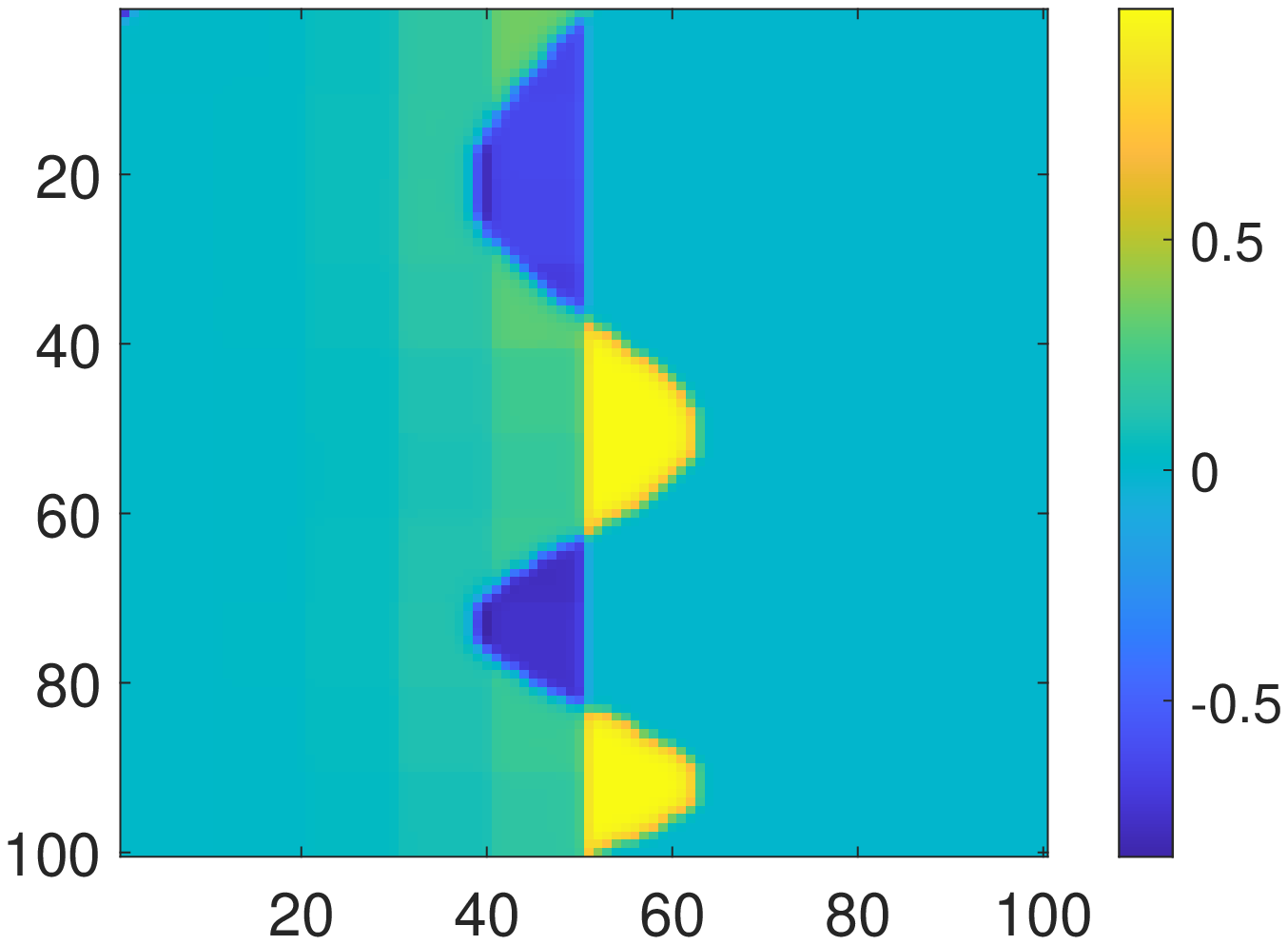}
\end{minipage}
\caption{Snapshots  of $S$ and comparison of the solutions with different data at $T=0.1$. Top: snapshot  of $S$ . Bottom: pointwise error of the solution. Left: two-phase  solution. Middle:
full data. Right: $\frac{1}{2}$ data}
 \label{Sub1_2}
\end{figure}

\begin{figure}[!h]
\begin{minipage}[t]{0.2\textwidth}
         \centering
         \includegraphics[width=3cm,height=2cm]{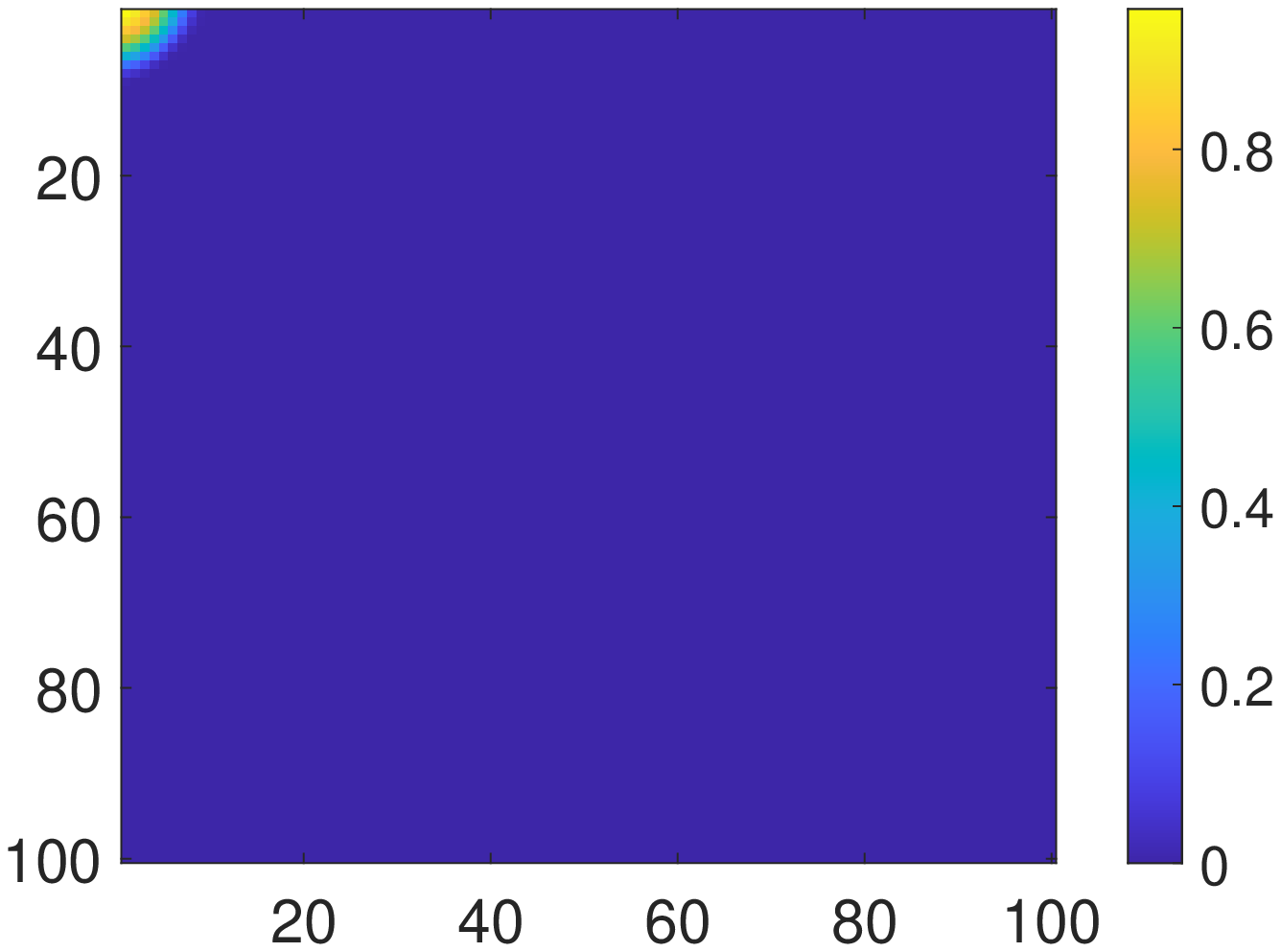}

        \includegraphics[width=3cm,height=2cm]{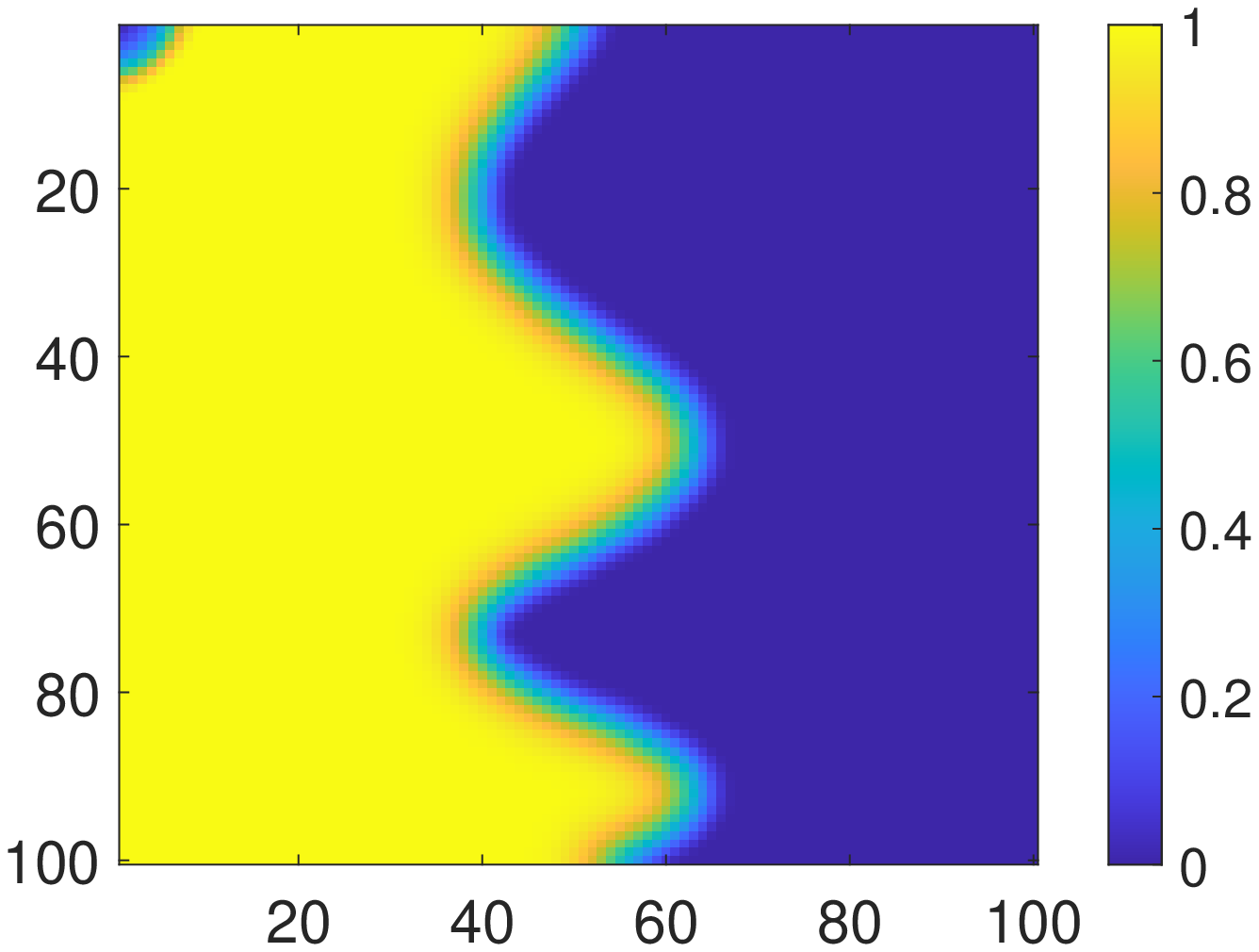}
\end{minipage}
\begin{minipage}[t]{0.2\textwidth}
         \centering
         \includegraphics[width=3cm,height=2cm]{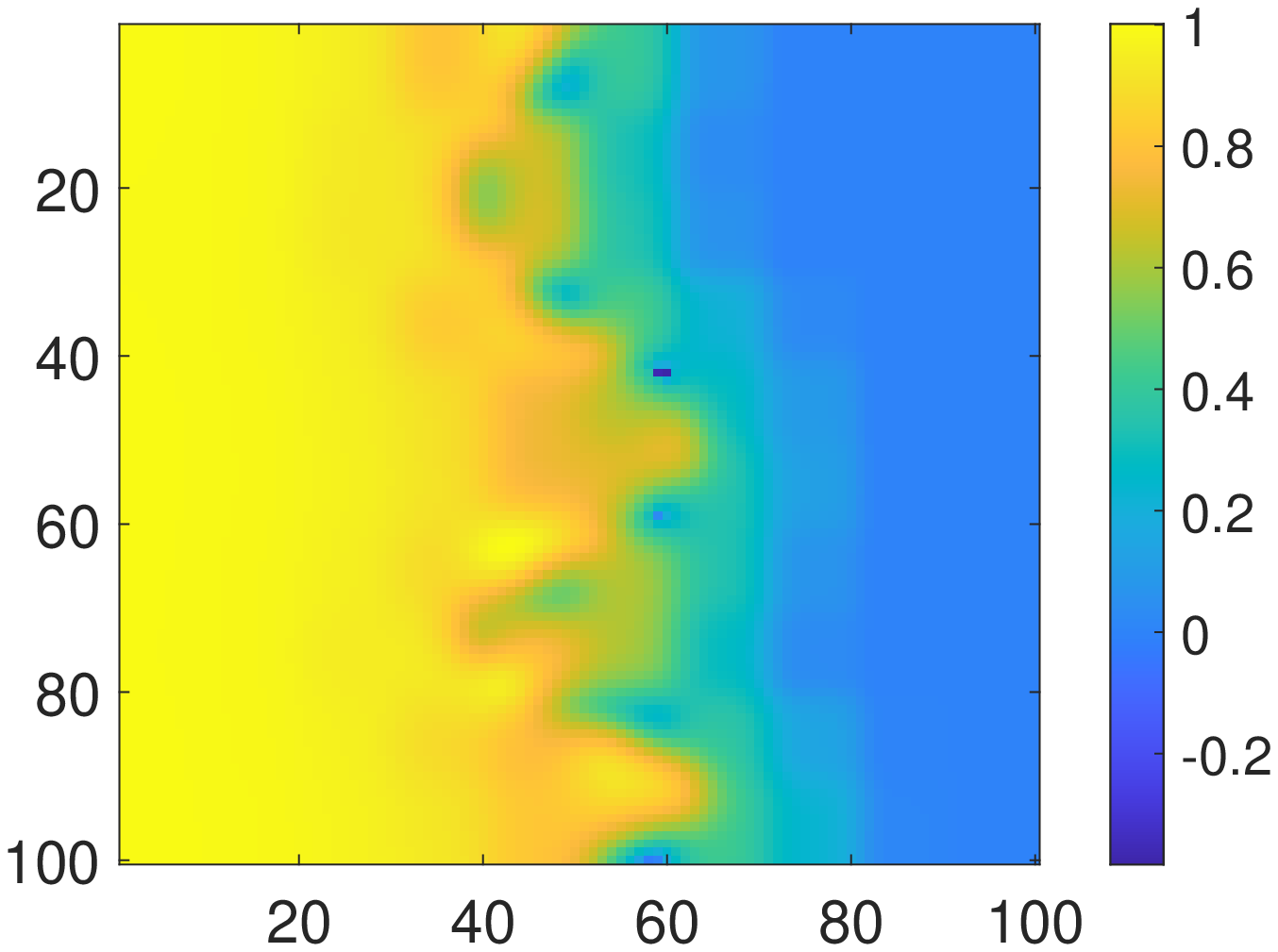}

         \includegraphics[width=3cm,height=2cm]{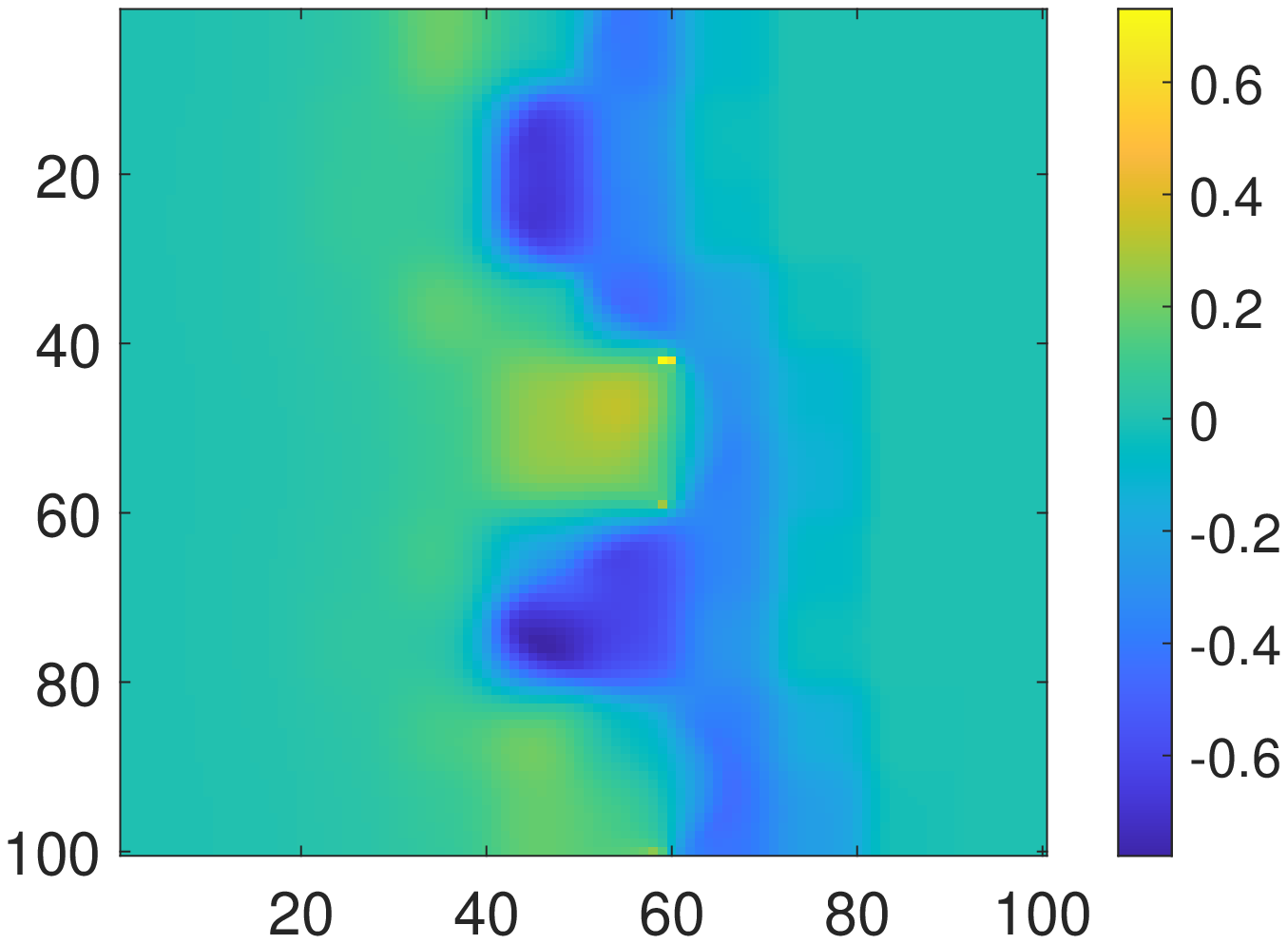}
\end{minipage}
\begin{minipage}[t]{0.2\textwidth}
         \centering
         \includegraphics[width=3cm,height=2cm]{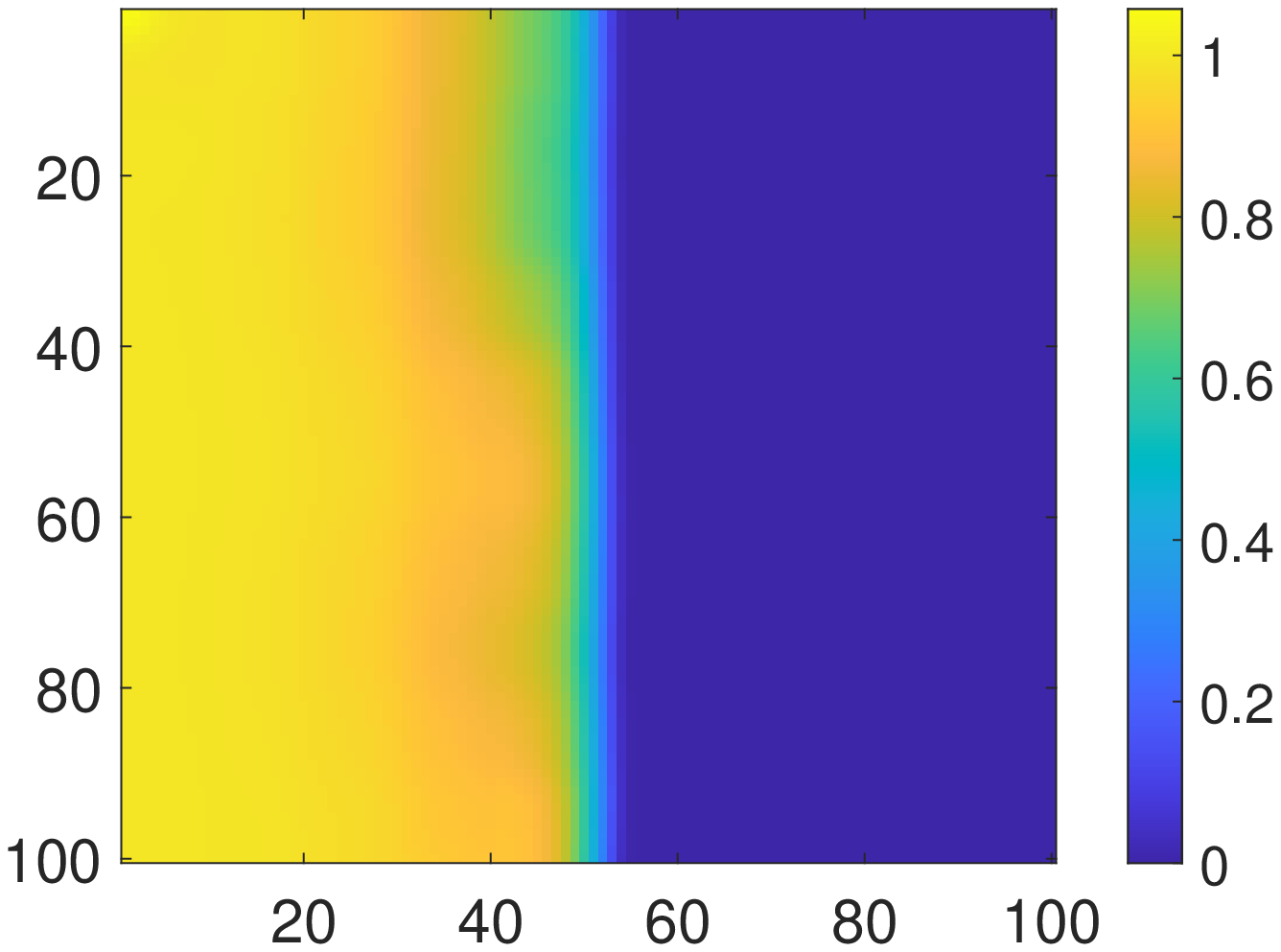}
         
         \includegraphics[width=3cm,height=2cm]{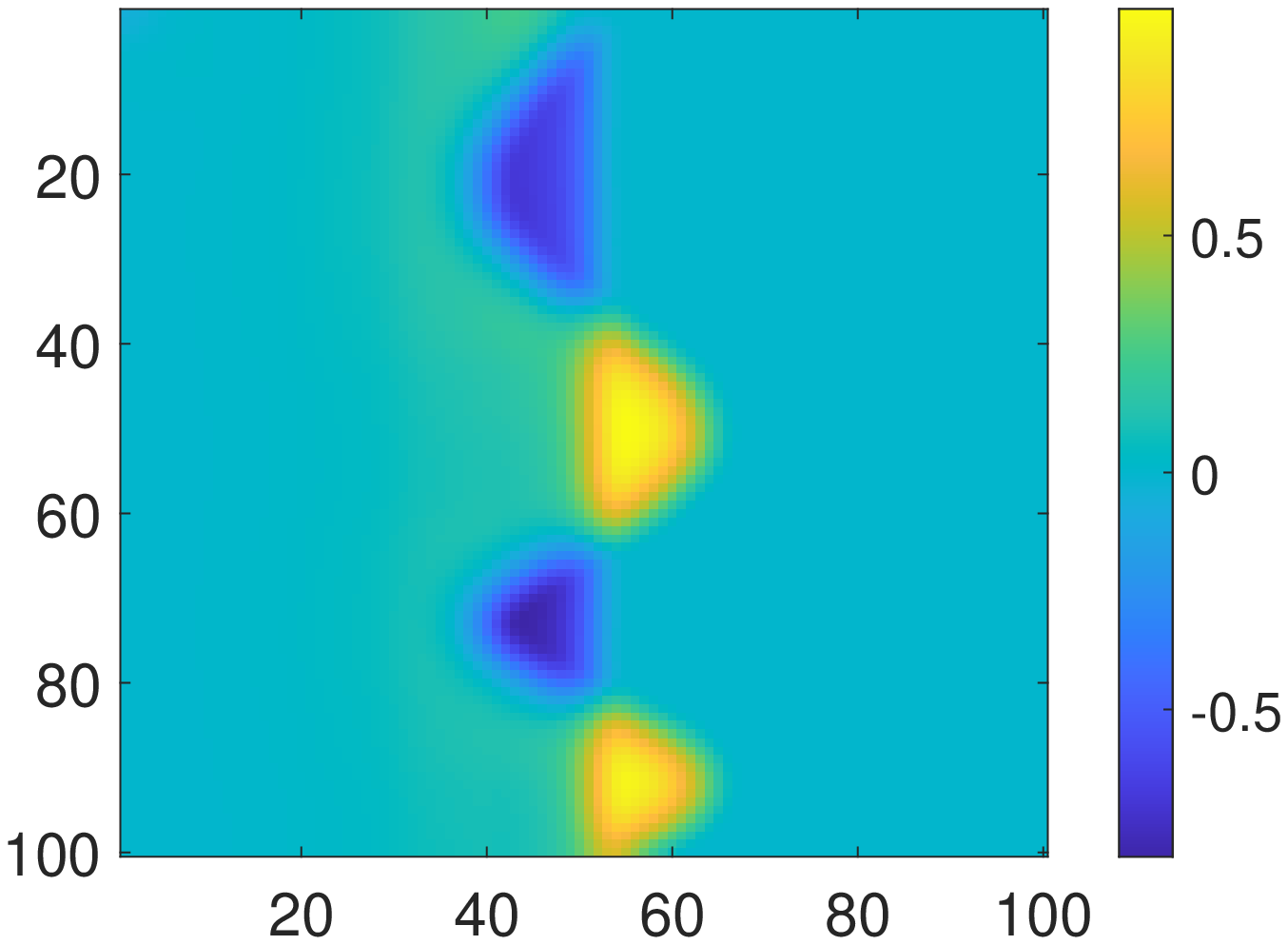}
\end{minipage}
\caption{Snapshots  of $S$ and comparison of the solutions with different data at $T=1$. Top: snapshot  of $S$ . Bottom: pointwise error of the solution. Left: two-phase  solution. Middle:
full data. Right: $\frac{1}{2}$ data}
 \label{Sub2_2}
\end{figure}

\begin{figure}[!h]
\begin{minipage}[t]{0.2\textwidth}
         \centering
         \includegraphics[width=3cm,height=2cm]{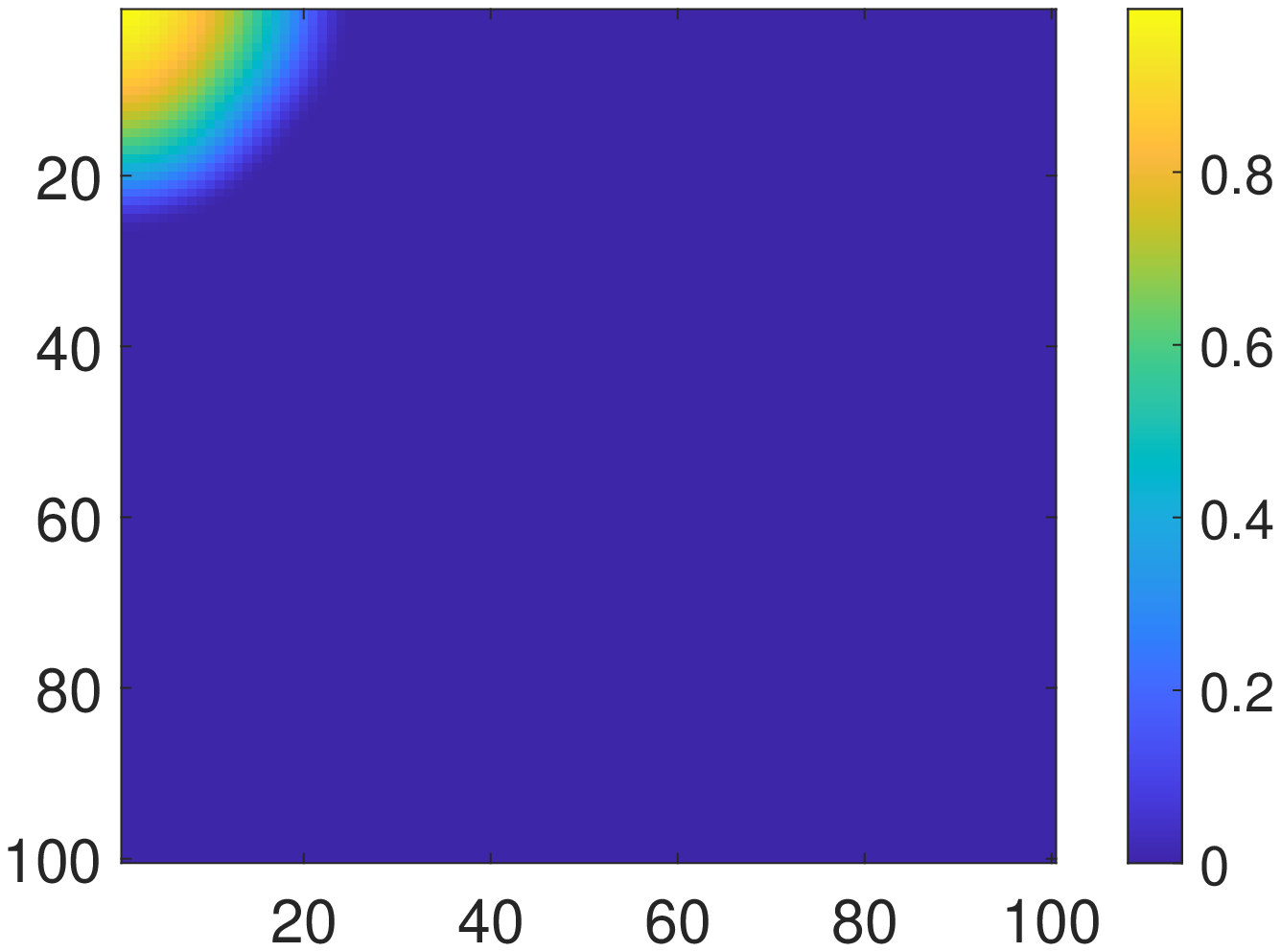}

        \includegraphics[width=3cm,height=2cm]{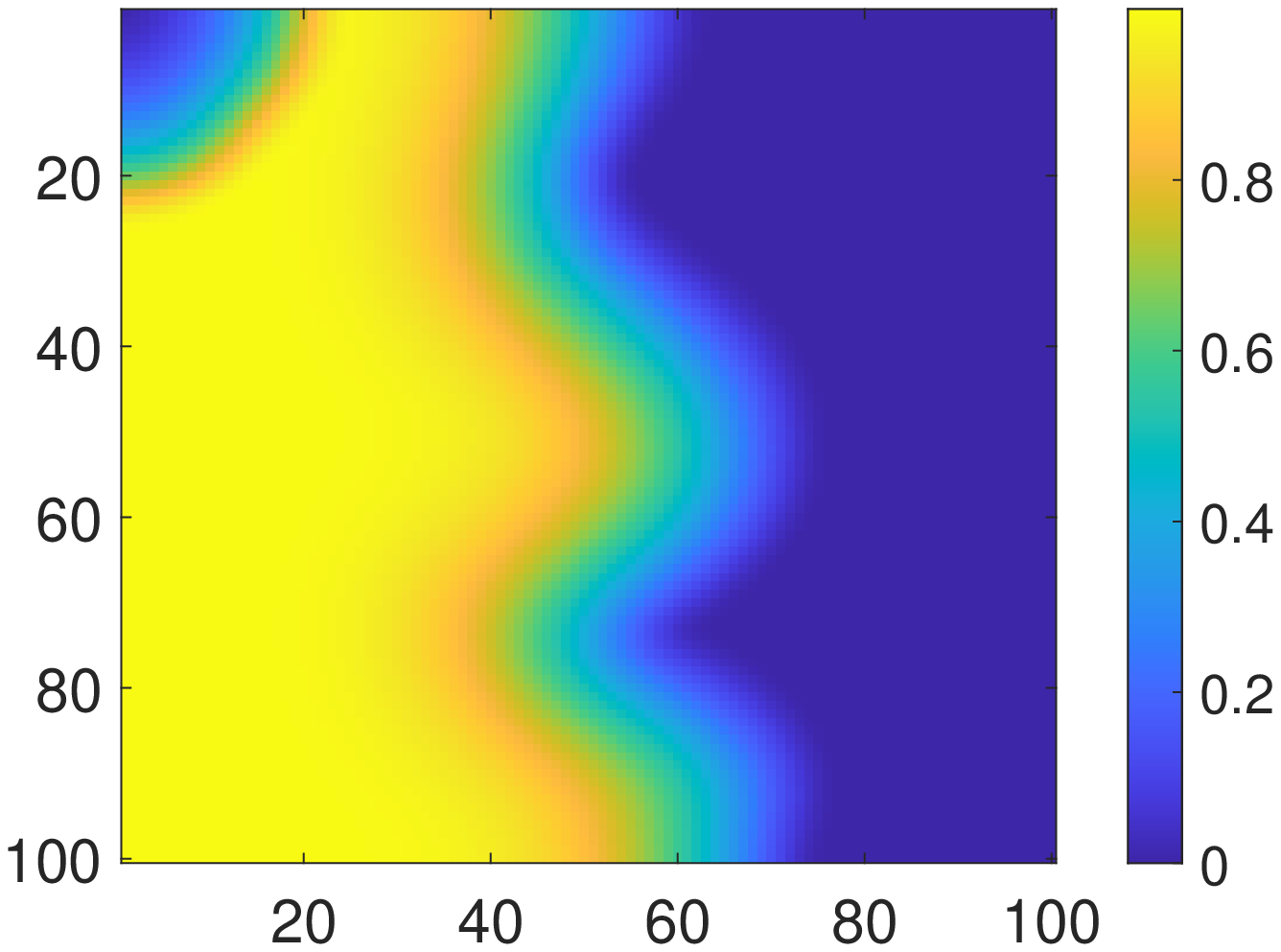}
\end{minipage}
\begin{minipage}[t]{0.2\textwidth}
         \centering
         \includegraphics[width=3cm,height=2cm]{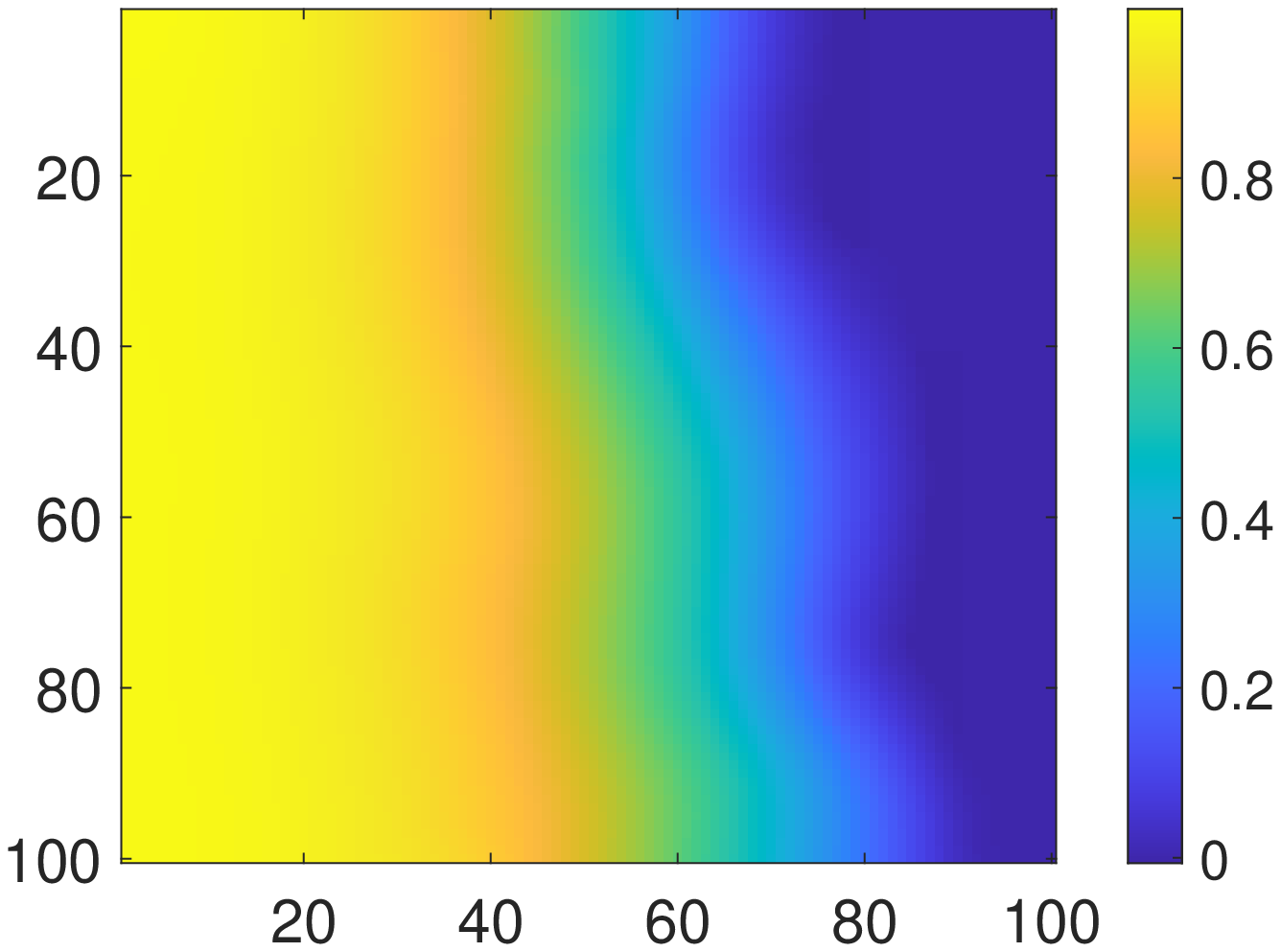}

         \includegraphics[width=3cm,height=2cm]{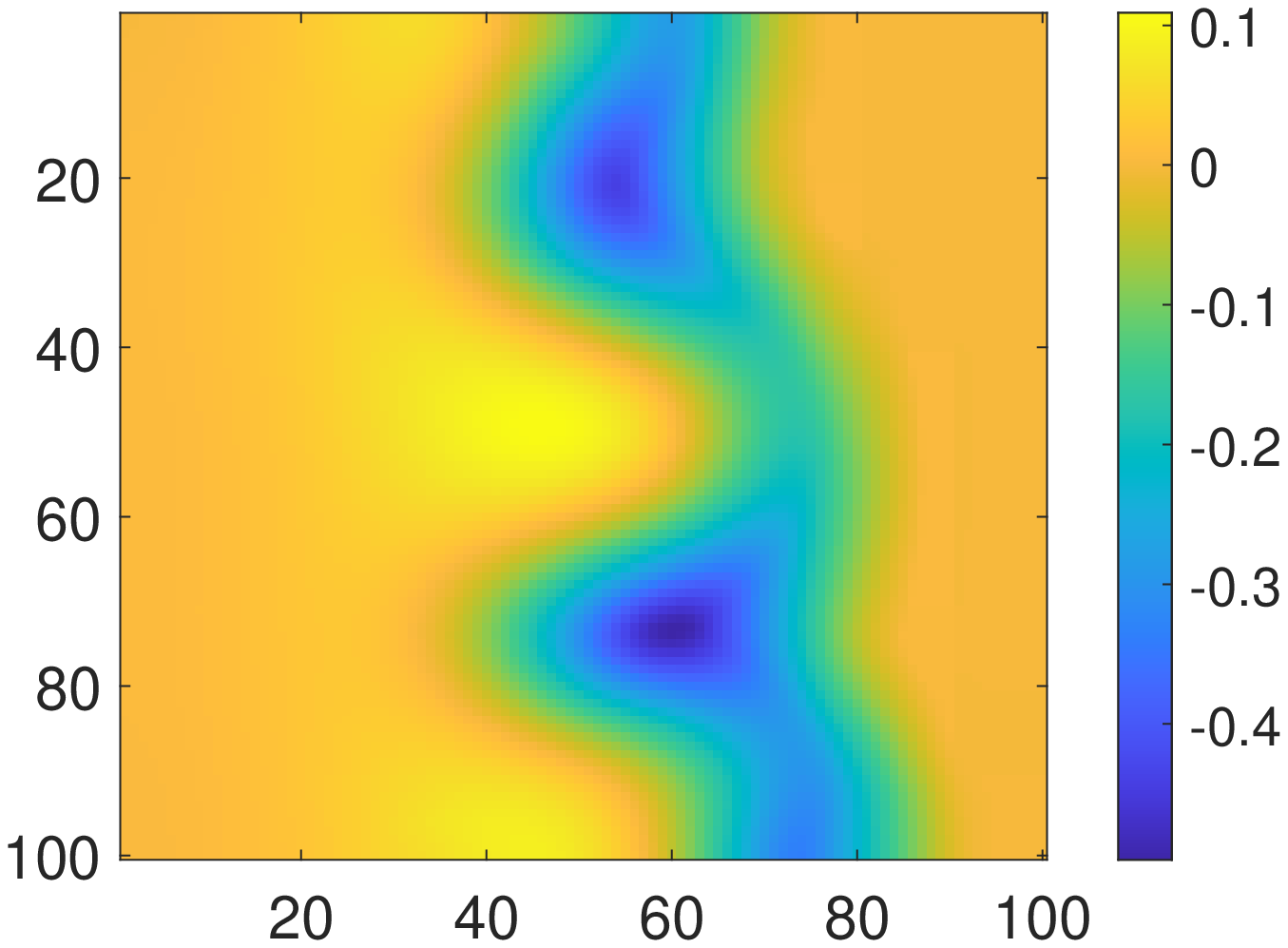}
\end{minipage}
\begin{minipage}[t]{0.2\textwidth}
         \centering
         \includegraphics[width=3cm,height=2cm]{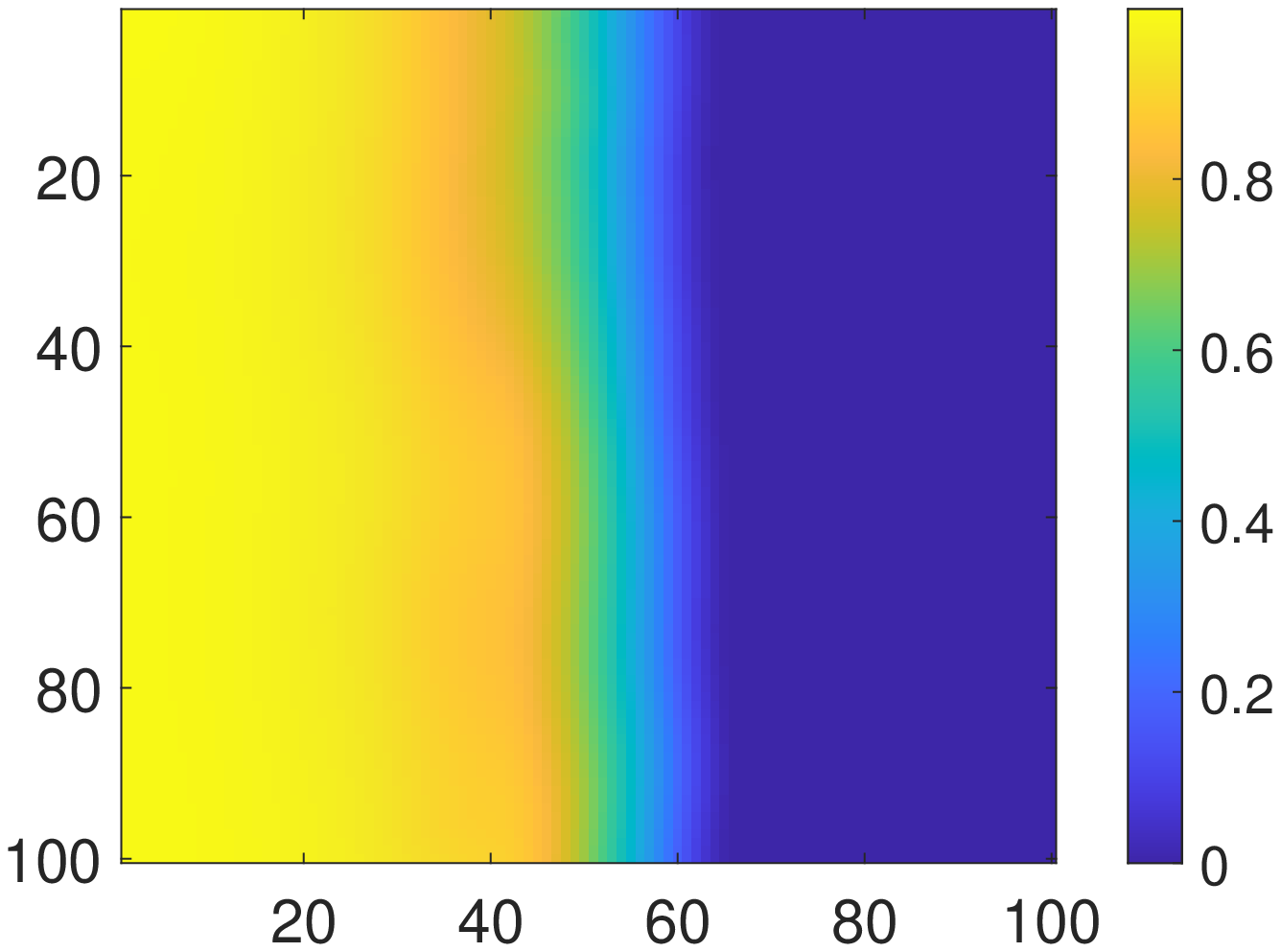}
         
         \includegraphics[width=3cm,height=2cm]{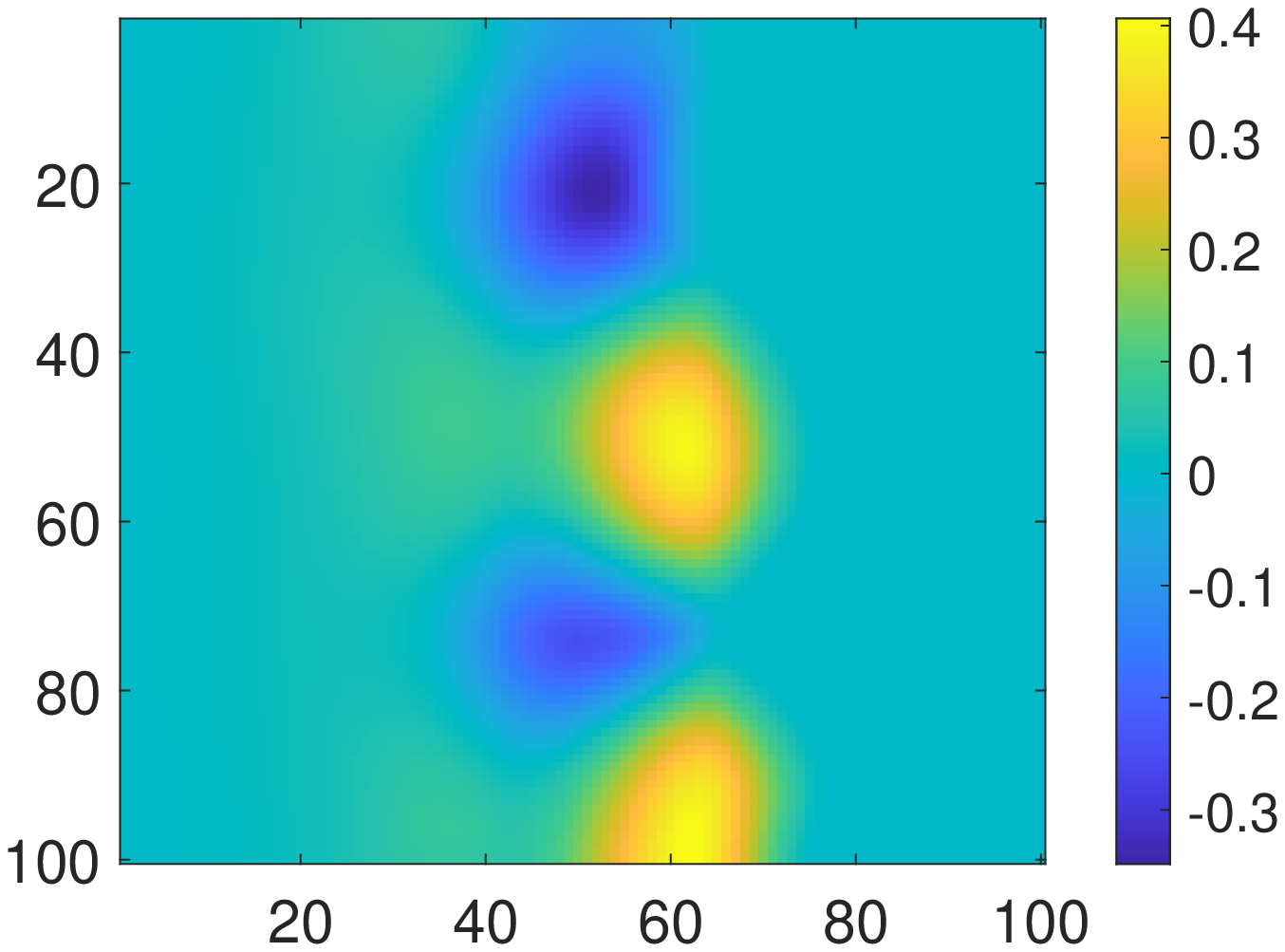}
\end{minipage}
\caption{Snapshots of $S$  and comparison of the solutions with different data at $T=10$. Top: snapshot  of $S$ . Bottom: pointwise error of the solution. Left: two-phase  solution. Middle:
full data. Right: $\frac{1}{2}$ data}
 \label{Sub3_2}
\end{figure}

\begin{figure}[!h]
\begin{minipage}[t]{0.2\textwidth}
         \centering
         \includegraphics[width=3cm,height=2cm]{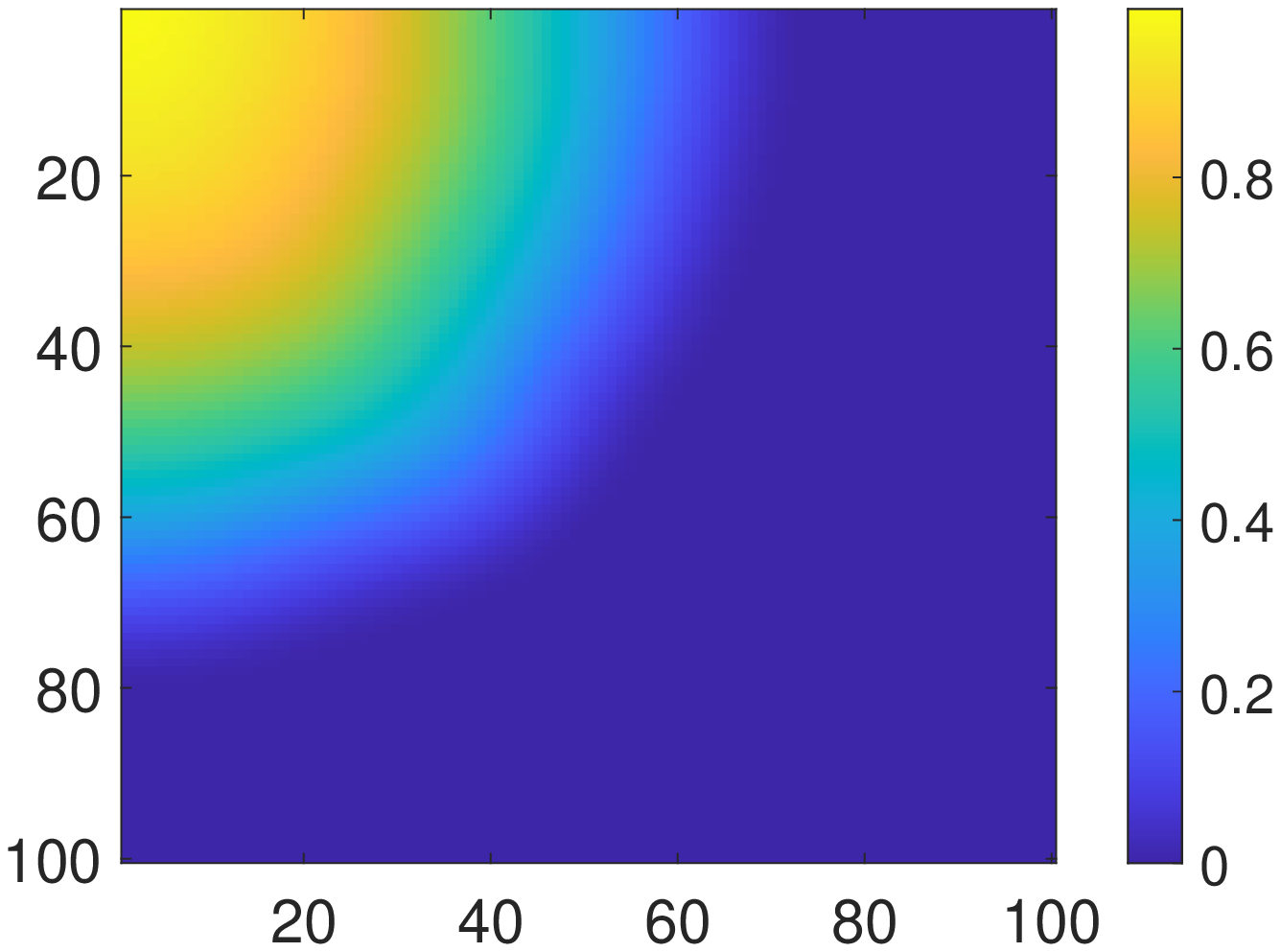}

        \includegraphics[width=3cm,height=2cm]{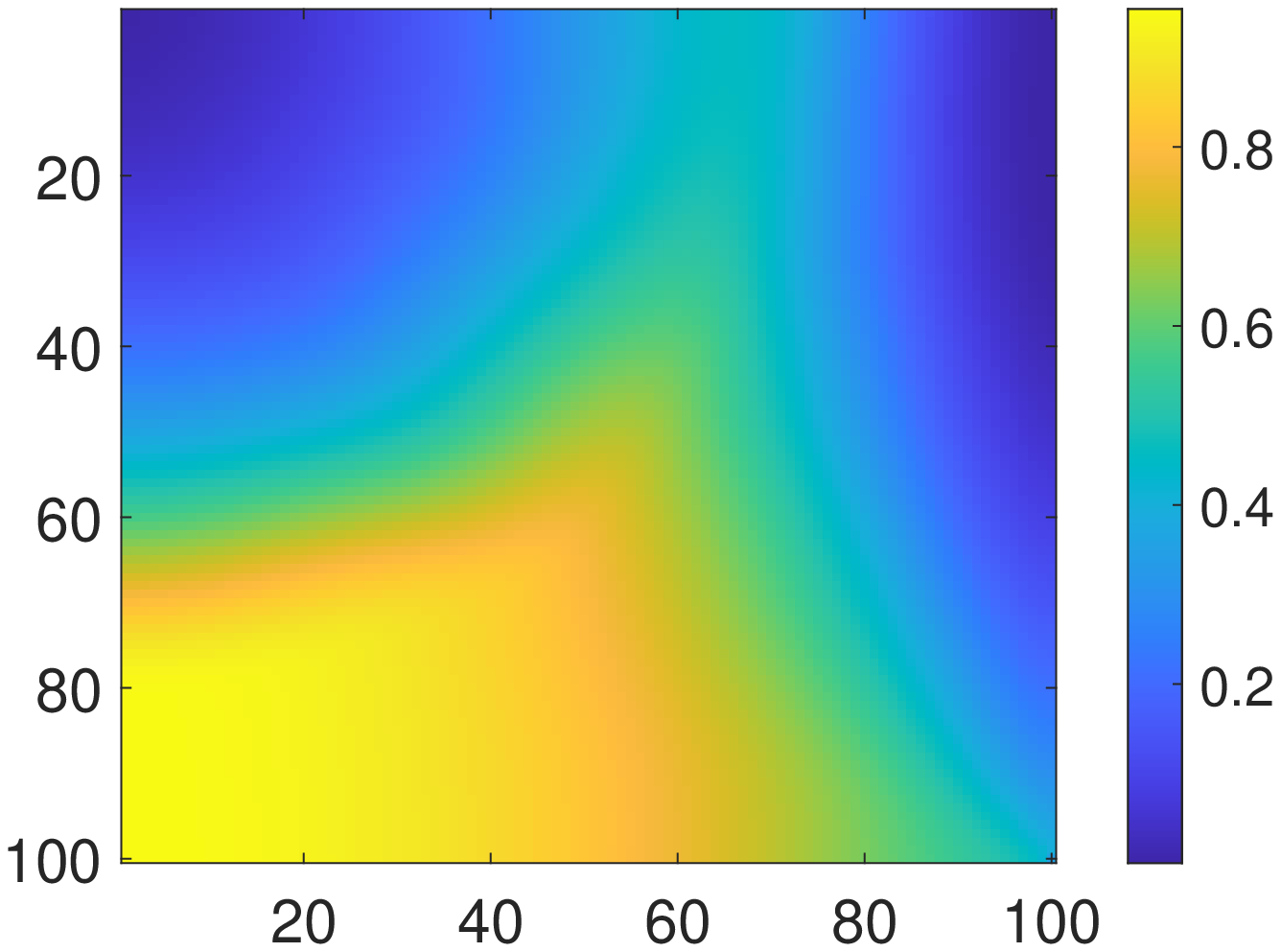}
\end{minipage}
\begin{minipage}[t]{0.2\textwidth}
         \centering
         \includegraphics[width=3cm,height=2cm]{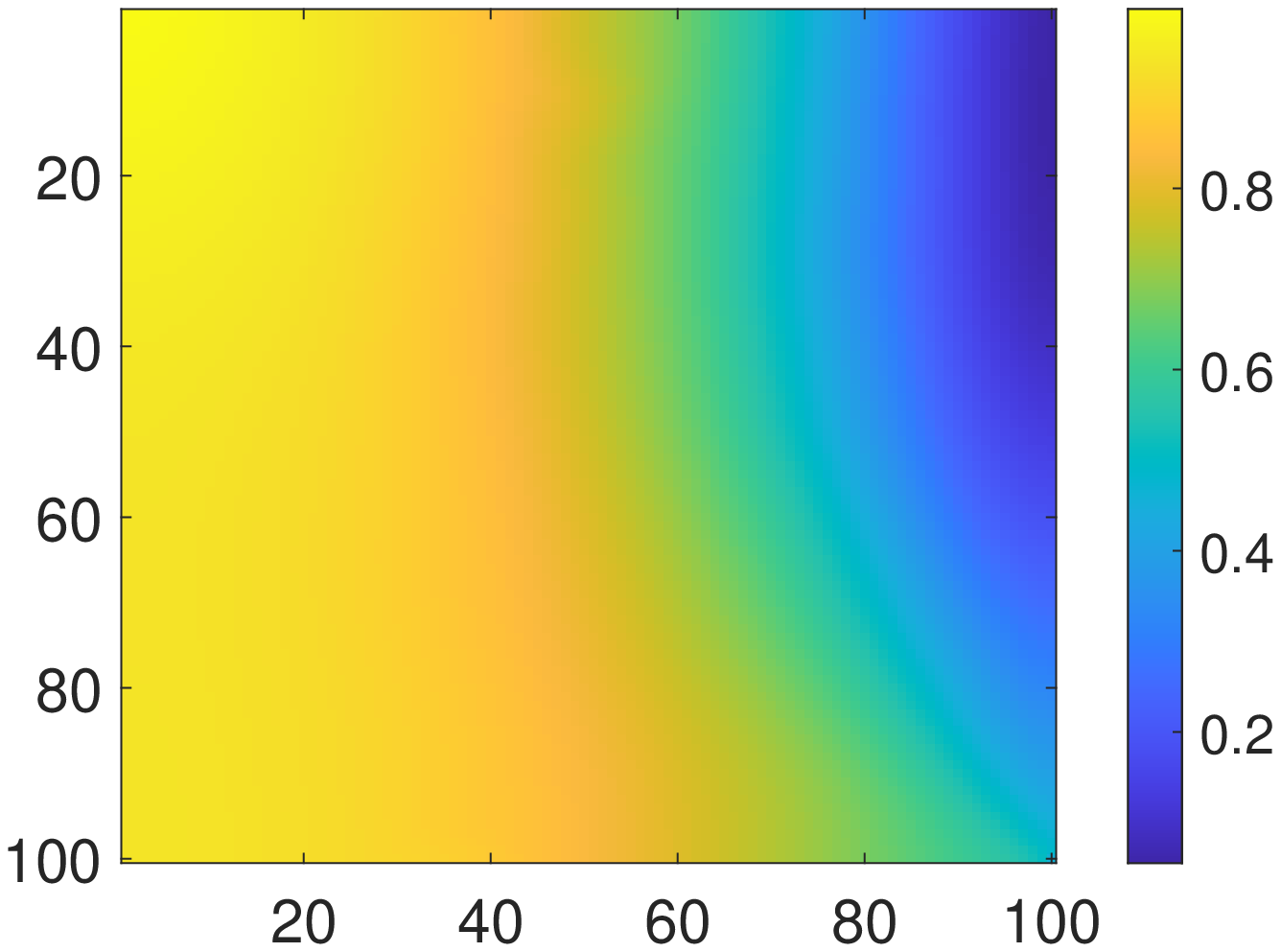}

         \includegraphics[width=3cm,height=2cm]{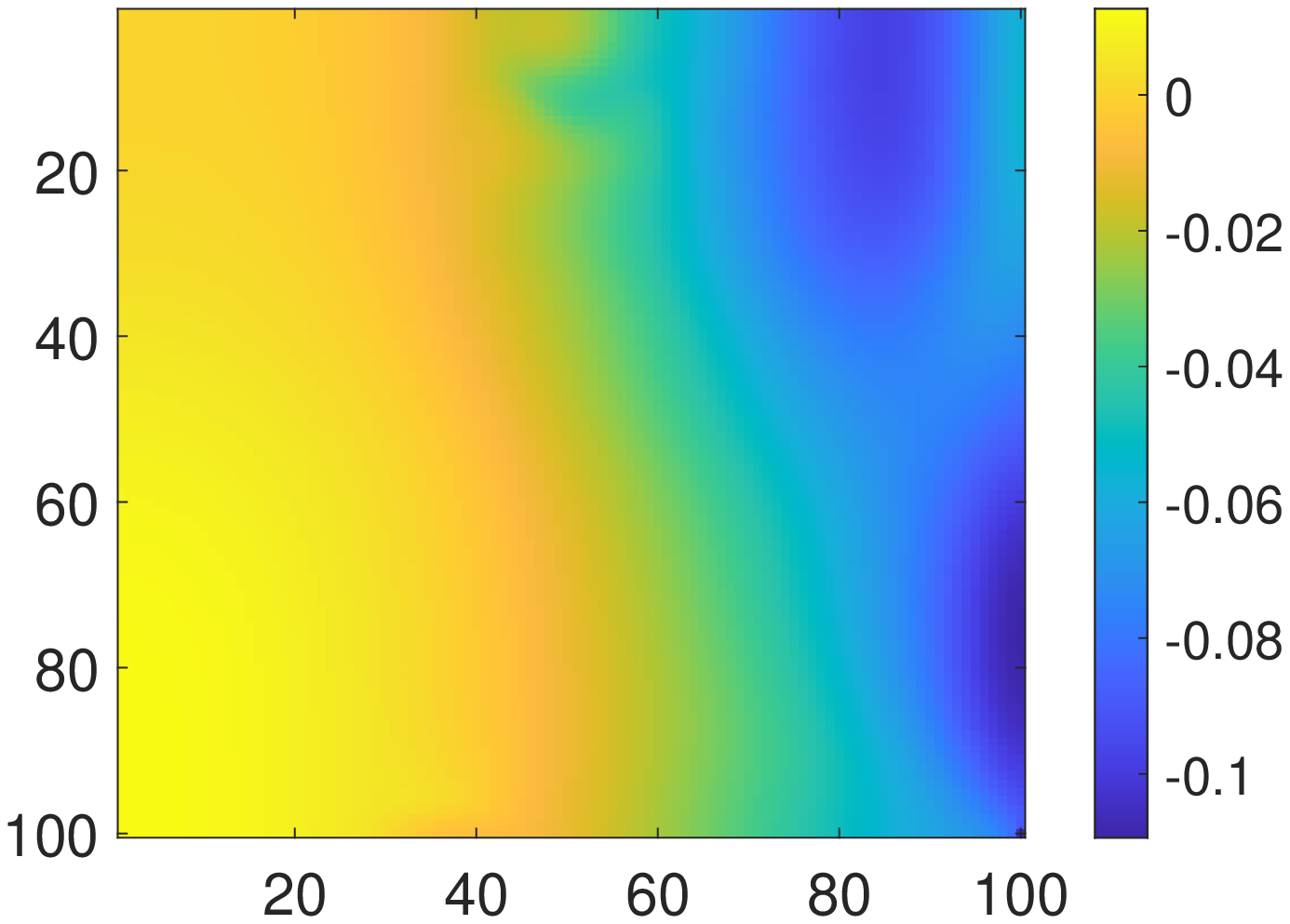}
\end{minipage}
\begin{minipage}[t]{0.2\textwidth}
         \centering
         \includegraphics[width=3cm,height=2cm]{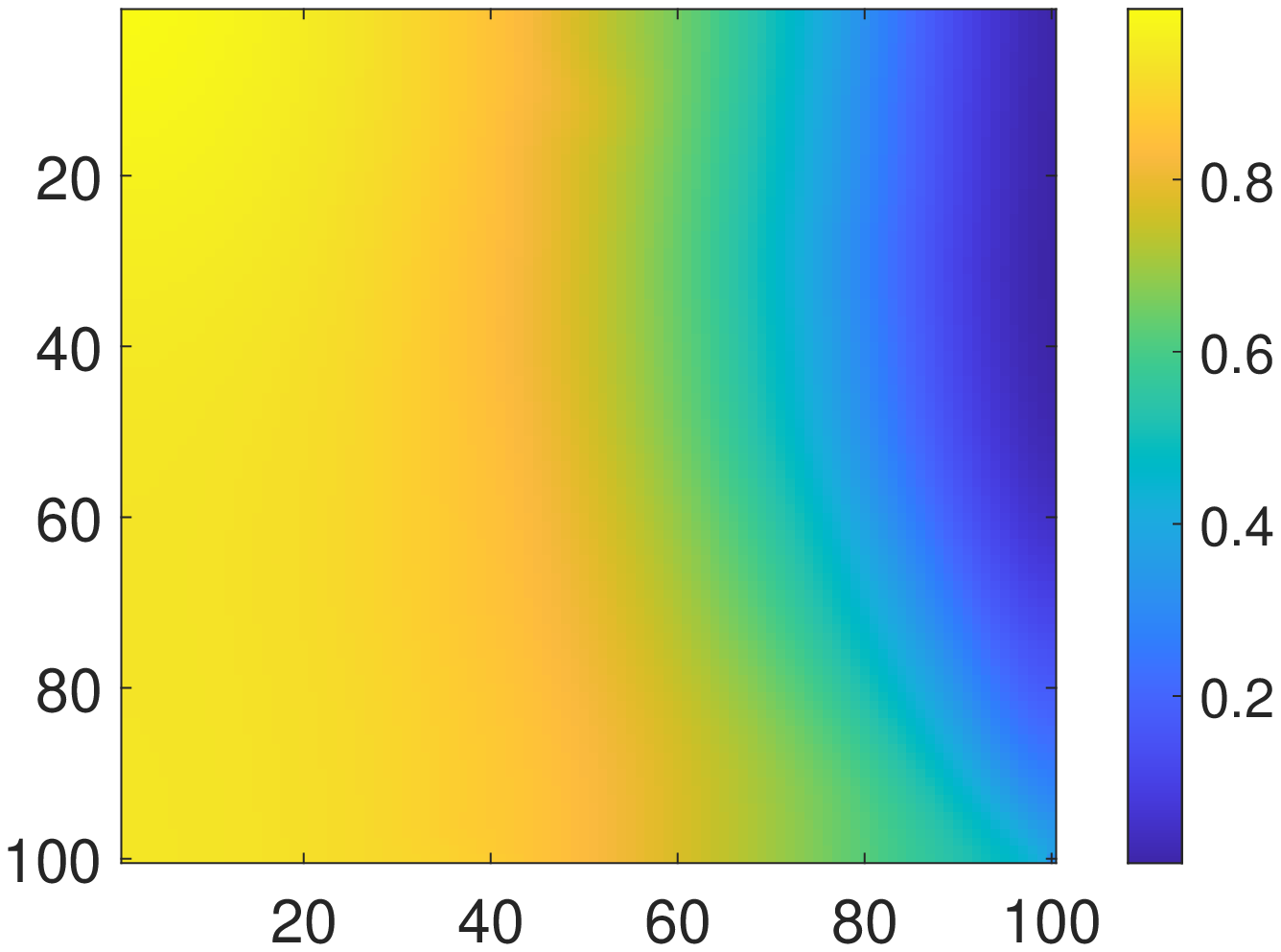}
         
         \includegraphics[width=3cm,height=2cm]{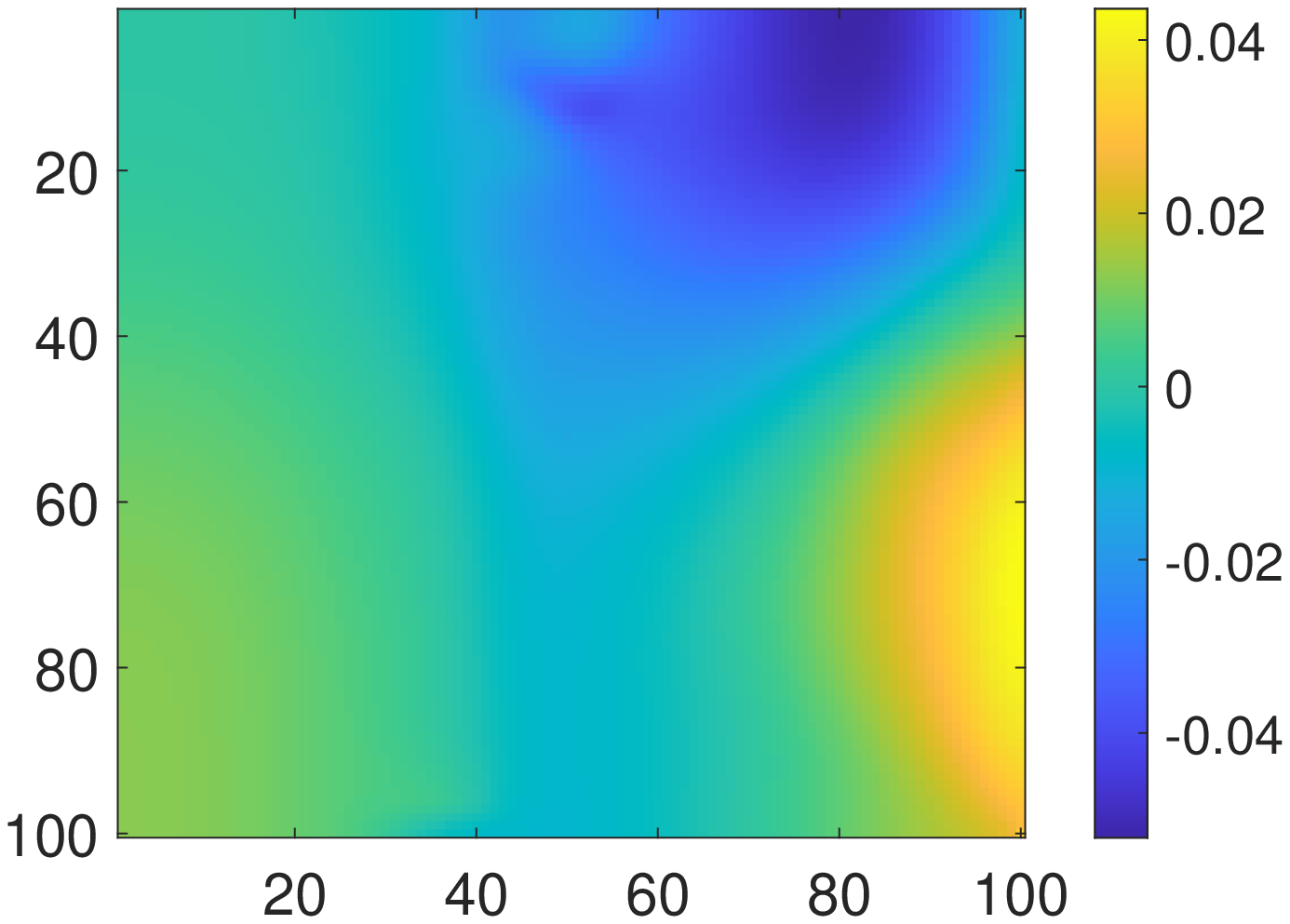}
\end{minipage}
\caption{Snapshots  of $S$ and comparison of the solutions with different data at $T=100$. Top: snapshot  of $S$ . Bottom: pointwise error of the solution. Left: two-phase  solution. Middle:
full data. Right: $\frac{1}{2}$ data}
 \label{Sub4_2}
\end{figure}

\begin{figure}[!b]
\centering

     \begin{subfigure}[b]{0.4\textwidth}
         \centering
         \includegraphics[width=3cm,height=2cm]{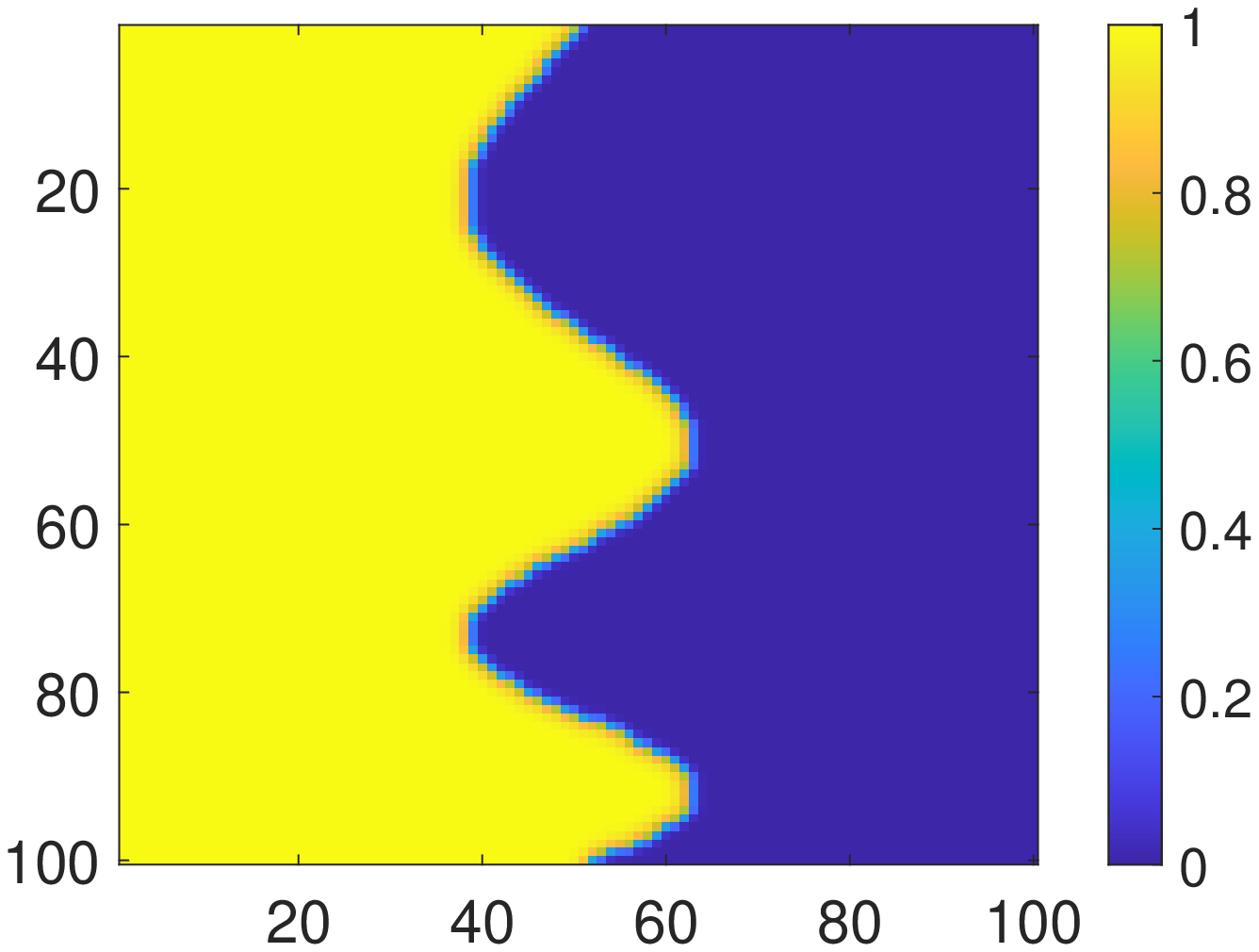}
         
         \includegraphics[width=3cm,height=2cm]{figure/data_sol_2_k2}
         \includegraphics[width=3cm,height=2cm]{figure/data_different_2_k2}
         
         \includegraphics[width=3cm,height=2cm]{figure/twophase_sol_2_k2}
         \includegraphics[width=3cm,height=2cm]{figure/twophase_different_2_k2}
         \caption{$T=0.1$}
     \end{subfigure}
     \begin{subfigure}[b]{0.4\textwidth}
         \centering
         \includegraphics[width=3cm,height=2cm]{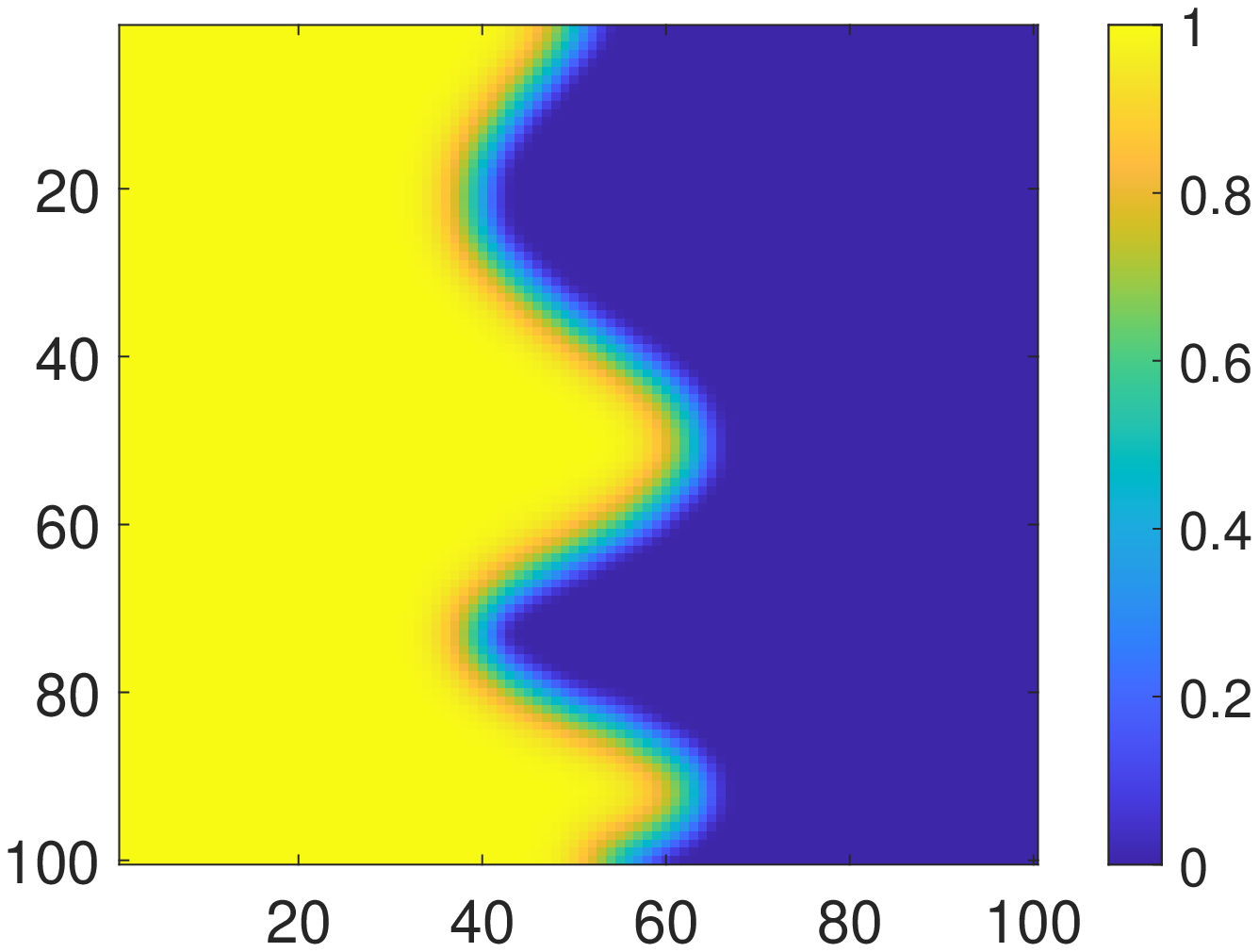}
         
         \includegraphics[width=3cm,height=2cm]{figure/data_sol_2_k20}
         \includegraphics[width=3cm,height=2cm]{figure/data_different_2_k20}
         
         \includegraphics[width=3cm,height=2cm]{figure/twophase_sol_2_k20}
         \includegraphics[width=3cm,height=2cm]{figure/twophase_different_2_k20}
         \caption{$T=1$}
     \end{subfigure}    
     \begin{subfigure}[b]{0.4\textwidth}
         \centering
         \includegraphics[width=3cm,height=2cm]{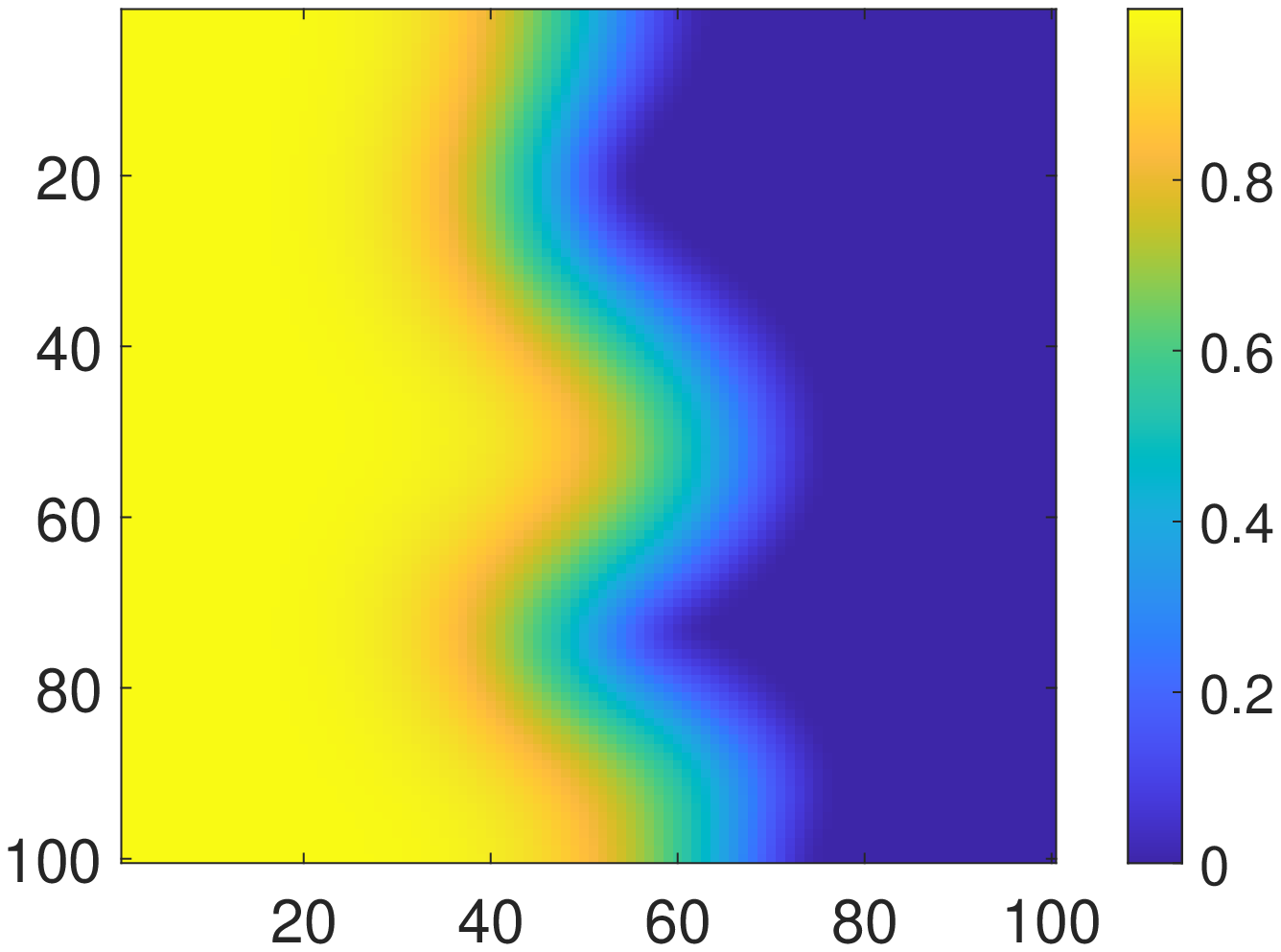}
         
         \includegraphics[width=3cm,height=2cm]{figure/data_sol_2_k200}
         \includegraphics[width=3cm,height=2cm]{figure/data_different_2_k200}
         
         \includegraphics[width=3cm,height=2cm]{figure/twophase_sol_2_k200}
         \includegraphics[width=3cm,height=2cm]{figure/twophase_different_2_k200}
         \caption{$T=10$}
     \end{subfigure}
     \begin{subfigure}[b]{0.4\textwidth}
         \centering
         \includegraphics[width=3cm,height=2cm]{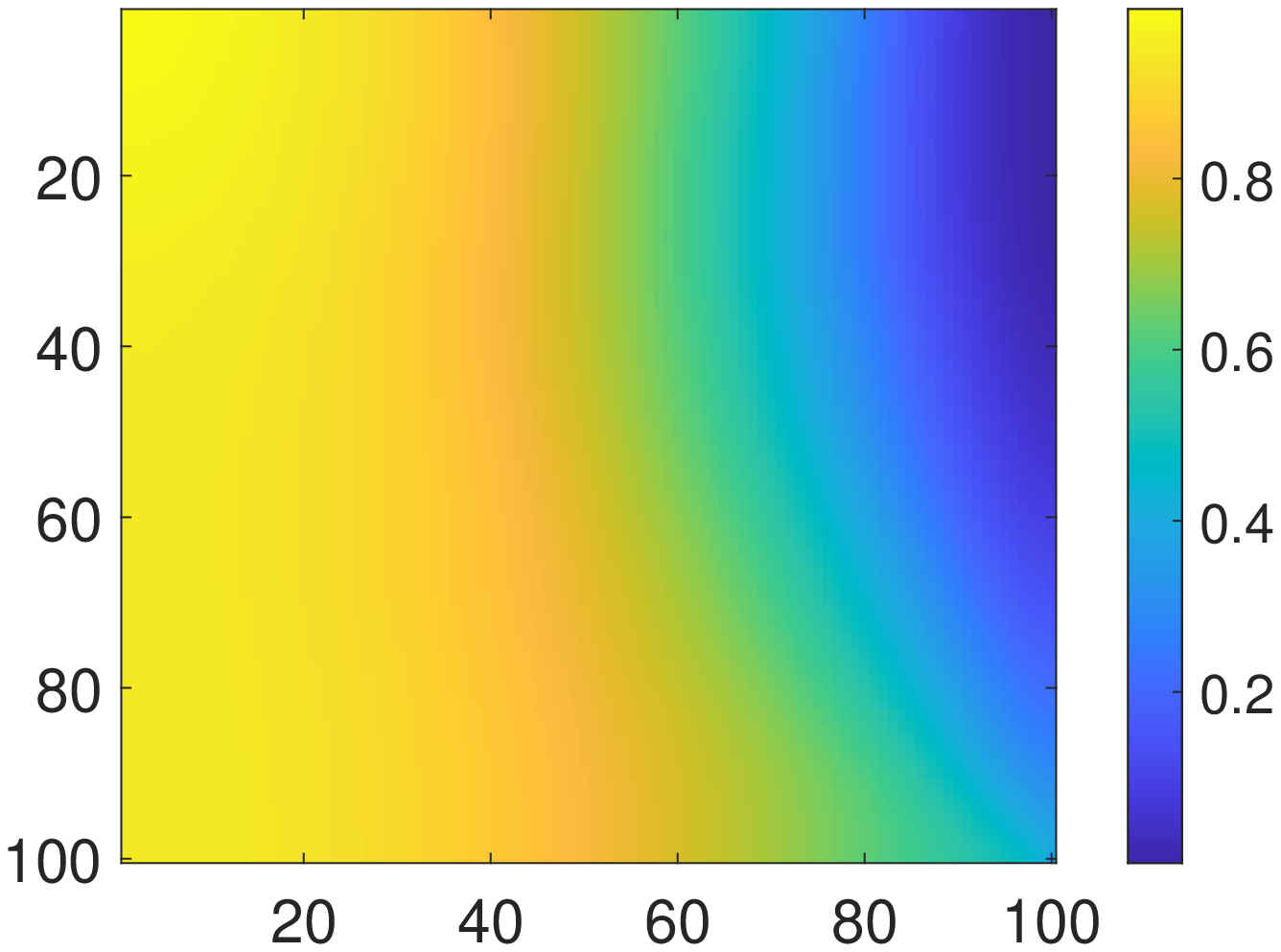}
         
         \includegraphics[width=3cm,height=2cm]{figure/data_sol_2_k2000}
         \includegraphics[width=3cm,height=2cm]{figure/data_different_2_k2000}
         
         \includegraphics[width=3cm,height=2cm]{figure/twophase_sol_2_k2000}
         \includegraphics[width=3cm,height=2cm]{figure/twophase_different_2_k2000}
         \caption{$T=100$}
     \end{subfigure}

\caption{Snapshot of the solutions $S$. Top: Reference solution (exact). Middle-Left:
data assimilation solution. Middle-Right:difference between data assimilation solution and reference solution. Bottom-Left: two-phase  solution. Bottom-Right:
difference between two-phase  solution and reference solution.}
\label{fig:error_case1-1-4}
\end{figure}

\section{Conclusion}
While two-phase models have a long history of success on numerical reservoir simulation, they tend to lose accuracy
on more complicated problems due to the insufficient and inaccurate knowledge of the initial state.
To overcome this difficulty, the observational measurements
can be directly inserted into the mathematical model to improve the accuracy.     In our work, we have proposed, analyzed, and tested a novel continuous data assimilation two-phase flow algorithm for reservoir simulation which combines the coarse grid saturation measurement data with the two-phase flow problem. 
We have shown a stability estimate and an exponentially decaying error bound between the exact and approximate solutions until the error hits a threshold depending on the order of data resolution.  We also looked at two numerical computations and observed that synchronization is achievable at a resolution much more coarse than suggested by the rigorous analysis, which is consistent with experiments carried out at other studies. Moreover,  we demonstrated numerically that machine precision synchronization is achieved for data collected from a small
fraction of the domain. Tests of nudging on moving subdomains are also encouraging, a matter we will explore in a future work.

Future directions include the extension of the method and stability analysis in multi-phase flow and other physically relevant fluid models, e.g. plasma, magneto-fluid, etc, as well as pursuing a stability estimate of our model with partial observation on a small portion of the domain.

\begin{appendices}
In this appendix, we prove lemma \ref{lem:3.7}, lemma \ref{lem:eq1_bound}, lemma \ref{lem:eq2_bound} and theorem \ref{thm:3.11}. First, we will present the proof of lemma \ref{lem:3.7}.
\begin{proof}
Since  $P^{\eta}(t,\cdot)\in V_{0}$, set  $w=P^{\eta}$ in \eqref{DisSol} to  obtain 
\begin{align*}
\underline{\kappa}(t_{i+1}-t_{i})|P^{\eta}(t,\cdot)|_{K}^{2} & \leq\int_{t_{i}}^{t_{i+1}}\int_{\Omega}\kappa(S_{w})K|\nabla P^{\eta}(t,\cdot)|^{2}=\int_{t_{i}}^{t_{i+1}}\int_{\Omega}q_{t}(t,\cdot)P^{\eta}(t,\cdot)\\
 & \leq C(t_{i+1}-t_{i}) \, \, \|q_{t}\|_{L^{\infty}(0,T;H^{-1}(\Omega))}\, \, \|P^{\eta}\|_{L^{\infty}(0,T;H^{1}(\Omega))}.
\end{align*}
Therefore, we have $c_{K}\|P^{\eta}\|_{L^{\infty}(0,T;V_{0})}\leq\|P^{\eta}\|_{L^{\infty}(0,T;H^{1}(\Omega))}\leq C_{K}\|P^{\eta}\|_{L^{\infty}(0,T;V_{0})}$
and
\[
\|P^{\eta}\|_{L^{\infty}(0,T;H^{1})}\leq C_{K}^{2}\underline{\kappa}^{-1}\|q_{t}\|_{L^{\infty}(0,T;H^{-1}(\Omega))}\,.
\]
Then consider $v=\theta^{\eta}$ in  \eqref{DisSol} and obtain 
\[
\int_{0}^{T}\int_{\Omega}\Big(\partial^{\eta}(T(\theta^{\eta}))\theta^{\eta}+\kappa_{w}K\nabla P^{\eta}\cdot\nabla\theta^{\eta}+K|\nabla\theta^{\eta}|^{2}\Big)=\int_{0}^{T}\int_{\Omega}q_{w}\theta^{\eta}\,.
\]
After using Lemma \ref{lem:lem3}, we have 
\[
\int_{0}^{T}\int_{\Omega}\Big(\partial^{\eta}(T(\theta^{\eta}))\theta^{\eta}\Big)\geq-\int_{\Omega}\int_{0}^{\theta_{0}}(T(c)-T(\xi))d\xi\, , 
\]
and 
\[
\int_{\Omega}T(\theta^{\eta})(t,\cdot)-\int_{\Omega}T(\theta^{\eta})(0,\cdot)=\int_{0}^{t}\int_{\Omega}q_{w}\, . 
\]
Hence, we obtain 
\[
|\int_{\Omega}\theta^{\eta}(t,\cdot)|\leq C \, \left(|\int_{\Omega}\theta^{\eta}(0,\cdot)|+ |\int_{0}^{t}\int_{\Omega}q_{w} \, |\right)\, , 
\]
and therefore 
\[\|\theta^{\eta}(t,\cdot)\|_{H^{1}}^{2}\leq C\Big(\|\theta^{\eta}(t,\cdot)\|_{V_{0}}^{2}+\Big|\int_{0}^{t}\int_{\Omega}q_{w}\Big|^{2}+\Big|\int_{\Omega}\theta^{\eta}(0,\cdot)\Big|^{2}\Big)\,.\]
With the above  inequalities, one can  arrive at 
\begin{equation}
  \begin{split}
  |\theta^{\eta}|_{L^{2}(0,T;H^{1})}^{2} & \leq  C \,  |\theta^{\eta}|_{L^{2}(0,T;V_{0})}^{2}+\Big|\int_{0}^{t}\int_{\Omega}q_{w}\Big|^{2}+\Big|\int_{\Omega}\theta^{\eta}(0,\cdot)\Big|^{2}\\
& \leq C \, \big( \|q_{w}\|_{L^{2}(0,T;H^{-1}(\Omega))}^{2}+\|q_{t}\|_{L^{2}(0,T;H^{-1}(\Omega))}^{2}+\int_{\Omega}\int_{0}^{\theta_{0}}(T(\theta_{0})-T(\xi))d\xi\\
 &  \quad  \,  +\Big|\int_{0}^{T}\int_{\Omega}|q_{w}|\Big|^{2}+\Big|\int_{\Omega}\theta^{\eta}(0,\cdot)\Big|^{2} \, \big).
      \end{split}  
\end{equation}
\end{proof}

We will next present the proof of lemma \ref{lem:eq1_bound}. 
\begin{proof}
After subtracting the two  equations, we have 
\[
-\int_{0}^{T}\int_{\Omega}\kappa(S_{1})K\nabla(P_{2}-P_{1})\cdot\nabla w=\int_{0}^{T}\int_{\Omega}(\kappa(S_{2})-\kappa(S_{1}))K\nabla P_{2}\cdot\nabla w+\int_{0}^{T}\int_{\Omega}(q_{t,1}-q_{t,2})w.
\]
for all test function $w\in V_{0}$. Set  $w=P_{1}-P_{2}$ to get
\begin{equation}
\begin{split}
& \int_{0}^{T}\|\nabla(P_{2}-P_{1})\|_{L^{2}}^{2} \leq  C\int_{0}^{T}\int_{\Omega}\kappa(S_{1})K|\nabla(P_{2}-P_{1})|^{2}\\
 & =  \int_{0}^{T}\int_{\Omega}(\kappa(S_{2})-\kappa(S_{1}))K\nabla P_{2}\cdot\nabla(P_{1}-P_{2})+\int_{0}^{T}\int_{\Omega}(q_{t,1}-q_{t,2})(P_{1}-P_{2})\\
 & \leq \,  C\int_{0}^{T}\|\nabla P_{2}\|_{L^{\infty}}\|\kappa(S_{2})-\kappa(S_{1})\|_{L^{2}}\|\nabla(P_{2}-P_{1})\|_{L^{2}}+\int_{0}^{T}\|q_{t,1}-q_{t,2}\|_{L^{2}}\|P_{2}-P_{1}\|_{L^{2}}\,.
    \end{split}
    \end{equation}
And Since $\|P_{2}-P_{1}\|_{L^{2}}\leq C \, \|\nabla(P_{2}-P_{1})\|_{L^{2}}$,
we obtain the result
\[
\|\nabla(P_{2}-P_{1})\|_{L^{2}}\leq C\left(\|\kappa(S_{2})-\kappa(S_{1})\|_{L^{2}}+\|q_{t,2}-q_{t,1}\|_{L^{2}}\right).
\]
\end{proof}

We then present the proof of lemma \ref{lem:eq2_bound}.
\begin{proof}
After subtracting the two equations, we get 
\begin{equation}\label{eq:lemma_10_eq}
\begin{split}
    & \int_{0}^{t}\int_{\Omega}\partial_{t}(S_{2}-S_{1})v+K\nabla(\theta_{2}-\theta_{1})\cdot\nabla v+\kappa_{w}(S_{1})K\nabla(P_{2}-P_{1})\cdot\nabla v \\
 &  = \int_{0}^{t}\int_{\Omega}(\kappa_{w}(S_{1})-\kappa_{w}(S_{2}))K\nabla P_{2}\cdot\nabla v+\int_{0}^{t}(q_{w,1}-q_{w,2},v) \,,  
\end{split}
\end{equation}
for all test functions $v\in L^{2}(0,T;H^{1}(\Omega))$. Using Definition \ref{Greenfun}, we have the following estimates on the first two terms of the above equations
$$
\int_{\Omega}\partial_{t}(S_{2}-S_{1})G(e)  =\int_{\Omega}\partial_{t}(e)G(e)=a\left(G(\partial_{t}(e)\right),G(e) ) =\cfrac{1}{2}\partial_{t}\|G(e)\|_{V_{0}}^{2}=\cfrac{1}{2}\partial_{t}\|e\|_{V_{0}^{*}}^{2}\, ,\\
$$
and 
\begin{align*}
\int_{\Omega}K \, \nabla(\theta_{2}-\theta_{1})\cdot\nabla G(e) & =a(\theta_{2}-\theta_{1},G(e)) =(\theta_{2}-\theta_{1},e)\\
& =(\theta_{2}-\theta_{1},S_{2}-S_{1})-(\theta_{2}-\theta_{1},\pi(S_{2}-S_{1}))\, . 
\end{align*}
By Holder's inequality and Young's inequality, we then obtain
\begin{align*}
\int_{\Omega}K \, \nabla(\theta_{2}-\theta_{1})\cdot\nabla G(e) & \geq(\theta_{2}-\theta_{1},S_{2}-S_{1})-\|\theta_{2}-\theta_{1}\|_{L^{p}}\|\pi(S_{2}-S_{1})\|_{L^{q}}\\
 & \geq(\theta_{2}-\theta_{1},S_{2}-S_{1})-\cfrac{\delta_{p}^{p}}{p}\, \, \|\theta_{2}-\theta_{1}\|_{L^{p}}^{p}-\cfrac{1}{q\delta_{p}^{q}} \, \, \|\pi(S_{2}-S_{1})\|_{L^{q}}^{q}.
\end{align*}
Since $P_{2}\in L^{\infty}(0,T; W^{1,\infty} (\Omega))$,  the term $\int_{0}^{t}\int_{\Omega}(\kappa_{w}(S_{1})-\kappa_{w}(S_{2}))K \, \nabla P_{2}\cdot\nabla G(e)$ in \eqref{eq:lemma_10_eq} can be estimated as 
\begin{align*}
\int_{0}^{t}\int_{\Omega}(\kappa_{w}(S_{1})-\kappa_{w}(S_{2}))K\nabla P_{2}\cdot\nabla G(e) & \leq C\|P_{2}\|_{L^{\infty}(0,T,W^{1,\infty})}\int_{0}^{t}\int_{\Omega}|(\kappa_{w}(S_{1})-\kappa_{w}(S_{2}))|\cdot|\nabla G(e)|\\
 & \leq C\int_{0}^{t}\|\kappa_{w}(S_{1})-\kappa_{w}(S_{2})\|_{L^{2}}\|\nabla G(e)\|_{L^{2}}\\
 & \leq C\int_{0}^{t}\Big(\cfrac{\delta}{2}\|\kappa(S_{2})-\kappa(S_{1})\|_{L^{2}}^{2}+\cfrac{1}{2\delta}\|\nabla G(e)\|_{L^{2}}^{2}\Big).
\end{align*}
Similarly, the other two  terms $\int_{0}^{t}(q_{w,1}-q_{w,2},G(e))$ and $\int_{0}^{t}\int_{\Omega}\kappa_{w}(S_{1})K\nabla(P_{2}-P_{1})\cdot\nabla G(e)$
in \eqref{eq:lemma_10_eq} can be bounded above as 
\[
\int_{0}^{t}(q_{w,1}-q_{w,2},G(e))\leq\int_{0}^{t}\cfrac{\delta}{2}\|q_{w,1}-q_{w,2}\|_{L^{2}}^{2}+\cfrac{1}{2\delta}\int_{0}^{t}\|G(e)\|_{L^{2}}^{2}\, , 
\]
and 
\begin{align*}
\int_{0}^{t}\int_{\Omega}\kappa_{w}(S_{1})K\nabla(P_{2}-P_{1})\cdot\nabla G(e) & \leq C\int_{0}^{t}\left(\cfrac{\delta}{2}\|\nabla(P_{2}-P_{1})\|_{L^{2}}^{2}+\cfrac{1}{2\delta}\|\nabla G(e)\|_{L^{2}}^{2}\right).
\end{align*}
After inserting the above estimates into \eqref{eq:lemma_10_eq}, we obtain
\begin{align*}
 & \cfrac{1}{2}\partial_{t}\int_{0}^{t}\|e\|_{V_{0}^{*}}^{2}+\int_{0}^{t}(\theta_{2}-\theta_{1},S_{2}-S_{1})\\
 & \leq  \cfrac{\delta_{p}^{p}}{p}\int_{0}^{t}\|\theta_{2}-\theta_{1}\|_{L^{p}}^{p}+\cfrac{1}{q\delta_{p}^{q}}\int_{0}^{t}\|\pi(S_{2}-S_{1})\|_{L^{q}}^{q}+C\cfrac{1}{2\delta}\int_{0}^{t}\left(\|\nabla G(e)\|_{L^{2}}^{2}+\|G(e)\|_{L^{2}}^{2}\right)\\
 & + C \, \cfrac{\delta}{2}\int_{0}^{t}\left(\|\kappa(S_{2})-\kappa(S_{1})\|_{L^{2}}^{2}+\|\nabla(P_{2}-P_{1})\|_{L^{2}}^{2}+\|q_{w,1}-q_{w,2}\|_{L^{2}}^{2}\right)\, , 
\end{align*}
for some constant $C>0$. Since $|\theta_{2}-\theta_{1}|\leq C|S_{2}-S_{1}|$
 and  $\|e\|_{L^{2}}\leq C\|\nabla G(e)\|_{L^2}$, we have 
\begin{align*}
 & \cfrac{1}{2}\partial_{t}\int_{0}^{t}\|e\|_{V_{0}^{*}}^{2}+\int_{0}^{t}(\theta_{2}-\theta_{1},S_{2}-S_{1})\\
& \leq  \cfrac{\delta_{p}^{p}}{p}\int_{0}^{t}\|S_{2}-S_{1}\|_{L^{p}}^{p}+\cfrac{1}{q\delta_{p}^{q}}\int_{0}^{t}\|\pi(S_{2}-S_{1})\|_{L^{q}}^{q}+C\cfrac{1}{2\delta}\int_{0}^{t}\|\nabla G(e)\|_{L^{2}}^{2}\\
 & + \, C\cfrac{\delta}{2}\int_{0}^{t}\Big(\|\kappa(S_{2})-\kappa(S_{1})\|_{L^{2}}^{2}+\|\nabla(P_{2}-P_{1})\|_{L^{2}}^{2}+\|q_{w,1}-q_{w,2}\|_{L^{2}}^{2}\Big).
\end{align*}
Then  consider $v=\pi(S_{2}-S_{1})$ in \eqref{eq:lemma_10_eq},  and obtain 
\begin{align*}
\cfrac{1}{q}\Big(\|\pi(S_{2}-S_{1})(t,\cdot)\|_{L^{q}}^{q}-\|\pi(S_{2}-S_{1})(0,\cdot)\|_{L^{q}}^{q}\Big) & =\int_{0}^{t}\int_{\Omega}\partial_{t}(S_{2}-S_{1})(\pi(S_{2}-S_{1}))^{q-1}\\
 & =\int_{0}^{t}(q_{w,2}-q_{w,1},(\pi(S_{2}-S_{1}))^{q-1}).
\end{align*}
Therefore, we have 
$$
\cfrac{1}{q}\|\pi(S_{2}-S_{1})\|_{L^{q}(0,t;L^{q})}^{q}\leq\|\pi(S_{2}-S_{1})\|_{L^{q}(0,t;L^{q})}^{q-1}\|q_{w,2}-q_{w,1}\|_{L^{q}(0,t;L^{q})}+\cfrac{1}{q}\|\pi(S_{2}-S_{1})(0,\cdot)\|_{L^{q}(\Omega)}^{q}\, , 
$$
and 
$$
\|\pi(S_{2}-S_{1})\|_{L^{q}(0,t;L^{q})}^{q}\leq C_{q}\left(\|\pi(S_{2}-S_{1})(0,\cdot)\|_{L^{q} }^{q}+\|q_{w,2}-q_{w,1}\|_{L^{q}(0,t;L^{q})}^{q}\right), 
$$
which proves the lemma. 
\end{proof}

Finally, we will present the proof of theorem \ref{thm:3.11}
\begin{proof}
From Lemma \ref{lem:eq1_bound}, we have 
\begin{align*}
\|\nabla(P_{2}-P_{1})\|_{L^{2}}^{2} & \leq C\Big(\|\kappa(S_{2})-\kappa(S_{1})\|_{L^{2}}^{2}+\|q_{t,2}-q_{t,1}\|_{L^{2}}^{2}\Big)\,.
\end{align*}
Therefore, we observe that 
\begin{align*}
 & \cfrac{1}{2}\partial_{t}\int_{0}^{t}\|e\|_{V_{0}^{*}}^{2}+\int_{0}^{t}(\theta_{2}-\theta_{1},S_{2}-S_{1})\\
\leq & \cfrac{\delta_{p}^{p}}{p}\int_{0}^{t}\|S_{1}-S_{2}\|_{L^{p}}^{p}+\cfrac{C_{q}}{q\delta_{p}^{q}}\Big(\|\pi(S_{2}-S_{1})(0,\cdot)\|_{L^{q}}^{q}+\|q_{w,2}-q_{w,1}\|_{L^{q}(0,t;L^{q})}^{q}\Big)\\
 & +C\cfrac{\delta}{2}\int_{0}^{t}\Big(\|\kappa(S_{2})-\kappa(S_{1})\|_{L^{2}}^{2}+\|q_{w,2}-q_{w,1}\|_{L^{2}}^{2}+\|q_{t,2}-q_{t,1}\|_{L^{2}}^{2}\Big)+C\cfrac{1}{2\delta}\int_{0}^{t}\|\nabla G(e)\|_{L^{2}}^{2}.
\end{align*}
Since $|\kappa(S_{2})-\kappa(S_{1})|\leq(\theta_{2}-\theta_{1})(S_{2}-S_{1})$,
we have
\begin{align*}
 & \cfrac{1}{2}\partial_{t}\int_{0}^{t}\|e\|_{V_{0}^{*}}^{2}+(1-\cfrac{C\delta}{2})\int_{0}^{t}(\theta_{2}-\theta_{1},S_{2}-S_{1})\\
 & \leq  \cfrac{\delta_{p}^{p}}{p}\int_{0}^{t}\|S_{1}-S_{2}\|_{L^{p}}^{p}+\cfrac{C_{q}}{q\delta_{p}^{q}}\Big(\|\pi(S_{2}-S_{1})(0,\cdot)\|_{L^{q}}^{q}+\|q_{w,2}-q_{w,1}\|_{L^{q}(0,t;L^{q})}^{q}\Big)\\
 & +C\cfrac{\delta}{2}\int_{0}^{t}\Big(\|q_{w,2}-q_{w,1}\|_{L^{2}}^{2}+\|q_{t,2}-q_{t,1}\|_{L^{2}}^{2}\Big)+C\cfrac{1}{2\delta}\int_{0}^{t}\|\nabla G(e)\|_{L^{2}}^{2}.
\end{align*}
By  taking $\delta=C^{-1}$ in the above inequality, we obtain
\begin{align*}
 & \partial_{t}\int_{0}^{t}\|e\|_{V_{0}^{*}}^{2}+\int_{0}^{t}(\theta_{2}-\theta_{1},S_{2}-S_{1})\\
 & \leq2\left(  \cfrac{\delta_{p}^{p}}{p}\int_{0}^{t}\|S_{1}-S_{2}\|_{L^{p}}^{p}+\cfrac{C_{q}}{q\delta_{p}^{q}}\Big(\|\pi(S_{2}-S_{1})(0,\cdot)\|_{L^{q}}^{q}+\|q_{w,2}-q_{w,1}\|_{L^{q}(0,t;L^{q})}^{q}\right)\\
 & +\cfrac{1}{2}\int_{0}^{t}\left(\|q_{w,2}-q_{w,1}\|_{L^{2}}^{2}+\|q_{t,2}-q_{t,1}\|_{L^{2}}^{2}\Big)+\cfrac{C^{2}}{2}\int_{0}^{t}\|\nabla G(e)\|_{L^{2}}^{2}\right).
\end{align*}
Since $\int_{0}^{t}(\theta_{2}-\theta_{1},S_{2}-S_{1})\geq C\int_{0}^{t}\int_{\Omega}|S_{2}-S_{1}|^{2+\tau},$
we can choose a $\delta_{p}>0$ such that 
\begin{align*}
 & \partial_{t}\int_{0}^{t}\|e\|_{V_{0}^{*}}^{2}+\cfrac{1}{2}\int_{0}^{T}(\theta_{2}-\theta_{1},S_{2}-S_{1})\\
 & \leq  C \, \left(\|\pi(S_{2}-S_{1})(0,\cdot)\|_{L^{q_{0}}}^{q_{0}}+\|q_{w,2}-q_{w,1}\|_{L^{q_{0}}(0,T;L^{q_{0}})}^{q_{0}}\right)\\
 & + \, C \, \int_{0}^{t} \|\nabla G(e)\|_{L^{2}}^{2}+\cfrac{1}{2}\, \|q_{w,2}-q_{w,1}\|_{L^{2}(0,T;L^{2})}^{2}+ \cfrac{1}{2}\, \|q_{t,2}-q_{t,1}\|_{L^{2}(0,T;L^{2})}^{2}\,.
\end{align*}
Then denote $E=\|\nabla G(e)\|_{L^{2}}^{2}$, and  since $$\|e\|_{V_{0}^{*}}^{2}=\int_{\Omega}K|\nabla G(e)|^{2}\geq c\|\nabla G(e)\|_{L^{2}}^{2}=c\, E\,,$$
and $$\|\nabla G(e)\|_{L^{q_{p}}}^{q_{p}}\leq C\|\nabla G(e)\|_{L^{2}}^{q_{p}}\, , $$
we can get 
\begin{align*}
 & E(t)-E(0)+\cfrac{1}{2}\int_{0}^{t}(\theta_{2}-\theta_{1},S_{2}-S_{1})\\
& \leq  C\int_{0}^{t}E+C \, \left(\|\pi(S_{2}-S_{1})(0,\cdot)\|_{L^{q_{0}}(\Omega)}^{q_{0}}+\|q_{w,2}-q_{w,1}\|_{L^{q_{0}}(0,T;L^{q_{0}} (\Omega))}^{q_{0}} \right)\\
 & +\cfrac{1}{2}\, \left(\|q_{w,2}-q_{w,1}\|_{L^{2}(0,T;L^{2} (\Omega))}^{2}+\|q_{t,2}-q_{t,1}\|_{L^{2}(0,T;L^{2}(\Omega))}^{2}\right) \,.
\end{align*}
With that, we at last arrive at
\begin{align*}
E(t)\leq & \, C\, e^{Ct}\Big(\|\pi(S_{2}-S_{1})(0,\cdot)\|_{L^{q_{0}}}^{q_{0}}+\|q_{w,2}-q_{w,1}\|_{L^{q_{0}}(0,T;L^{q_{0}})}^{q_{0}}+E(0)\\
 & \quad\, \quad  +\|q_{w,2}-q_{w,1}\|_{L^{2}(0,T;L^{2})}^{2}+\|q_{t,2}-q_{t,1}\|_{L^{2}(0,T;L^{2})}^{2}\Big)\, , 
\end{align*}
where $q_{0}=\cfrac{2+\tau}{1+\tau}.$
\end{proof}
\end{appendices}

\end{document}